\documentclass[11pt,final]{siamltex}
\usepackage[leqno]{amsmath}
\usepackage{graphicx}% Figure
\usepackage{subfigure}% Subfigure
 \usepackage{threeparttable}

\usepackage{bm}
\usepackage{color}
\usepackage{amssymb,version}
\usepackage{cases}
\newtheorem{rem}{Remark}[section]
\allowdisplaybreaks

\renewcommand\arraystretch{1.5}

    \title{Global error estimates of high-order fully decoupled schemes for the Cahn-Hilliard-Navier-Stokes model of Two-Phase Incompressible Flows
    \thanks{This work is supported by the National Natural Science Foundation of China  grants  12271302, 11971407 and 12131014, and by the Hong Kong Polytechnic University Postdoctoral Research Fund 1-W22P.}
}
    \author{ Xiaoli Li
        \thanks{School of Mathematics, Shandong University, Jinan, Shandong, 250100, P.R. China. Email: xiaolimath@sdu.edu.cn}.
         \and Nan Zheng
         \thanks{Department of Applied Mathematics, The Hong Kong Polytechnic University, Hung Hom, Kowloon, Hong Kong. Email: znan2017@163.com}.
        \and Jie Shen 
         \thanks{Corresponding Author. Eastern Institute of Technology, Ningbo, China, and Department of Mathematics, Purdue University, West Lafayette,  USA. Email: shen7@purdue.edu}.
                 \and Zhengguang Liu
        \thanks{School of Mathematics and Statistics, Shandong Normal University, Jinan, Shandong, 250358, China. Email: liuzhgsdu@yahoo.com}.
}

\begin{document}
\maketitle

\begin{abstract}
In this paper we construct new fully decoupled and high-order implicit-explicit (IMEX) schemes for the two-phase incompressible flows based on the new generalized scalar auxiliary variable approach with optimal energy approximation (EOP-GSAV) for Cahn-Hilliard equation and consistent splitting method for Navier-Stokes equation. These schemes are linear, fully decoupled, unconditionally energy stable, only require solving a sequence of elliptic equations with constant coefficients at each time step, and provide a new technique to preserve the consistency between original energy and modified energy. We derive that numerical solutions of these schemes are uniformly bounded without any restriction on time step size. Furthermore, we carry out a rigorous error analysis for the first-order scheme and establish optimal global error estimates for the phase function, velocity and pressure in two and three-dimensional cases. Numerical examples are presented to validate the proposed schemes.
\end{abstract}

 

\pagestyle{myheadings}
\thispagestyle{plain}
\markboth{XIAOLI LI, NAN ZHENG, JIE SHEN AND ZHENGGUANG LIU}{FULLY DECOUPLED CS-GSAV SCHEMES}%==================================================================
 \section{Introduction} 
 
In recent years the phase field method has been widely used and become one of the main tools to approximate a variety of interfacial dynamics. The essential idea for the phase field method is that the interface is represented as a thin transition layer between two phases \cite{van1979thermodynamic,chen2016efficient}. The major advantage of the phase field methods is that one can easily derive the governing system from an energy-based variational approach. Hence the existence and uniqueness of solutions (at least local in time) for the complex coupled system can be derived by using the consistent energy dissipation law.

As for the numerical algorithm for Cahn-Hilliard-Navier-Stokes phase-field models, it is crucial that numerical schemes preserve a dissipative  energy law  at the discrete level. Various energy stable numerical methods have been proposed  in the last few decades for Navier-Stokes equations and for Cahn-Hilliard equations.
The main issue in dealing with the Navier-Stokes equation is the coupling of velocity and pressure by the incompressible condition $\nabla\cdot \textbf{u}=0$. 
A partial list of earlier works includes those three categories \cite{GMS06}: the pressure-correction method \cite{shen1992error,weinan1995projection,guermond2004error,li2022new}, the velocity-correction method \cite{guermond2003velocity,serson2016velocity} and the consistent splitting method \cite{guermond2003new,Johnston2004Accurate,shen2007error} (see also the gauge method \cite{E2003Gauge,Nochetto2003Gauge}).  Among these, the consistent splitting scheme has great advantages in  two aspects: (i) this method can achieve full accuracy of the time discretization since it is not limited by splitting error;  (ii)  The inf-sup condition between the velocity and the pressure approximation spaces is no longer enforced from a computational point of view. A main difficulty in solving the Cahn-Hilliard equation is how to deal the nonlinear term efficiently so that the resulting system can be effectively solved while preserving energy dissipation law. There are several popular approaches including the convex splitting  method \cite{eyre1998unconditionally}, stabilized semi-implicit method \cite{shen2010numerical}, invariant energy quadratization (IEQ) \cite{yang2016linear}, and scalar auxiliary  variable (SAV) \cite{shen2019new}.  For an up-to-date review on various classical methods for gradient flows especially for the Cahn-Hilliard equation, one can refer to \cite{Du.F19, MR4132124}.

During the past several decades there have been many studies devoted to developing efficient numerical schemes and carrying out corresponding error analysis for the Cahn-Hilliard-Navier-Stokes phase-field models, which are highly coupled nonlinear systems that rely on delicate cancellations of various nonlinear interactions to establish their dissipation law.
Fully coupled first-order-in-time implicit semi-discrete and fully discrete finite element schemes are considered by Feng, He and Liu \cite{feng2007analysis} and convergence results are established rigorously. Shen and Yang constructed weakly coupled \cite{shen2010phase} and fully decoupled \cite{shen2015decoupled} linear, first-order unconditionally energy stable schemes for two-phase incompressible flows with modified double-well potential. Gr\"un \cite{grun2013convergent}  established an abstract convergence result for a fully discrete implicit scheme for diffuse interface models of two-phase incompressible fluids with different densities.  A coupled second-order energy stable scheme for the  Cahn-Hilliard-Navier-Stokes system based on convex splitting for the Cahn-Hilliard equation is constructed by Han and Wang \cite{MR3324579}. In addition, Han et al. \cite{MR3608328} developed a class of second-order IEQ  schemes which can preserve energy stability. In 2020, we  \cite{Li2019SAV} constructed a second-order weakly-coupled, linear, energy stable SAV-MAC scheme  for the Cahn-Hilliard-Navier-Stokes equations, and established  second order convergence both in time and space for the simpler Cahn-Hilliard-Stokes equations.

 It is important to note that all the aforementioned works involve solving a coupled linear or nonlinear system with variable coefficients at each time step. Recently in \cite{li2022fully}, we construct first- and second-order time discretization schemes for the Cahn-Hilliard- Navier-Stokes system based on the MSAV approach for gradient systems and (rotational) pressure-correction for Navier-Stokes equations. These schemes are linear, fully decoupled, unconditionally energy stable, and only require solving a sequence of elliptic equations with constant coefficients at each time step. In the above work, we only established error estimates for two-dimensional case. However, obtaining results for the three-dimensional case is still crucial and highly challenging. Moreover it is difficult to construct high-order fully decoupled numerical schemes due to the splitting error of the pressure-correction method. 
   
The main purposes of this work are to construct a class of high-order fully decoupled, linear and unconditionally energy stable schemes and  to carry out a rigorous error analysis in  two and three-dimensional cases. By using a combination of techniques in the GSAV approach \cite{huang2022new} with a new energy-optimal  
relaxation \cite{liu2023novel} and consistent splitting method \cite{guermond2003new,li2023error}, and carefully handling the nonlinear decoupled terms, we are finally able to construct a several fully decoupled high-order implicit-explicit
schemes for the two-phase incompressible flows.  
Our main contributions are:
\begin{itemize}
\item We construct new fully decoupled, linear and  high-order schemes for the two-phase incompressible flows, which only requires solving a sequence of  Poisson type equations with constant coefficients at each time step. 
\item The constructed schemes are unconditionally energy stable and provide a new technique to preserve the consistency between original energy and modified energy, which can be proved to be an optimal energy approximation by proposing a novel technique to correct the modified energy of the GSAV approach.
\item Global in time error estimates in  $l^{\infty}(0,T;H^1(\Omega)) \bigcap l^{2}(0,T;H^2(\Omega))$ for the velocity and \newline
 $l^{2}(0,T;H^1(\Omega))$ for the pressure, and  $ l^{\infty}(0,T;H^1(\Omega)) $ for the phase function are established in two and three-dimensional cases. 
\end{itemize}
We believe that our high-order scheme is the first fully decoupled, linear, unconditionally energy stable scheme for the two-phase incompressible flows, and our global error analysis  in the three-dimensional case is the first for any  linear and fully decoupled schemes with explicit treatment of all nonlinear  terms .  

The paper is organized as follows. In Section 2, we give the problem description and preliminaries. In Section 3, we construct the fully decoupled consistent splitting GSAV schemes 
and prove that they are unconditionally energy stable. 
In Section 4, we carry out error estimates for the first-order consistent splitting GSAV scheme for all functions.   In Section 5, we present numerical experiments to verify the accuracy of the theoretical results. 

%==================================================================
\section{The Problem Description and Preliminaries} \label{sec:Notation}

We consider in this paper the construction and analysis of efficient time discretization  schemes for the following Cahn-Hilliard-Navier-Stokes system:
  \begin{subequations}\label{e_model}
    \begin{align}
    \frac{\partial \phi}{\partial t}+ (\textbf{u} \cdot \nabla ) \phi
    =M\Delta \mu  \quad &\ in\ \Omega\times J,
    \label{e_modelA}\\
    \mu=-\lambda\Delta \phi+   \lambda G^{\prime}(\phi) \quad &\ in\ \Omega\times J,
    \label{e_modelB}\\
     \frac{\partial \textbf{u}}{\partial t}+ \textbf{u}\cdot \nabla\textbf{u}
     -\nu\Delta\textbf{u}+\nabla p= \mu \nabla \phi
     \quad &\ in\ \Omega\times J,
      \label{e_modelC}\\
      \nabla\cdot\textbf{u}=0
      \quad &\ in\ \Omega\times J,
      \label{e_modelD}\\
       \frac{\partial \phi}{\partial \textbf{n}}= \frac{\partial \mu}{\partial \textbf{n}}=0,\
     \textbf{u}=\textbf{0}
      \quad &\ on\ \partial\Omega\times J,
      \label{e_modelE}
    \end{align}
  \end{subequations}
where $\displaystyle G(\phi)=\frac{1}{4\epsilon^2}(1-\phi^2)^2$ with $\epsilon$ representing the interfacial width, $M>0$ is the mobility constant, $\lambda>0$ is the mixing coefficient, $\nu>0$ is the fluid viscosity. $\Omega$ is a bounded domain in $\mathbb{R}^d$ and  $J=(0,T]$.  The unknowns are the velocity $\textbf{u}$, the pressure $p$, the phase function $\phi$ and the chemical potential $\mu$. Here we set $ \textbf{u} \cdot \nabla = u_1\frac{\partial }{\partial x}+u_2\frac{\partial }{\partial y}$. 
%It models the dynamics of the mixture of two-incompressible fluids with matching density, which is set to be $\rho_0=1$ for simplicity.
 We refer to \cite{hohenberg1977theory,gurtin1996two,MR1984386} for its physical interpretation and derivation as a phase-field model for the incompressible two phase flow with matching density (set to be $\rho_0=1$ for simplicity), and to \cite{MR2563636} for its mathematical analysis.
The above system satisfies the following energy dissipation law:
\begin{equation}\label{energy1}
 \frac{d E(\phi,\textbf{u})}{d t}= -M\|\nabla \mu\|^2-\nu\|\nabla\textbf{u}\|^2\;\;\text{with}\;\; E(\phi,\textbf{u})=\int_\Omega\{\frac 12|\textbf{u}|^2 +\frac \lambda 2 |\nabla \phi|^2+\lambda G(\phi) \}d\textbf{x}.
\end{equation}
% where $E(\phi,\textbf{u})=\int_\Omega\{\frac 12|\textbf{u}|^2+\frac \gamma 2 \phi^2 +\frac 12 |\nabla \phi|^2+G(\phi) \}d\textbf{x}$ is the total energy.

We next introduce some standard notations. Let $L^m(\Omega)$ be the standard Banach space with norm
$$\| v\|_{L^m(\Omega)}=\left(\int_{\Omega}| v|^md\Omega\right)^{1/m}.$$
For simplicity, let
$$(f,g)=(f,g)_{L^2(\Omega)}=\int_{\Omega}fgd\Omega$$
denote the $L^2(\Omega)$ inner product.
For the case $p=\infty$, set $\|v\|_{\infty}=\|v\|_{L^{\infty}(\Omega)} = ess \ sup \{ 
|f(x)|:  x \in \Omega \}.$
  And $W^{k,p}(\Omega)$ be the standard Sobolev space
$$W^{k,p}(\Omega)=\{g:~\| g\|_{W_p^k(\Omega)}<\infty\},$$
where
\begin{equation}\label{enorm1}
\| g\|_{W^{k,p}(\Omega)}=\left(\sum\limits_{|\alpha|\leq k}\| D^\alpha g\|_{L^p(\Omega)}^p \right)^{1/p},
\end{equation}
in the case $1 \leq p < \infty$, and in the case $p=\infty$,
$$\| g\|_{W^{k,\infty}(\Omega)}= \max_{|\alpha|\leq k}\| D^\alpha g\|_{L^{\infty}(\Omega)}.$$
For simplicity, we set $H^k(\Omega)=W^{k,2}(\Omega)$ and $\| f \|_{k}= \| f \|_{H^k(\Omega)}$.

By using  Poincar\'e inequality, we have
\begin{equation}\label{e_norm H1}
\aligned
\|\textbf{v}\|\leq c_1\|\nabla\textbf{v}\|, \ \forall \  \textbf{v} \in \textbf{H}^1_0(\Omega),
\endaligned
\end{equation}
where $c_1$ is a positive constant depending only on $\Omega$ and
$$ \textbf{H}^1_0(\Omega) = \{ \textbf{v}\in \textbf{H}^1(\Omega):  \textbf{v}|_{\Gamma}=0 \} . $$ 
Define   
  $$\textbf{H}=\{ \textbf{v}\in \textbf{L}^2(\Omega): div\textbf{v}=0, \textbf{v}\cdot \textbf{n}|_{\Gamma}=0 \},\ \ \textbf{V}=\{\textbf{v}\in \textbf{H}^1_0(\Omega):  div\textbf{v}=0 \},$$
and the trilinear form $b(\cdot,\cdot,\cdot)$ by
\begin{equation*}
\aligned
b(\textbf{u},\textbf{v},\textbf{w})=\int_{\Omega}(\textbf{u}\cdot\nabla)\textbf{v}\cdot \textbf{w}d\textbf{x}.
\endaligned
\end{equation*}
We can easily observe that the trilinear form $b(\cdot,\cdot,\cdot)$ is skew-symmetric with respect to its last two arguments, i.e., 
\begin{equation}\label{e_skew-symmetric1}
\aligned
b(\textbf{u},\textbf{v},\textbf{w})=-b(\textbf{u},\textbf{w},\textbf{v}),\ \ \forall \ \textbf{u}\in \textbf{H}, \ \ \textbf{v}, \textbf{w}\in \textbf{H}^1_0(\Omega),
\endaligned
\end{equation}
and 
\begin{equation}\label{e_skew-symmetric2}
\aligned
b(\textbf{u},\textbf{v},\textbf{v})=0,\ \ \forall \ \textbf{u}\in \textbf{H}, \ \ \textbf{v}\in \textbf{H}^1_0(\Omega).
\endaligned
\end{equation}
By applying a combination of integration by parts,  Holder's inequality,  and Sobolev inequalities \cite{Tema95,shen1992error}, we have that for $d\le 4$,

\begin{flalign}\label{e_estimate for trilinear form}
b(\textbf{u},\textbf{v},\textbf{w})\leq \left\{
   \begin{array}{l}
   c_2\|\textbf{u}\|_1\|\textbf{v}\|_1\|\textbf{w}\|_1,\ \ \forall  \ \textbf{u}, \textbf{v} \in \textbf{H}
   , \textbf{w}\in \textbf{H}^1_0(\Omega),\\
   c_2\|\textbf{u}\|_2\|\textbf{v}\|\|\textbf{w}\|_1, \ \ \forall \ \textbf{u}\in \textbf{H}^2(\Omega)\cap\textbf{H},\ \textbf{v} \in \textbf{H}, \textbf{w}\in \textbf{H}^1_0(\Omega),\\
   c_2\|\textbf{u}\|_2\|\textbf{v}\|_1\|\textbf{w}\|, \ \ \forall \ \textbf{u}\in \textbf{H}^2(\Omega)\cap\textbf{H},\ \textbf{v} \in \textbf{H}, \textbf{w}\in \textbf{H}^1_0(\Omega),\\
   c_2\|\textbf{u}\|_1\|\textbf{v}\|_2\|\textbf{w}\|, \ \ \forall \ \textbf{v}\in \textbf{H}^2(\Omega)\cap\textbf{H},\ \textbf{u}\in \textbf{H}, \textbf{w}\in \textbf{H}^1_0(\Omega),\\
   c_2\|\textbf{u}\|\|\textbf{v}\|_2\|\textbf{w}\|_1, \ \ \forall \ \textbf{v}\in \textbf{H}^2(\Omega)\cap\textbf{H},\ \textbf{u} \in \textbf{H}, \textbf{w} \in \textbf{H}^1_0(\Omega).
   \end{array}
   \right.
\end{flalign}
In addition, we have the following more precise inequalities \cite{temam1983nonlinear,huang2021stability}:
\begin{flalign}\label{e_estimate for trilinear form2}
b(\textbf{u},\textbf{v},\textbf{w})\leq \left\{
   \begin{array}{l}
    c_2\|\textbf{u}\|_1^{1/2}\|\textbf{u}\|^{1/2}\|\textbf{v}\|_1^{1/2}\|\textbf{v}\|_2^{1/2}\|\textbf{w}\|, \ \ \  d\le 2, \\
   c_2\|\textbf{u}\|_1 \| \textbf{v}\|^{1/2}_{1} \| \textbf{v}\|^{1/2}_{2} \|\textbf{w}\|, \ \ \  d\le 3,
   \end{array}
   \right.
\end{flalign}
where $c_2$ is a positive constant depending only on $\Omega$.

 The following lemmas will be frequently used in the sequel, one can refer the proof and more detailed information in  
\cite{brenner2008mathematical,an2021optimal}:

 \begin{lemma}(Holder inequality) \label{lem: Holder inequality}
 Suppose that $\textbf{u} \in \textbf{L}^p(\Omega)$, $\textbf{v}\in \textbf{L}^q(\Omega)$, $\textbf{w}\in \textbf{L}^s(\Omega)$, $\frac{1}{p}+\frac{1}{q}+\frac{1}{s}=1$, then we have
 \begin{equation}\label{e_Preliminaries1}
\aligned
\int_{\Omega} | \textbf{u} \textbf{v} \textbf{w}| d \textbf{x} \leq \| \textbf{u} \|_{\textbf{L}^p} \| \textbf{v} \|_{\textbf{L}^q} \| \textbf{w} \|_{\textbf{L}^s}.
\endaligned
\end{equation}
\end{lemma}
 
  \begin{lemma}(Interpolation inequalities)  \label{lem: Interpolation inequalities}
 For $k=3,4,6$, we have
 \begin{equation}\label{e_Preliminaries2}
\aligned
\| \textbf{f} \|_{\textbf{L}^k} \leq C \| \textbf{f} \|_{\textbf{L}^2}^{ \frac{6-k}{ 2k } } \| \textbf{f} \|_{\textbf{H}^1}^{ \frac{3k-6}{ 2k } },
\endaligned
\end{equation}
 \begin{equation}\label{e_Preliminaries3}
\aligned
\| \textbf{f} \|_{\textbf{L}^{\infty}} \leq C \| \textbf{f} \|_{\textbf{H}^1}^{ \frac{1}{ 2 } } \| \textbf{f} \|_{\textbf{H}^2}^{ \frac{1}{ 2 } }.
\endaligned
\end{equation}
\end{lemma}

We will frequently use the following discrete version of the Gr\"onwall lemma \cite{shen1990long,HeSu07}:

\medskip
\begin{lemma} \label{lem: gronwall1}
Let $a_k$, $b_k$, $c_k$, $d_k$, $\gamma_k$, $\Delta t_k$ be nonnegative real numbers such that
\begin{equation}\label{e_Gronwall3}
\aligned
a_{k+1}-a_k+b_{k+1}\Delta t_{k+1}+c_{k+1}\Delta t_{k+1}-c_k\Delta t_k\leq a_kd_k\Delta t_k+\gamma_{k+1}\Delta t_{k+1}
\endaligned
\end{equation}
for all $0\leq k\leq m$. Then
 \begin{equation}\label{e_Gronwall4}
\aligned
a_{m+1}+\sum_{k=0}^{m+1}b_k\Delta t_k \leq \exp \left(\sum_{k=0}^md_k\Delta t_k \right)\{a_0+(b_0+c_0)\Delta t_0+\sum_{k=1}^{m+1}\gamma_k\Delta t_k \}.
\endaligned
\end{equation}
\end{lemma}

Throughout the paper we use $C$, with or without subscript, to denote a positive
constant, independent of discretization parameters, which could have different values at different places.

%==================================================================
  \section{The CS-GSAV schemes}
 In this section, we first reformulate the Cahn-Hilliad-Navier-Stokes system into an equivalent 
  system with generalized scalar auxiliary variables (GSAV). Then,  
   we construct high-order fully decoupled semi-discrete consistent splitting GSAV schemes based on the IMEX BDF-$k$ formulae with $k=1,2,3,4,5$, and prove that they are unconditionally energy stable.
   %==================================================================
 \subsection{GSAV reformulation} 
 Let $\gamma>0$ be a  positive constant,  $F(\phi)=G(\phi)-\frac\gamma 2  \phi^2$ and  
\begin{equation}\label{energy_reformulation}
E(\phi,\textbf{u})=\int_\Omega\{\frac 12|\textbf{u}|^2 +\frac \lambda 2 |\nabla \phi|^2 + \frac { \lambda \gamma } {2}  \phi ^2 +\lambda F(\phi) \}d\textbf{x},
\end{equation} 
where the term  $\frac\gamma 2  \phi^2$ is introduced to simplify the analysis 
  (cf. \cite{shen2018convergence}).
We introduce the following generalized scalar auxiliary variable 
  \begin{subequations}\label{e_definition_SAV}
    \begin{align}
    &r(t) = E(\phi,\textbf{u}) + \kappa_0,
    \end{align}
  \end{subequations}
 where $\kappa_0$ is a positive constant to guarantee that $ \lambda \int_{\Omega} F(\phi) d\textbf{x} + 
 \kappa_0 >0 $. Next we reformulate the  system (\ref{e_model})  as:
 \begin{subequations}\label{e_model_r}
    \begin{align}
    \frac{\partial \phi}{\partial t} + (\textbf{u} \cdot \nabla ) \phi
    =M\Delta \mu  \quad &\ in\ \Omega\times J,
    \label{e_model_rA}\\
    \mu=- \lambda \Delta \phi + \lambda \gamma \phi + \lambda  F^{\prime}(\phi) \quad &\ in\ \Omega\times J,
    \label{e_model_rB}\\
    \frac{\partial \textbf{u}}{\partial t}+  \textbf{u}\cdot \nabla\textbf{u}-\nu\Delta\textbf{u}+\nabla p=   \mu \nabla \phi
     \quad &\ in\ \Omega\times J,
      \label{e_model_rC}\\
            \nabla\cdot\textbf{u}=0
      \quad &\ in\ \Omega\times J,
      \label{e_model_rD}\\
      \frac{\rm{d} r}{\rm{d} t}   = -M \|\nabla \mu\|^2 - \nu \|\nabla \textbf{u} \|^2 \quad &\ in\ \Omega\times J .
    \label{e_model_r_Modify}  
 \end{align}
  \end{subequations}
 It is easy to see that the above system is equivalent to the original system. We shall construct below efficient numerical schemes for the above system which are energy stable with respect to  \eqref{e_model_r_Modify}.

 % ****************************************************************************************** 
 Assuming $ \tilde{ \phi }^j $, $ \tilde{ \mu }^j $ and $\tilde{\textbf{u}}^{j}$ with $j=n,n-1,\ldots, n-l+1$ are given, we solve $ \tilde{ \phi }^{n+1} $, $ \tilde{ \mu }^{n+1} $ and $\tilde{\textbf{u}}^{n+1}$ from 
     \begin{eqnarray}
 &&  \frac{ \alpha_k \tilde{ \phi} ^{n+1} - A_k( \tilde{ \phi }^{n} )  } {  \Delta t } +   ( B_k( \textbf{u}^n ) \cdot \nabla ) B_k( \phi^n )
   = M \Delta \tilde{ \mu }^{n+1}, \label{e_model_semi1} \\
 &&  \tilde{ \mu }^{n+1} =- \lambda \Delta \tilde{ \phi }^{n+1}+ \lambda \gamma \tilde{ \phi }^{n+1} + \lambda  F^{\prime}( B_k( \phi^n) ),\label{e_model_semi2} \\   
 &&  \frac{ \alpha_k \tilde{\textbf{u}}^{n+1} - A_k( \tilde{\textbf{u}}^{n} )  } {  \Delta t }  - \nu \Delta \tilde{\textbf{u}}^{n+1} = 
B_k( \mu^n ) \nabla B_k( \phi^n ) - ( B_k( \textbf{u}^{n} ) \cdot \nabla ) B_k( \textbf{u}^{n} )  - \nabla B_k( p^{n} ), \label{e_model_semi3}  
 \end{eqnarray}
 where $\alpha_k$, the operators $A_k$ and $B_k$ are defined by

first-order scheme: 
\begin{equation*}\label{e_First-order_operator}
\aligned
\alpha_1=1, \ A_1( f^{n} ) =  f^{n} , 
\ B_1( g^{n}) =  g^{n} ;
\endaligned
\end{equation*}

second-order scheme: 
\begin{equation*}\label{e_Second-order_operator}
\aligned
\alpha_2=\frac{3}{2}, \ A_2( f^{n} ) = 2 f^{n} - \frac{1}{2} f^{n-1}, 
\ B_2( g^{n}) = 2 g^{n} -  g^{n-1};
\endaligned
\end{equation*} 

third-order scheme: 
\begin{equation*}\label{e_Third-order_operator}
\aligned
\alpha_3=\frac{11}{6}, \ A_3( f^{n} ) = 3 f^{n} - \frac{3}{2} f^{n-1} +  \frac{1}{3} f^{n-2},
\ B_3( g^{n} ) = 3 g^{n} -  3 g^{n-1} + g^{n-2} ;
\endaligned
\end{equation*} 

fourth-order scheme: 
\begin{equation*}\label{e_Fourth-order_operator}
\aligned
\alpha_4=\frac{25}{12}, \ & A_4( f^{n} ) = 4 f^{n} -3 f^{n-1} +  \frac{4}{3} f^{n-2} - \frac{1}{4} f^{n-3}, \\
& B_4( g^{n} ) = 4 g^{n} -  6 g^{n-1} + 4 g^{n-2}- g^{n-3} ;
\endaligned
\end{equation*} 

fifth-order scheme: 
\begin{equation*}\label{e_Fifth-order_operator}
\aligned
\alpha_5=\frac{137}{60}, \ & A_5( f^{n} ) = 5 f^{n} -5 f^{n-1} +  \frac{10}{3} f^{n-2} - \frac{5}{4} f^{n-3} + \frac{1}{5} f^{n-4}, \\
& B_5(g^{n} ) = 5 g^{n} -  10 g^{n-1} + 10 g^{n-2}- 5 g^{n-3} + g^{n-4}.
\endaligned
\end{equation*}  
 
 Then we solve $\tilde{R}^{n+1}$,  $\xi^{n+1}$  from
\begin{equation}\label{e_model_semi4}
\aligned
\frac{ \tilde{R}^{n+1} - R^n }{\Delta t} = - \xi^{n+1} \left(   M \| \nabla \tilde{ \mu}^{n+1} \|^2 + \nu \|\nabla B_k( \textbf{u}^n ) \|^2 \right), \ \ \xi^{n+1} = \frac{ \tilde{R}^{n+1} }{ E( \tilde{ \phi} ^{n+1} , \tilde{\textbf{u}}^{n+1} ) +\kappa_0 }.
\endaligned
\end{equation}  

Next we update $\phi^{n+1}$, $\mu^{n+1}$, $\textbf{u}^{n+1}$  and  $p^{n+1}$ by 
    \begin{eqnarray}
&& \phi^{n+1} = \eta_{k}^{n+1} \tilde{\phi} ^{n+1}, \ \ \mu^{n+1} = \eta_{k}^{n+1} \tilde{\mu} ^{n+1}, \ \ \textbf{u}^{n+1} =  \eta_{k}^{n+1} \tilde{ \textbf{u} } ^{n+1},    \label{e_model_semi5} \\
&&  (\nabla p^{n+1}, \nabla q) = \left(  \mu^{n+1} \nabla \phi^{n+1}   - ( \textbf{u}^{n+1} \cdot \nabla ) \textbf{u}^{n+1} - \nu \nabla \times \nabla \times \tilde{ \textbf{u} }^{n+1}, \nabla q) \right), \ \forall q\in H^1(\Omega), \label{e_model_semi6}
 \end{eqnarray}
where 
\begin{equation}\label{e_model_semi7}
\aligned
 \eta_{1}^{n+1} = 1- (1-\xi^{n+1} )^2, \  \eta_{k}^{n+1} = 1- (1-\xi^{n+1} )^k
\endaligned
\end{equation} 
with $k=2,3,4,5$. 

\textbf{Step II}: Update the scalar auxiliary variable $R^{n+1}$ as
\begin{equation}\label{e_model_semi8}
R^{n+1}=\min\left\{R^n, E( \phi ^{n+1} , \textbf{u}^{n+1} ) +\kappa_0  \right\}.
\end{equation}

 Next we prove the following unconditional energy stability for the above high-order schemes \eqref{e_model_semi1}-\eqref{e_model_semi8}: 
 \medskip

  \begin{theorem}\label{thm_energy stability_high order}
Given $R^n>0$, Then for \eqref{e_model_semi1}-\eqref{e_model_semi8}, we have  $\xi^{n+1}>0$ and 
 \begin{equation} \label{e_stability1}
\aligned
0<R^{n+1} \leq R^{n}, \ \ \forall n \leq T/\Delta t.
\endaligned
\end{equation} 
We further have the following original dissipative law:
\begin{equation} \label{e_stability_original}
\aligned
E( \phi ^{n+1} , \textbf{u}^{n+1} ) \leq E( \phi ^{n} , \textbf{u}^{n} ),
\endaligned
\end{equation}
under the condition of $ E( \phi ^{n+1} , \textbf{u}^{n+1} ) +\kappa_0   \leq R^n$. Here $ E( \phi ^{n+1} , \textbf{u}^{n+1} )  $ is the original energy. In addition, there exists a constant $M_T$ independent on $\Delta t$ such that
%\begin{equation}\label{e_stability2}
%\aligned
% \|  \textbf{u}^{n+1}  \|^2 + & \lambda \| \nabla  \phi ^{n+1} \|^2 +  \lambda \gamma \|  \phi ^{n+1} \|^2  \\
%& +  \sum_{n=0}^{m} \Delta t \xi^{n+1} ( M \| \nabla \tilde{ \mu} ^{n+1}  \|^2 + \nu \|\nabla \tilde{\textbf{u}}^{n+1}  \|^2 )
%\leq M_T, \ \ \forall n \leq T/\Delta t.
%\endaligned
%\end{equation} 

\begin{equation}\label{e_stability2}
\aligned
 \|  \textbf{u}^{n+1}  \|^2 + & \lambda \| \nabla  \phi ^{n+1} \|^2 +  \lambda \gamma \|  \phi ^{n+1} \|^2  
\leq M_T, \ \ \forall n \leq T/\Delta t.
\endaligned
\end{equation} 
\end{theorem}

\begin{proof}
Given $R^n>0$, it follows from \eqref{e_model_semi8} that \eqref{e_stability1}
 holds.
 
In addition, noting that $ \tilde{R}^0=R^0$.  It follows from \eqref{e_model_semi4} that 
 \begin{equation}\label{e_High-order_Stability1}
\aligned
0 < \tilde{R} ^{n+1} =\frac{1 }{ 1+ \Delta t \frac{ M \| \nabla \tilde{ \mu} ^{n+1}  \|^2 + \nu \|\nabla B_k(\textbf{u} ^{n} )  \|^2 }{ 
E( \tilde{ \phi}^{n+1} , \tilde{\textbf{u}}^{n+1} ) +\kappa_0 } }  R^n <  R^n,
\endaligned
\end{equation} 
then we have $\xi^{n+1}>0$. Recalling \eqref{e_model_semi8}, we have
\begin{equation}\label{e_High-order_Stability_Add1}
\aligned
R^{n} \leq E( \phi ^{n} , \textbf{u}^{n} ) +\kappa_0,
\endaligned
\end{equation} 
which leads to the desired result \eqref{e_stability_original} under the condition of $ E( \phi ^{n+1} , \textbf{u}^{n+1} ) +\kappa_0   \leq R^n$.

Denote $R^0 := M$, it then follows from \eqref{energy_reformulation} and \eqref{e_High-order_Stability1}  that 
\begin{equation}\label{e_High-order_Stability4}
\aligned
 \xi^{n+1} = \frac{ \tilde{R} ^{n+1} }{ E( \tilde{ \phi} ^{n+1} , \tilde{\textbf{u}}^{n+1} ) +\kappa_0 } \leq \frac{ 2 M }{ \| \tilde{\textbf{u}}^{n+1} \|^2 +  \lambda \| \nabla \tilde{ \phi} ^{n+1} \|^2 + \lambda \gamma \| \tilde{ \phi} ^{n+1} \|^2 +  2 },
\endaligned
\end{equation} 
where without loss of generality, we assume the positive constant $ \int_{\Omega} F(\phi) d \textbf{x} + \kappa_0>1$. Recalling \eqref{e_model_semi7}, we can derive from \eqref{e_High-order_Stability4} that there exists a positive constant $M_1$ such that
\begin{equation}\label{e_High-order_Stability5}
\aligned
| \eta_l^{n+1} | = | \xi^{n+1} P_{q}( \xi^{n+1} ) | \leq \frac{ M_1 }{  \| \tilde{\textbf{u}}^{n+1} \|^2 +  \lambda \| \nabla \tilde{ \phi} ^{n+1} \|^2 + \lambda \gamma \| \tilde{ \phi} ^{n+1} \|^2 +  2 } ,
\endaligned
\end{equation} 
where $ P_{q} $ is a polynomial function of degree $q$ with $q=1$ for $l=1$ and $q=l-1$ for $l=2,3,4,5$. Thus we have 
\begin{equation}\label{e_High-order_Stability6}
\aligned
& \|  \textbf{u}^{n+1}  \|^2 +  \lambda \| \nabla  \phi ^{n+1} \|^2 + \lambda \gamma \|  \phi ^{n+1} \|^2   = (  \eta_l^{n+1} )^2 ( \| \tilde{\textbf{u}}^{n+1} \|^2 +  \lambda \| \nabla \tilde{ \phi} ^{n+1} \|^2 + \lambda \gamma \| \tilde{ \phi} ^{n+1} \|^2 ) \\
\leq & \left( \frac{ M_1 }{ \| \tilde{\textbf{u}}^{n+1} \|^2 +  \lambda \| \nabla \tilde{ \phi} ^{n+1} \|^2 + \lambda \gamma \| \tilde{ \phi} ^{n+1} \|^2  + 2 }  \right) ^2  ( \| \tilde{\textbf{u}}^{n+1} \|^2 +  \lambda \| \nabla \tilde{ \phi} ^{n+1} \|^2 + \lambda \gamma \| \tilde{ \phi} ^{n+1} \|^2 ) \\
\leq & M_1^2,
\endaligned
\end{equation} 
which implies the desired results \eqref{e_stability2}.
%In addition, multiplying \eqref{e_model_semi4} with $\Delta t$ and taking the sum for $k$ from 0 to $n$ result in
%\begin{equation}\label{e_High-order_Stability7}
%\aligned
%\tilde{R} ^{n+1} + \sum_{k=0}^{n} \Delta t \xi^{k+1} ( M \| \nabla \tilde{ \mu} ^{k+1}  \|^2 + \nu \|\nabla \tilde{\textbf{u}}^{k+1}  \|^2 ) = R^0,
%\endaligned
%\end{equation} 
%Then we have
%\begin{equation}\label{e_High-order_Stability8}
%\aligned
%\tilde{R} ^{n+1} + \sum_{k=0}^{n} \Delta t \xi^{k+1} ( M \| \nabla \tilde{ \mu} ^{k+1}  \|^2 + \nu \|\nabla \tilde{\textbf{u}}^{k+1}  \|^2 ) \leq M,
%\endaligned
%\end{equation}

\end{proof}

Inspired by the numerical analysis technique for the SAV $\tilde{R}^{k+1}$ in \cite{zhang2022generalized}, we shall first establish the relation between $  \tilde{R}^{k+1}  $ and $ R^{k+1} $.

\medskip
\begin{lemma}\label{lem: relation_R and tildeR}
Using the definition \eqref{e_model_semi8}, we have
\begin{equation}\label{e_lem_relation_final}
\aligned
R^{k+1} = \sigma^{k+1}  \tilde{R}^{k+1} + ( 1- \sigma^{k+1} ) ( E( \phi ^{k+1} , \textbf{u}^{k+1} ) +\kappa_0 ),
\endaligned
\end{equation}  
where 
\begin{flalign*}
\renewcommand{\arraystretch}{1.5}
  \left\{
   \begin{array}{l}
   \sigma^{k+1} = 0, \  \ {\rm if \ } R^k \geq E( \phi ^{k+1} , \textbf{u}^{k+1} ) +\kappa_0, \\
   \sigma^{k+1} = 1 - \frac{  \tilde{R}^{k+1} \left(   M \| \nabla \tilde{ \mu}^{k+1} \|^2 + \nu \|\nabla  B_l ( \textbf{u}^k ) \|^2 \right) } {  (E( \tilde{ \phi} ^{k+1} , \tilde{\textbf{u}}^{k+1} ) +\kappa_0 ) ( E( \phi ^{k+1} , \textbf{u}^{k+1} ) +\kappa_0 -  \tilde{R}^{k+1} )  }  \Delta t , \ {\rm if \ }  R^k < E( \phi ^{k+1} , \textbf{u}^{k+1} ) +\kappa_0.
\end{array}\right.
\end{flalign*}
\end{lemma}

\begin{proof}
 Recalling \eqref{e_model_semi8}, it is easy to obtain that if $ R^k \geq E( \phi ^{k+1} , \textbf{u}^{k+1} ) +\kappa_0 $, 
  \begin{equation}\label{e_lem_relation1}
R^{k+1}=\min\left\{R^k, E( \phi ^{k+1} , \textbf{u}^{k+1} ) +\kappa_0  \right\} = E( \phi ^{k+1} , \textbf{u}^{k+1} ) +\kappa_0.
\end{equation}
Thus we can easily obtain that $ \sigma^{k+1} = 0 $ in \eqref{e_lem_relation_final}.
In addition, we have $ R^{k+1} = R^{k } $ under the condition that $ R^k < E( \phi ^{k+1} , \textbf{u}^{k+1} ) +\kappa_0 $. Thus by using \eqref{e_High-order_Stability1}, we have
  \begin{equation}\label{e_lem_relation2}
\tilde{R}^{k+1} < R^k = R^{k+1} < E( \phi ^{k+1} , \textbf{u}^{k+1} ) +\kappa_0. 
\end{equation}
Thus there exists a constant $0 < \sigma^{k+1} <1$ to satisfy \eqref{e_lem_relation_final}.
 Next we shall prove that 
$$ \sigma^{k+1} = 1 - \frac{  \tilde{R}^{k+1} \left(   M \| \nabla \tilde{ \mu}^{k+1} \|^2 + \nu \|\nabla  \textbf{u}^k \|^2 \right) } {  (E( \tilde{ \phi} ^{k+1} , \tilde{\textbf{u}}^{k+1} ) +\kappa_0 ) ( E( \phi ^{k+1} , \textbf{u}^{k+1} ) +\kappa_0 -  \tilde{R}^{k+1} )  }  \Delta t .$$  
Using \eqref{e_High-order_Stability1} leads to 
 \begin{equation}\label{e_lem_relation3}
\aligned
R^{k+1} = & R^k = \left( 1+ \Delta t \frac{ M \| \nabla \tilde{ \mu} ^{k+1}  \|^2 + \nu \|\nabla B_l ( \textbf{u}^{k} )  \|^2 }{ 
E( \tilde{ \phi}^{k+1} , \tilde{\textbf{u}}^{k+1} ) +\kappa_0 }  \right) \tilde{R} ^{k+1} \\
= & \sigma^{k+1}  \tilde{R}^{k+1} + ( 1- \sigma^{k+1} ) ( E( \phi ^{k+1} , \textbf{u}^{k+1} ) +\kappa_0 ) ,
\endaligned
\end{equation} 
which implies that
$$ \sigma^{k+1} = 1 - \frac{  \tilde{R}^{k+1} \left(   M \| \nabla \tilde{ \mu}^{k+1} \|^2 + \nu \|\nabla  B_l ( \textbf{u}^k ) \|^2 \right) } {  (E( \tilde{ \phi} ^{k+1} , \tilde{\textbf{u}}^{k+1} ) +\kappa_0 ) ( E( \phi ^{k+1} , \textbf{u}^{k+1} ) +\kappa_0 -  \tilde{R}^{k+1} )  }  \Delta t .$$
 \end{proof}

%==================================================================
 \section{Error estimates} 

 In this section, we carry out an error analysis for the first-order semi-discrete scheme \eqref{e_model_semi1}-\eqref{e_model_semi8}. 
 While in principle the error analysis for the high-order scheme \eqref{e_model_semi1}-\eqref{e_model_semi8} can be carried out  by combing the procedures below and those in \cite{huang2022new} for the analytical technique for the high-order scheme, but it will be much more involved and beyond the scope of this paper.

 For notational simplicity, we shall drop the dependence on $x$ for all functions when there is no confusion. Let $(\phi,\mu, \textbf{u}, p, r)$ be the exact solution of \eqref{e_model_r}, and $(\phi^{n+1},\mu^{n+1},\tilde{\textbf{u}}^{n+1}, \textbf{u}^{n+1},p^{n+1}, R^{n+1})$ be the solution of the scheme \eqref{e_model_semi1}-\eqref{e_model_semi8}, we
denote
   \begin{numcases}{}
\displaystyle \tilde{e}_{\phi}^{n+1} = \tilde{\phi}^{n+1} -  \phi(t^{n+1}),    \ \  \ e_{\phi}^{n+1}=\phi^{n+1}- \phi(t^{n+1}),  \notag \\
 \displaystyle \tilde{e}_{\mu}^{n+1} = \tilde{\mu}^{n+1} -  \mu(t^{n+1}), \ \ \ 
\displaystyle e_{\mu}^{n+1}=\mu^{n+1}-\mu(t^{n+1}), \notag \\  
\displaystyle \tilde{e}_{\textbf{u}}^{n+1}=\tilde{\textbf{u}}^{n+1}-\textbf{u}(t^{n+1}),\ \ \
\displaystyle e_{\textbf{u}}^{n+1}=\textbf{u}^{n+1}-\textbf{u}(t^{n+1}), \notag\\
\displaystyle \tilde{e}_{R}^{n+1}=\tilde{R}^{n+1}-r(t^{n+1}),  \ \ \
\displaystyle e_{R}^{n+1}=R^{n+1}-r(t^{n+1}). \notag \\
\displaystyle e_{p}^{n+1}=p^{n+1}-p(t^{n+1}).
\end{numcases}
 
The main results are stated in the following theorem:
\begin{theorem} \label{thm: error_estimate_final}
Assuming  $ \phi \in W^{2,\infty}(0,T; L^2(\Omega)) \bigcap W^{1,\infty}(0,T; H^1(\Omega)) \bigcap L^{ \infty}(0,T; H^2(\Omega))  $, $ \mu \in W^{1,\infty}(0,T; H^2(\Omega)) $, $\textbf{u}\in W^{1,\infty}(0,T;\textbf{H}^2(\Omega)) \bigcap W^{2,\infty}(0,T;\textbf{L}^{2}(\Omega)) $, and $p \in W^{1,\infty}(0,T;L^2(\Omega))$, 
then for the first-order  scheme \eqref{e_model_semi1}-\eqref{e_model_semi8}, we have
\begin{equation*}
\aligned
 \| e _{\phi}^{n+1} \|^2 + & \| \nabla e _{\phi}^{n+1} \|^2 +  \Delta t \sum\limits_{k=0}^{n} \| e _{\mu}^{k+1} \|^2 + \Delta t \sum\limits_{k=0}^{n} \| \nabla e_{\mu}^{k+1} \|^2  \\
& + \|  \nabla e_{\textbf{u}}^{n+1} \|^2  + \Delta t \sum\limits_{k=0}^{n} \| \Delta e_{\textbf{u}}^{k+1} \|^2  +  \Delta t \sum\limits_{k=0}^{n} \| \nabla e_p^{k+1} \|^2 
\leq  C (\Delta t)^2, \ \forall n\leq T/\Delta t ,
\endaligned
\end{equation*}  
where the positive constant $C$ is independent of $\Delta t$ and we modify
\begin{equation}\label{e_error_sigmak}
\aligned
\sigma^{k+1} = 1 - \frac{  \tilde{R}^{k+1} \left(   M \| \nabla \tilde{ \mu}^{k+1} \|^2 + \nu \|\nabla  \textbf{u}^k \|^2 \right) } {  (E( \tilde{ \phi} ^{k+1} , \tilde{\textbf{u}}^{k+1} ) +\kappa_0 ) ( E( \phi ^{k+1} , \textbf{u}^{k+1} ) +\kappa_0 -  \tilde{R}^{k+1} )  } (\Delta t )^2 
\endaligned
\end{equation} 
 under the condition that  $R^k < E( \phi ^{k+1} , \textbf{u}^{k+1} ) +\kappa_0$.

\end{theorem}

\medskip
\begin{rem} \label{rem_sigmak}
Since $R^k$  is an approximation of $E( \phi ^{k} , \textbf{u}^{k} ) +\kappa_0 $, we can easily obtain that $R^k \geq E( \phi ^{k+1} , \textbf{u}^{k+1} ) +\kappa_0$ in most cases by using the energy dissipation law. Thus we need a slight modification for $ \sigma^{k+1}  $ 
only in rare cases $R^k < E( \phi ^{k+1} , \textbf{u}^{k+1} ) +\kappa_0$ .
\end{rem}
 
 \medskip
 \begin{proof}
 The proof of the above theorem will be carried out through a sequence of intermediate lemmas. First we shall make the hypothesis that there exists a positive constant $C_0$  such that
\begin{equation}\label{e_error_estimate1}
\aligned
| 1- \xi^k | \leq C_0 \Delta t, \ \forall k \leq T/ \Delta t,
\endaligned
\end{equation}
\begin{equation}\label{e_error_estimate2}
\aligned
\| \tilde{e}_{\textbf{u}}^{k}  \|_{H^2} + \| \tilde{e}_{\mu}^{k}  \|_{H^1} \leq (\Delta t)^{1/6} , \ \forall k \leq T/ \Delta t,
\endaligned
\end{equation}
which will be proved in the induction process below by using a bootstrap argument.

We can easily obtain that \eqref{e_error_estimate1} and \eqref{e_error_estimate2} hold for $k=0$. Now we suppose 
\begin{equation}\label{e_error_estimate3}
\aligned
| 1- \xi^k | \leq C_0 \Delta t, \ \forall k \leq n,
\endaligned
\end{equation}
\begin{equation}\label{e_error_estimate4}
\aligned
\| \tilde{e}_{\textbf{u}}^{k}  \|_{H^2} + \| \tilde{e}_{\mu}^{k}  \|_{H^1}  \leq (\Delta t)^{1/6} , \ \forall k \leq n,
\endaligned
\end{equation}
and we shall prove that  $| 1- \xi^{n+1} | \leq C_0 \Delta t$ and $ \| \tilde{e}_{\textbf{u}}^{n+1}  \|_{H^2} + \| \tilde{e}_{\mu}^{n+1}  \|_{H^1}  \leq (\Delta t)^{1/6}  $ hold true.

\noindent{\bf Step 1:  Estimates for $H^2$ bounds of $\tilde{e}_{\textbf{m}}^{n+1}$.} 
First using exactly the same procedure in \cite{huang2021stability},
we can easily obtain that 
\begin{equation}\label{e_error5**}
\aligned
\frac{1}{2} \leq | \xi^k |, \  | \eta^k |  \leq 2, 
\endaligned
\end{equation} 
 under the condition $\Delta t \leq \min\{ \frac{1}{4 C_0 },1\}$. 
 
 We shall first derive an $H^2(\Omega)$ bound for $\phi^n$ without assuming the Lipschitz condition on $F(\phi)$. A key ingredient is the following  
 stability result   
  \begin{equation}\label{e_error1}
  \aligned
  \| \textbf{u}^{n+1} \|^2 +\| \phi^{n+1} \|_{H^1}^2+ |R^{n+1}|^2  \leq K_1,
 \endaligned
\end{equation} 
 where the positive constant $K_1$ is dependent on $\textbf{u}^{0}$ and $\phi^0$, which can be derived from the unconditionally energy stability \eqref{e_stability2}.

  \medskip
\begin{lemma} \label{lem_phi_H2_boundness}
There exists a positive constant $K_2$ independent of $\Delta t$  such  that 
$$ \| \Delta \tilde{\phi}^{k+1} \|^2+ \| \tilde{\mu}^{k+1} \|^2 
+  \sum\limits_{k=0}^{n} \Delta t \| \nabla \tilde{\mu}^{k+1} \|^2 \leq K_2,  \  \forall \ 0\leq k \leq n+1. $$  
\end{lemma}
\begin{proof}
Combining \eqref{e_model_semi1} with \eqref{e_model_semi2} and taking the inner product with $\Delta^2 \tilde{\phi}^{k+1}$ leads to
\begin{equation}\label{e_error2}
\aligned
\frac{1}{2 \Delta t}& ( \| \Delta  \tilde{\phi}^{k+1} \|^2-\| \Delta  \tilde{\phi}^{k} \|^2 +\| \Delta  \tilde{\phi}^{k+1}- \Delta  \tilde{\phi}^{k}\|^2 )+ M \lambda \| \Delta^2  \tilde{\phi}^{k+1} \|^2 +M \lambda \gamma \| \nabla \Delta  \tilde{\phi}^{k+1} \|^2 \\
=&M \lambda  ( \Delta F^{\prime}(\phi^{k}), \Delta^2 \tilde{\phi}^{k+1} )-   (  \textbf{u} ^{k} \cdot \nabla  \phi^{k} , \Delta^2 \tilde{\phi}^{k+1} ).
\endaligned
\end{equation} 
Similarly to the estimate in \cite[Lemma 2.4]{shen2018convergence}, the first term on the right hand side of \eqref{e_error2} can be controlled by the following equation with the aid of \eqref{e_error1}:
\begin{equation}\label{e_error3}
\aligned
M \lambda  ( \Delta F^{\prime}(\phi^{k}), \Delta^2 \tilde{\phi}^{k+1} ) 
\leq  &\frac{M \lambda }{4} \| \Delta^2 \tilde{\phi}^{k+1} \|^2 +C(K_1) \| \Delta F^{\prime}(\phi^{k}) \|^2 \\
\leq & \frac{M \lambda }{4} \| \Delta^2 \tilde{\phi}^{k+1} \|^2+  \frac{M \lambda}{2} \| \Delta^2 \tilde{\phi}^{k} \|^2+ C(K_1).
\endaligned
\end{equation} 

Recalling \eqref{e_error_estimate4}, we have $  \| \textbf{u}^{n} \|_{H^1} \leq C $. Then using  \eqref{e_error1} and lemmas \ref{lem: Holder inequality} and \ref{lem: Interpolation inequalities}, 
the last term on the right hand side of \eqref{e_error2} can be bounded by
\begin{equation}\label{e_error5}
\aligned
- & (\textbf{u}^{k} \cdot \nabla \phi^k,  \Delta^2 \tilde{\phi}^{k+1}) \leq  \| \textbf{u}^{k}\|_{L^6} \| \nabla \phi^k \|_{L^3} \| \Delta^2   \tilde{\phi}^{k+1} \| \\
\leq &C   \| \textbf{u}^{k}\|_{H^1}  \| \nabla  \phi^k \|^{1/2}
  \| \nabla \tilde{\phi}^k \|_{H^1}^{1/2} \| \Delta^2  \tilde{\phi}^{k+1} \| \\
 \leq & C \| \textbf{u}^{k}\|_{H^1}^2  \| \nabla \tilde{\phi}^k \|_{H^1}^2 + \frac{M \lambda}{16}\| \Delta^2 \tilde{\phi}^{k+1}\|^2 +C \| \textbf{u}^{k}\|_{H^1}^2  \| \nabla \phi^k \|^2 \\
 \leq &  \frac{M \lambda}{4}\| \Delta^2 \tilde{\phi}^{k+1}\|^2+  C \| \textbf{u}^{k}\|_{H^1}^2  (\| \Delta \tilde{\phi}^k \|^2+C(K_1) ).
  \endaligned
\end{equation} 
Combining \eqref{e_error2} with \eqref{e_error3}-\eqref{e_error5} leads to
\begin{equation}\label{e_error6}
\aligned
\frac{1}{2 \Delta t}& ( \| \Delta \tilde{\phi}^{k+1} \|^2-\| \Delta \tilde{\phi}^{k} \|^2 +\| \Delta \tilde{\phi}^{k+1}- \Delta \tilde{\phi}^{k}\|^2 )+ \frac {M \lambda} {2} \| \Delta^2 \tilde{\phi}^{k+1} \|^2 +M \lambda \gamma \| \nabla \Delta \tilde{\phi}^{k+1} \|^2 \\
&\leq  \frac{M \lambda }{2} \| \Delta^2 \tilde{\phi}^{k} \|^2 +  C \| \textbf{u}^{k}\|_{H^1}^2  (\| \Delta \tilde{\phi}^k \|^2+C(K_1) )+ C(K_1). 
\endaligned
\end{equation} 
Then multiplying \eqref{e_error6} by $2\Delta t$ and summing over  $k$, $k=0,1,2,\ldots,n$, we have 
\begin{equation}\label{e_error7}
\aligned
 \| \Delta \tilde{\phi}^{n+1} \|^2&+ M \lambda \Delta t  \| \Delta^2 \tilde{\phi}^{n+1} \|^2 
+ M  \lambda \gamma \Delta t \sum\limits_{k=0}^{n} \| \nabla \Delta \tilde{\phi}^{k+1} \|^2 \\
\leq & \| \Delta \tilde{\phi}^{0} \|^2 + M  \lambda \Delta t  \| \Delta^2 \tilde{\phi}^{0} \|^2 + C \Delta t \sum\limits_{k=0}^{n} \| \textbf{u}^{k}\|_{H^1}^2 \| \Delta \tilde{\phi}^k \|^2 + C(K_1),
\endaligned
\end{equation} 
which, together with   lemma \ref{lem: gronwall1} and equations \eqref{e_model_semi2} and \eqref{e_error1}, lead to the desired result. 
\end{proof}
 
 \medskip
\begin{lemma}\label{lem: error_estimate_phi}
Under the assumption of Theorem \ref{thm: error_estimate_final},  we have
\begin{equation} \label{e_error_estimate_phi}
\aligned
&  \lambda (
 \| \nabla \tilde{e} _{\phi}^{n+1} \|^2-  \| \nabla \tilde{e} _{\phi}^{n} \|^2 +  \| \nabla \tilde{e} _{\phi}^{n+1}- \nabla \tilde{e} _{\phi}^{n} \|^2 ) + M \Delta t  \| \nabla \tilde{e} _{\mu}^{n+1} \|^2  \\
 &  +  \lambda \gamma ( \| \tilde{e} _{\phi}^{n+1}\|^2- \| \tilde{e} _{\phi}^{n}\|^2 +\| \tilde{e} _{\phi}^{n+1}- \tilde{e} _{\phi}^{n} \|^2 )  + M \Delta t   \|  \tilde{e} _{\mu}^{n+1} \|^2 \\
 \leq &  C \Delta t \| \nabla \tilde{e} _{\phi}^{n+1} \|^2  + C\Delta t  \| \tilde{e} _{\phi}^{n} \|^2 + C\Delta t  \| \nabla \tilde{e}_{\phi}^{n} \|^2 + C\Delta t   \| \nabla \tilde{e}_{\textbf{u}}^{n} \|^2  \\
&    + C  ( \| \nabla \tilde{\textbf{u}}^{n} \|^2 + \|  \tilde{\phi}^{n} \|_{H^1}^2 ) 
C_0^4 (\Delta t)^5 + C (\Delta t)^3,
\endaligned
\end{equation}
where $C$ is a positive constant  independent of $\Delta t$ and $C_0$.
\end{lemma}

 \begin{proof}
Let $R_{\phi}^{n+1}$ be the truncation error defined by 
\begin{equation} \label{e_phi1}
\aligned
R_{\phi}^{n+1}=\frac{\partial \phi(t^{n+1})}{\partial t}- \frac{\phi(t^{n+1})-\phi(t^{n})}{\Delta t} =\frac{1}{\Delta t}\int_{t^n}^{t^{n+1}}(t^n-t)\frac{\partial^2 \phi }{\partial t^2}dt.
\endaligned
\end{equation}
 Subtracting \eqref{e_model_rA} at $t^{n+1}$ from \eqref{e_model_semi1}, we have
\begin{equation}\label{e_phi2}
\aligned
\frac{ \tilde{e} _{\phi}^{n+1}- \tilde{e} _{\phi}^n}{ \Delta t } - & M \Delta \tilde{e} _{\mu}^{n+1} = (\textbf{u}(t^{n+1}) \cdot \nabla ) \phi(t^{n+1}) - ( \textbf{u} ^{n} \cdot \nabla ) \phi^{n} + R_{\phi}^{n+1} . 
\endaligned
\end{equation}
Taking the inner product of \eqref{e_phi2} with $ \tilde{e} _{\mu}^{n+1}$ and $ \lambda \tilde{e} _{\phi}^{n+1} $, respectively leads to  
\begin{equation}\label{e_phi3}
\aligned
( \frac{ \tilde{e} _{\phi}^{n+1} - \tilde{e} _{\phi}^n}{ \Delta t },  \tilde{e} _{\mu}^{n+1}) +&  M \| \nabla \tilde{e} _{\mu}^{n+1} \|^2 \\
& =  ( (\textbf{u}(t^{n+1}) \cdot \nabla ) \phi(t^{n+1}) - ( \textbf{u} ^{n} \cdot \nabla ) \phi^{n},  \tilde{e} _{\mu}^{n+1}) + ( R_{\phi}^{n+1},  \tilde{e} _{\mu}^{n+1}) ,
\endaligned
\end{equation}
and 
\begin{equation}\label{e_phi4}
\aligned
  \frac{ \lambda }{2 \Delta t} ( \| \tilde{e} _{\phi}^{n+1} \|^2- & \| \tilde{e} _{\phi}^{n} \|^2+  \| \tilde{e} _{\phi}^{n+1}- \tilde{e} _{\phi}^{n} \|^2 )
   \\
= & \lambda ( (\textbf{u}(t^{n+1}) \cdot \nabla ) \phi(t^{n+1}) - ( \textbf{u} ^{n} \cdot \nabla ) \phi^{n},
 \tilde{e} _{\phi}^{n+1}) \\
&+ \lambda  ( R_{\phi}^{n+1},  \tilde{e} _{\phi}^{n+1}) - M \lambda ( \nabla \tilde{e} _{\mu}^{n+1}, \nabla \tilde{e} _{\phi}^{n+1} ) .
\endaligned
\end{equation}
Subtracting \eqref{e_model_rB} at $t^{n+1}$ from \eqref{e_model_semi2}, we have
\begin{equation}\label{e_phi5}
\aligned
 \tilde{e} _{\mu}^{n+1}= & - \lambda \Delta  \tilde{e} _{\phi}^{n+1} +  \lambda \gamma  \tilde{e} _{\phi}^{n+1} + \lambda  F^{\prime}( \phi^n ) - \lambda  F^{\prime}( \phi(t^{n+1}) ) . 
\endaligned
\end{equation}
Taking the inner product of \eqref{e_phi5} with $ M  \tilde{e} _{\mu}^{n+1}$ and $ \frac{  \tilde{e} _{\phi}^{n+1}-  \tilde{e} _{\phi}^{n} } {\Delta t} $, respectively results in  
\begin{equation}\label{e_phi6}
\aligned
 M \|  \tilde{e} _{\mu}^{n+1} \|^2= & M \lambda ( \nabla  \tilde{e} _{\mu}^{n+1}, \nabla  \tilde{e} _{\phi}^{n+1} ) + M \lambda \gamma ( \tilde{e} _{\phi}^{n+1},  \tilde{e} _{\mu}^{n+1} ) \\
& +M \lambda \left(  F^{\prime}( \phi^{n}) - F^{\prime}( \phi(t^{n+1}) ) , \tilde{e} _{\mu}^{n+1} \right),
\endaligned
\end{equation}
and 
\begin{equation}\label{e_phi7}
\aligned
 ( \frac{ \tilde{e} _{\phi}^{n+1}- \tilde{e} _{\phi}^{n} } {\Delta t}, \tilde{e} _{\mu}^{n+1} )= &\frac{ \lambda }{2\Delta t} (
 \| \nabla \tilde{e} _{\phi}^{n+1} \|^2-  \| \nabla \tilde{e} _{\phi}^{n} \|^2 +  \| \nabla \tilde{e} _{\phi}^{n+1}- \nabla \tilde{e} _{\phi}^{n} \|^2 ) \\
 & + \lambda \left(  F^{\prime}( \phi^{n}) - F^{\prime}( \phi(t^{n+1}) ), \frac{ \tilde{e} _{\phi}^{n+1}- \tilde{e} _{\phi}^{n} } {\Delta t} \right) \\
 &  + \frac{ \lambda \gamma}{2 \Delta t} ( \| \tilde{e} _{\phi}^{n+1}\|^2- \| \tilde{e} _{\phi}^{n}\|^2 +\| \tilde{e} _{\phi}^{n+1}- \tilde{e} _{\phi}^{n} \|^2 ).
\endaligned
\end{equation}
Combining \eqref{e_phi3} with \eqref{e_phi4}-\eqref{e_phi7}, we have
\begin{equation}\label{e_phi8}
\aligned
& \frac{ \lambda }{2\Delta t} (
 \| \nabla \tilde{e} _{\phi}^{n+1} \|^2-  \| \nabla \tilde{e} _{\phi}^{n} \|^2 +  \| \nabla \tilde{e} _{\phi}^{n+1}- \nabla \tilde{e} _{\phi}^{n} \|^2 ) + M \| \nabla \tilde{e} _{\mu}^{n+1} \|^2  \\
 &  + \frac{ \lambda \gamma}{ \Delta t} ( \| \tilde{e} _{\phi}^{n+1}\|^2- \| \tilde{e} _{\phi}^{n}\|^2 +\| \tilde{e} _{\phi}^{n+1}- \tilde{e} _{\phi}^{n} \|^2 )  + M \|  \tilde{e} _{\mu}^{n+1} \|^2 \\
 = &   ( (\textbf{u}(t^{n+1}) \cdot \nabla ) \phi(t^{n+1}) - ( \textbf{u} ^{n} \cdot \nabla ) \phi^{n},  \tilde{e} _{\mu}^{n+1} +  \lambda  \tilde{e} _{\phi}^{n+1} ) + ( R_{\phi}^{n+1},  \tilde{e} _{\mu}^{n+1}) \\
&+ \lambda  ( R_{\phi}^{n+1},  \tilde{e} _{\phi}^{n+1})  + M \lambda \gamma ( \tilde{e} _{\phi}^{n+1},  \tilde{e} _{\mu}^{n+1} ) \\
& + \lambda \left(  F^{\prime}( \phi^{n}) - F^{\prime}( \phi(t^{n+1}) , M \tilde{e} _{\mu}^{n+1} - \frac{ \tilde{e} _{\phi}^{n+1}- \tilde{e} _{\phi}^{n} } {\Delta t} \right).
\endaligned
\end{equation}
Using lemmas \ref{lem: Holder inequality} and \ref{lem: Interpolation inequalities}, the first term on the right hand side of \eqref{e_phi8} can be recast as
\begin{equation}\label{e_phi_error1}
\aligned
 &   \left( (\textbf{u}(t^{n+1}) \cdot \nabla ) \phi(t^{n+1}) - ( \textbf{u} ^{n} \cdot \nabla ) \phi^{n},  \tilde{e} _{\mu}^{n+1} +  \lambda  \tilde{e} _{\phi}^{n+1} \right) \\
 = &    \left( (\textbf{u}(t^{n+1}) \cdot \nabla ) \phi(t^{n+1}) -(\textbf{u}(t^{n}) \cdot \nabla ) \phi(t^{n}),   \tilde{e} _{\mu}^{n+1} +  \lambda  \tilde{e} _{\phi}^{n+1} \right)  \\
  & + \left((\textbf{u}(t^{n}) \cdot \nabla ) \phi(t^{n}) -  ( \textbf{u} ^{n} \cdot \nabla ) \phi^{n} ,   \tilde{e} _{\mu}^{n+1} +  \lambda  \tilde{e} _{\phi}^{n+1} \right)  \\
  =  &  \left( (\textbf{u}(t^{n+1}) \cdot \nabla ) \phi(t^{n+1}) -(\textbf{u}(t^{n}) \cdot \nabla ) \phi(t^{n}),  \tilde{e} _{\mu}^{n+1} +  \lambda  \tilde{e} _{\phi}^{n+1} \right)  \\
  & - \left( ( e_{\textbf{u}}^{n} \cdot \nabla ) \phi(t^{n}), \tilde{e} _{\mu}^{n+1} \right) -  \left( ( \textbf{u} ^{n} \cdot \nabla ) e_{\phi}^{n},  \tilde{e} _{\mu}^{n+1} +  \lambda  \tilde{e} _{\phi}^{n+1} \right) \\
  = & \frac{ M }{ 8 } \| \nabla \tilde{e} _{\mu}^{n+1} \|^2+  C \| \nabla \tilde{e} _{\phi}^{n+1} \|^2 + C \| \phi \|_{L^{\infty}(0,T; H^2(\Omega))}^2  \| e_{\textbf{u}}^{n} \|^2  \\
  &+ C \| \textbf{u} ^{n} \|_{L^3}^2 \| \nabla  e_{\phi}^{n} \|_{L^2}^2  + C \| \textbf{u} \|_{L^{\infty}(0,T; H^1(\Omega))} ^2 \| \phi \|_{W^{1,\infty}(0,T; H^1(\Omega))}^2 (\Delta t)^2 \\
  & + C \| \textbf{u} \|_{W^{1,\infty}(0,T; L^2(\Omega))} ^2 \| \phi \|_{L^{\infty}(0,T; H^1(\Omega))}^2 (\Delta t)^2.
\endaligned
\end{equation}
Recalling lemma \ref{lem_phi_H2_boundness} and \eqref{e_phi2}, the last term on the right hand side of \eqref{e_phi8} can be estimated by
\begin{equation}\label{e_phi_error2}
\aligned
& \lambda \left(  F^{\prime}( \phi^{n}) - F^{\prime}( \phi(t^{n+1}) , M \tilde{e} _{\mu}^{n+1} - \frac{ \tilde{e} _{\phi}^{n+1}- \tilde{e} _{\phi}^{n} } {\Delta t} \right) \\
= & \lambda \left(  F^{\prime}( \phi^{n}) - F^{\prime}( \phi(t^{n+1}) , M \tilde{e} _{\mu}^{n+1}  - R_{\phi}^{n+1} \right) \\
& - \lambda \left(  F^{\prime}( \phi^{n}) - F^{\prime}( \phi(t^{n+1}) , (\textbf{u}(t^{n+1}) \cdot \nabla ) \phi(t^{n+1}) - ( \textbf{u} ^{n} \cdot \nabla ) \phi^{n} \right) \\
& + \lambda \left( \nabla ( F^{\prime}( \phi^{n}) - F^{\prime}( \phi(t^{n+1}) ) , M \nabla \tilde{e} _{\mu}^{n+1}  \right) \\
\leq & \frac{ M }{8} \| \tilde{e} _{\mu}^{n+1}  \|^2 +  \frac{ M }{8} \| \nabla \tilde{e} _{\mu}^{n+1}  \|^2  + C \| e_{\phi}^{n} \|^2 + C \| \nabla e_{\phi}^{n} \|^2 \\
&  + C \| \phi \|_{W^{1,\infty}(0,T; H^1(\Omega))}^2 (\Delta t)^2  + C \| \phi \|_{L^{\infty}(0,T; H^2(\Omega))}^2  \| e_{\textbf{u}}^{n} \|^2  \\
  &+ C \| \textbf{u} ^{n} \|_{L^3}^2 \| \nabla  e_{\phi}^{n} \|_{L^2}^2  + C \| \textbf{u} \|_{L^{\infty}(0,T; H^1(\Omega))} ^2 \| \phi \|_{W^{1,\infty}(0,T; H^1(\Omega))}^2 (\Delta t)^2 \\
  & + C \| \textbf{u} \|_{W^{1,\infty}(0,T; L^2(\Omega))} ^2 \| \phi \|_{L^{\infty}(0,T; H^1(\Omega))}^2 (\Delta t)^2 \\
  & + C \| \phi \|_{W^{2,\infty}(0,T; L^2(\Omega))}^2  (\Delta t)^2 .
\endaligned
\end{equation}
Using Poincare inequality, we have $\| e_{\textbf{u}}^{n} \|^2 \leq C \| \nabla e_{\textbf{u}}^{n} \|^2$ and using \eqref{e_model_semi5} and \eqref{e_error_estimate3}, we have 
\begin{equation}\label{e_relation_u_tilde}
\aligned
  \| \nabla e_{\textbf{u}}^{n} \|^2 \leq&  2 \| \nabla \tilde{e}_{\textbf{u}}^{n} \|^2 + 2 |1- \eta^{n} |^2 \| \nabla \tilde{\textbf{u}}^{n} \|^2 \\
\leq & 2 \| \nabla \tilde{e}_{\textbf{u}}^{n} \|^2 + 2  \| \nabla \tilde{\textbf{u}}^{n} \|^2 C_0^4 (\Delta t)^4,
\endaligned
\end{equation} 
and 
\begin{equation}\label{e_relation_phi_tilde}
\aligned
  \|  e_{\phi}^{n} \|_{H^2}^2 \leq &   2 \| \tilde{e}_{\phi}^{n} \|_{H^2}^2 + 2  \|  \tilde{\phi}^{n} \|_{H^2}^2 C_0^4 (\Delta t)^4,
\endaligned
\end{equation}
Then combining \eqref{e_phi8} with \eqref{e_phi_error1} and \eqref{e_phi_error2} and multiplying $ 2\Delta t$ on both sides, we have
\begin{equation}\label{e_phi_error3}
\aligned
&  \lambda (
 \| \nabla \tilde{e} _{\phi}^{n+1} \|^2-  \| \nabla \tilde{e} _{\phi}^{n} \|^2 +  \| \nabla \tilde{e} _{\phi}^{n+1}- \nabla \tilde{e} _{\phi}^{n} \|^2 ) + M \Delta t  \| \nabla \tilde{e} _{\mu}^{n+1} \|^2  \\
 &  +  \lambda \gamma ( \| \tilde{e} _{\phi}^{n+1}\|^2- \| \tilde{e} _{\phi}^{n}\|^2 +\| \tilde{e} _{\phi}^{n+1}- \tilde{e} _{\phi}^{n} \|^2 )  + M \Delta t   \|  \tilde{e} _{\mu}^{n+1} \|^2 \\
 \leq &  C \Delta t  \| \nabla \tilde{e} _{\phi}^{n+1} \|^2  + C \Delta t  \| \tilde{e} _{\phi}^{n} \|^2 + C \Delta t  \| \nabla \tilde{e}_{\phi}^{n} \|^2 + C \Delta t  \| \textbf{u} ^{n} \|_{L^3}^2 \| \nabla \tilde{e}_{\phi}^{n} \|^2  \\
&  + C \| \phi \|_{W^{1,\infty}(0,T; H^1(\Omega))}^2 (\Delta t)^3  + C \Delta t  \| \phi \|_{L^{\infty}(0,T; H^2(\Omega))}^2  \| \nabla \tilde{e}_{\textbf{u}}^{n} \|^2  \\
  & + C \| \textbf{u} \|_{L^{\infty}(0,T; H^1(\Omega))} ^2 \| \phi \|_{W^{1,\infty}(0,T; H^1(\Omega))}^2 (\Delta t)^3 + C  \| \nabla \tilde{\textbf{u}}^{n} \|^2 C_0^4 (\Delta t)^5 \\
  & + C \| \textbf{u} \|_{W^{1,\infty}(0,T; L^2(\Omega))} ^2 \| \phi \|_{L^{\infty}(0,T; H^1(\Omega))}^2 (\Delta t)^3 + C   \|  \tilde{\phi}^{n} \|_{H^1}^2 C_0^4 (\Delta t)^5 \\
  & + C \| \phi \|_{W^{2,\infty}(0,T; L^2(\Omega))}^2  (\Delta t)^3 ,
\endaligned
\end{equation}
which implies the desired result \eqref{e_error_estimate_phi}.
\end{proof} 

Next we first establish error estimate for the  commutator of the Laplacian and Leray-Helmholtz projection operators to bound the part of pressure. Similar to \cite{liu2007stability},  we let $\mathcal{P}$ denote the Leray-Helmholtz projection operator onto divergence-free fields, defined as follows. Given any $\textbf{b} \in L^2(\Omega, \mathbb{R}^d)$, there is a unique $q \in H^1(\Omega) $ with $ \int_{\Omega} q =0 $ such that $ \mathcal{P} \textbf{b} =  \textbf{b} + \nabla q$ satisfies 
\begin{equation}\label{e_error_stokespre1}
\aligned
(\textbf{b} + \nabla q, \nabla \phi ) = ( \mathcal{P} \textbf{b}, \nabla \phi )=0, \ \forall \phi \in 
H^1(\Omega).
\endaligned
\end{equation} 
Then for $\textbf{u} \in L^2(\Omega, \mathbb{R}^d)$, we have  \cite{liu2007stability}
\begin{equation}\label{e_error_stokespre2}
\aligned
\Delta \mathcal{P} \textbf{u} = \Delta \textbf{u} - \nabla \nabla \cdot \textbf{u} = - \nabla \times \nabla \times  \textbf{u}.
\endaligned
\end{equation} 
Next we recall the estimate for commutator of the Laplacian and Leray-Helmholtz projection operators. 
\begin{lemma}\label{lem: commutator}
\cite{liu2007stability} Let $\Omega \subset \mathbb{R}^d $ be a connected bounded domain with $C^3$ boundary. Then for any $\epsilon >0$, there exists a positive constant $C \geq 0$ such that for all vector fields $\textbf{u} \in H^2 \cap H^1_0 (\Omega, \mathbb{R}^d )$,
\begin{equation}\label{e_error_stokespre3}
\aligned
\int_{\Omega} | (\Delta \mathcal{P} - \mathcal{P} \Delta) \textbf{u} |^2 \leq (\frac1 2 + \epsilon ) \int_{\Omega} |\Delta \textbf{u} |^2 + C \int_{\Omega} |\nabla \textbf{u} |^2.
\endaligned
\end{equation} 
\end{lemma}

We define the Stokes pressure $p_s(\textbf{u})$ by 
\begin{equation}\label{e_error_stokespre4}
\aligned
\nabla p_s(\textbf{u}) =  (\Delta \mathcal{P} - \mathcal{P} \Delta) \textbf{u},
\endaligned
\end{equation} 
where the Stokes pressure is generated by the tangential part of vorticity at the boundary in two and three dimensions by \cite{liu2007stability,liu2009error}
\begin{equation}\label{e_error_stokespre5}
\aligned
\int_{\Omega} \nabla p_s(\textbf{u}) \cdot \nabla \phi = \int_{\Gamma} (\nabla \times \textbf{u} ) \cdot ( \textbf{n} \times \nabla \phi ), \ \forall \phi \in H^1(\Omega).
\endaligned
\end{equation} 
Then by using \eqref{e_error_stokespre2}, we have 
\begin{equation}\label{e_error_stokespre6}
\aligned
\nabla p_s(\textbf{u}) =  (\Delta \mathcal{P} - \mathcal{P} \Delta) \textbf{u} = (I - \mathcal{P} ) \Delta  \textbf{u} - \nabla \nabla \cdot \textbf{u} = (I - \mathcal{P} ) ( \Delta  \textbf{u} - \nabla \nabla \cdot \textbf{u} ). 
\endaligned
\end{equation} 
Recalling \eqref{e_error_stokespre1}, we have 
\begin{equation}\label{e_error_stokespre7}
\aligned
\int_{\Omega} \nabla p_s(\textbf{u}) \cdot \nabla \phi = \int_{\Omega} ( \Delta \textbf{u} - \nabla \nabla \cdot \textbf{u} ) \cdot \nabla \phi, \ \forall \phi \in H^1(\Omega).
\endaligned
\end{equation} 

\medskip
\begin{lemma}\label{lem: error_estimate_u}
Under the assumption of Theorem \ref{thm: error_estimate_final},  we have
\begin{equation}  \label{e_error_estimate_u}
\aligned
& \frac{M}{ 2 K_3} ( \| \nabla \tilde{e}_{\textbf{u}}^{n+1} \|^2 - \| \nabla \tilde{e}_{\textbf{u}}^{n} \|^2 + \| \nabla \tilde{e}_{\textbf{u}}^{n+1}- \nabla \tilde{e}_{\textbf{u}}^{n} \|^2) \\
& +\frac{M}{ 2 K_3} (1- \frac{3  (1-\alpha) }{8 }  ) \nu \Delta t \| \Delta \tilde{e}_{\textbf{u}}^{n+1} \|^2 \\
\leq & \frac{M}{ 2 K_3} (  \alpha +  \frac{ (1-\alpha) }{8 } ) \nu \Delta t \| \Delta \tilde{e}_\textbf{u} ^{n} \|^2+  C \Delta t \| \nabla  \tilde{e}_\textbf{u} ^{n} \|^2 \\
& + C \Delta t \| \tilde{e}_{\phi}^n \|^2 + C \Delta t \| \nabla \tilde{e}_{\phi}^n \|^2  + \frac{M}{  2} \Delta t  (\|  \tilde{e}_{\mu}^{n} \|^2+ \| \nabla \tilde{e}_{\mu}^{n} \|^2 ) \\
&   + C ( \| \Delta \tilde{\textbf{u}}^{n} \|^2+ \| \nabla \tilde{\textbf{u}}^{n} \|^2 )C_0^4 (\Delta t)^5 \\
 &  + C \|  \tilde{\phi}^{n} \|_{H^2}^2 C_0^4 (\Delta t)^5
 + C  \| \nabla  \tilde{\mu}^{n} \|^2 C_0^4 (\Delta t)^5  + C  (\Delta t)^3,
\endaligned
\end{equation}
where the positive constant $C$ is independent of $\Delta t$ and $C_0$. 
\end{lemma}

\begin{proof}
Let $\textbf{R}_{\textbf{u}}^{n+1}$ be the truncation error defined by  
\begin{equation}\label{e_velocity1}
\aligned
\textbf{R}_{\textbf{u}}^{n+1}=\frac{\partial \textbf{u}(t^{n+1})}{\partial t}- \frac{\textbf{u}(t^{n+1})-\textbf{u}(t^{n})}{\Delta t}=\frac{1}{\Delta t}\int_{t^n}^{t^{n+1}}(t^n-t)\frac{\partial^2 \textbf{u}}{\partial t^2}dt.
\endaligned
\end{equation}
Subtracting \eqref{e_model_rC} at $t^{n+1}$ from \eqref{e_model_semi3}, we obtain
\begin{equation}\label{e_velocity2}
\aligned
\frac{ \tilde{e}_{\textbf{u}}^{n+1} - \tilde{e}_{\textbf{u}}^n}{\Delta t} - & \nu\Delta\tilde{e}_{\textbf{u}}^{n+1}
=   (\textbf{u}(t^{n+1})\cdot \nabla)\textbf{u}(t^{n+1}) -  \textbf{u}^{n}\cdot \nabla\textbf{u}^{n} \\
& - \nabla (p^n-p(t^{n+1}))  + \mu ^{n} \nabla \phi^{n} - \mu(t^{n+1}) \nabla \phi(t^{n+1} ) 
+\textbf{R}_{\textbf{u}}^{n+1}.
\endaligned
\end{equation}

Next we establish an error equation for pressure corresponding to \eqref{e_model_semi6} by
\begin{equation}\label{e_velocity3}
\aligned
(\nabla e_p^{n+1}, \nabla q) = & \left(  ( \textbf{u}(t^{n+1} ) \cdot \nabla ) \textbf{u}( t^{n+1} ) - ( \textbf{u}^{n+1} \cdot \nabla ) \textbf{u}^{n+1} , \nabla q \right) \\
& + \left(   \mu^{n+1} \nabla \phi^{n+1} -  \mu( t^{n+1} ) \nabla \phi ( t^{n+1} ), \nabla q \right) \\
& - ( \nu \nabla \times \nabla \times \tilde{e}_\textbf{u} ^{n+1}, \nabla q), \ \forall q\in H^1(\Omega).
\endaligned
\end{equation}
Taking $ q = e_p^{n+1} $ in \eqref{e_velocity3} leads to
\begin{equation}\label{e_velocity4}
\aligned
 \| \nabla e_p^{n+1} \| \leq &  \nu \|  \nabla e_{ps}^{n+1}( \tilde{e}_\textbf{u} ) \|+  \| e_{ \textbf{u} }^{n+1} \|_{L^6} \| \nabla \textbf{u}(t^{n+1} ) \|_{L^3} + \|  \textbf{u}^{n+1} \|_{L^6} \| \nabla e_{ \textbf{u} }^{n+1}\|_{L^3} \\ 
& + \|  e_{ \mu }^{n+1} \|_{L^6} \| \nabla  \phi ^{n+1} \|_{L^3} + \| \mu(t^{n+1} ) \| _{L^6} \| \nabla e_{ \phi}^{n+1} \|_{L^3} .
\endaligned
\end{equation} 
Recalling \eqref{e_error_stokespre7} and lemma \ref{lem: commutator}, we have
\begin{equation}\label{e_velocity5}
\aligned
\nu \|  \nabla e_{ps}^{n+1}( \tilde{e}_\textbf{u} ) \|^2 
\leq & \nu \alpha \| \Delta \tilde{e}_\textbf{u} ^{n+1} \|^2 + \nu  C_{\alpha} \| \nabla  \tilde{e}_\textbf{u} ^{n+1} \|^2 ,
\endaligned
\end{equation} 
where the positive constant $ \frac1 2 < \alpha < 1$.

Taking the inner product of \eqref{e_velocity2} with $-2 \Delta t \Delta \tilde{e}_{\textbf{u}}^{n+1}$ and using lemmas \ref{lem: Holder inequality} and \ref{lem: Interpolation inequalities}, we have
\begin{equation}\label{e_velocity6}
\aligned
& ( \| \nabla \tilde{e}_{\textbf{u}}^{n+1} \|^2 - \| \nabla \tilde{e}_{\textbf{u}}^{n} \|^2 + \| \nabla \tilde{e}_{\textbf{u}}^{n+1}- \nabla \tilde{e}_{\textbf{u}}^{n} \|^2) + 2 \nu \Delta t \| \Delta \tilde{e}_{\textbf{u}}^{n+1} \|^2 \\
\leq &  2\Delta t \| \Delta \tilde{e}_{\textbf{u}}^{n+1}  \| \| \nabla e_p^n \|  + 2\Delta t \| \Delta \tilde{e}_{\textbf{u}}^{n+1}  \|  (  
\| e_{ \textbf{u} }^{n} \|_{L^6} \| \nabla \textbf{u}(t^{n} ) \|_{L^3} + \|  \textbf{u}^{n} \|_{L^6} \| \nabla e_{ \textbf{u} }^{n}\|_{L^3} ) \\
& +  2 \Delta t \| \Delta \tilde{e}_{\textbf{u}}^{n+1}  \|  \|p(t^n)-p(t^{n+1}) \| +  2 \Delta t \| \Delta \tilde{e}_{\textbf{u}}^{n+1}  \|  ( 
 \| e_{ \mu }^{n} \|_{L^6} \| \nabla  \phi ^{n} \|_{L^3} + \| \mu(t^{n} ) \| _{L^6} \| \nabla e_{ \phi}^{n} \|_{L^3} ) \\
& + 2 \Delta t \| \Delta \tilde{e}_{\textbf{u}}^{n+1}  \|  (\| \textbf{R}_{\textbf{u}}^{n+1} \|+ \| (\textbf{u}(t^{n+1}) \cdot \nabla)\textbf{u}(t^{n+1}) - ( \textbf{u}(t^{n})\cdot \nabla)\textbf{u}(t^{n})  \|  ) \\
& + 2 \Delta t \| \Delta \tilde{e}_{\textbf{u}}^{n+1}  \|  \| \mu(t^{n+1}) \nabla \phi(t^{n+1} )  - \mu(t^{n}) \nabla \phi(t^{n} )  \|   \\
\leq & 2 \nu \Delta t \| \Delta \tilde{e}_{\textbf{u}}^{n+1}  \| \|  \nabla e_{ps}^{n}( \tilde{e}_\textbf{u} ) \| +  
4 \Delta t \| \Delta \tilde{e}_{\textbf{u}}^{n+1}  \|  (  
\| e_{ \textbf{u} }^{n} \|_{L^6} \| \nabla \textbf{u}(t^{n} ) \|_{L^3} + \|  \textbf{u}^{n} \|_{L^6} \| \nabla e_{ \textbf{u} }^{n}\|_{L^3} )  \\
 & + 4 \Delta t \| \Delta \tilde{e}_{\textbf{u}}^{n+1}  \|  (  \|  e_{ \mu }^{n} \|_{L^6} \| \nabla  \phi ^{n} \|_{L^3} + \| \mu(t^{n} ) \| _{L^6} \| \nabla e_{ \phi}^{n} \|_{L^3} ) \\
& +  2 \Delta t \| \Delta \tilde{e}_{\textbf{u}}^{n+1}  \| ( \|p(t^n)-p(t^{n+1}) \|+\| \textbf{R}_{\textbf{u}}^{n+1} \|+ \| (\textbf{u}(t^{n+1})\cdot \nabla)\textbf{u}(t^{n+1}) - ( \textbf{u}(t^{n})\cdot \nabla)\textbf{u}(t^{n})  \|  ) \\
&+ 2 \Delta t \| \Delta \tilde{e}_{\textbf{u}}^{n+1}  \|  \| \mu(t^{n+1}) \nabla \phi(t^{n+1} )  - \mu(t^{n}) \nabla \phi(t^{n} )  \| .
\endaligned
\end{equation}
Using Cauchy-Schwarz inequality, the first term on the right hand side of \eqref{e_velocity6} can be estimated by
\begin{equation}\label{e_velocity7}
\aligned
& 2 \nu \Delta t \| \Delta \tilde{e}_{\textbf{u}}^{n+1}  \| \|  \nabla e_{ps}^{n}( \tilde{e}_\textbf{u} ) \| \leq \nu \Delta t  \| \Delta \tilde{e}_{\textbf{u}}^{n+1}  \|^2 +  \nu \alpha \Delta t \| \Delta \tilde{e}_\textbf{u} ^{n} \|^2+ \nu  C_{\alpha} \Delta t \| \nabla  \tilde{e}_\textbf{u} ^{n} \|^2.
\endaligned
\end{equation}
Recalling \eqref{e_error_estimate4}, we can obtain $  \| \textbf{u}^{n} \|_{H^2} \leq C $ and by using \eqref{e_model_semi5} and \eqref{e_error_estimate3}, we have
\begin{equation}\label{e_velocity8}
\aligned
  \| \Delta e_{\textbf{u}}^{n} \|^2 \leq&  2 \| \Delta \tilde{e}_{\textbf{u}}^{n} \|^2 + 2 |1- \eta^{n} |^2 \| \Delta \tilde{\textbf{u}}^{n} \|^2 \\
\leq & 2 \| \Delta \tilde{e}_{\textbf{u}}^{n} \|^2 + 2  \| \Delta \tilde{\textbf{u}}^{n} \|^2 C_0^4 (\Delta t)^4.
\endaligned
\end{equation}
Thus the second term on the right hand side of \eqref{e_velocity6} can be estimated by
\begin{equation}\label{e_velocity10}
\aligned
& 4 \Delta t \| \Delta \tilde{e}_{\textbf{u}}^{n+1}  \|  (  
\| e_{ \textbf{u} }^{n} \|_{L^6} \| \nabla \textbf{u}(t^{n} ) \|_{L^3} + \|  \textbf{u}^{n} \|_{L^6} \| \nabla e_{ \textbf{u} }^{n}\|_{L^3} ) \\
\leq &  \frac{ (1-\alpha) \nu}{8 }  \Delta t  \| \Delta \tilde{e}_{\textbf{u}}^{n+1} \|^2 + C \Delta t  \| \nabla e_{ \textbf{u} }^{n} \|^2 \| \nabla \textbf{u}(t^{n} ) \|  \| \nabla \textbf{u}(t^{n})  \|_{H^1} + C \Delta t \| \nabla \textbf{u}^{n} \|^2 \| \nabla e_{ \textbf{u} }^{n}  \|  \| \nabla e_{ \textbf{u} }^{n}  \|_{H^1}  \\
\leq & \frac{ (1-\alpha) \nu}{8 } \Delta t  \| \Delta \tilde{e}_{\textbf{u}}^{k+1} \|^2 +   \frac{ (1-\alpha) \nu}{16 } \Delta t  \| \Delta e_{\textbf{u}}^{k} \|^2 +  C \Delta t  \| \nabla e_{ \textbf{u} }^{k} \|^2 \\
\leq & \frac{ (1-\alpha) \nu}{8 } \Delta t  \| \Delta \tilde{e}_{\textbf{u}}^{n+1} \|^2 +   \frac{ (1-\alpha) \nu}{8 }  \Delta t  \| \Delta \tilde{e}_{\textbf{u}}^{n} \|^2 + C \Delta t  \| \nabla \tilde{e}_{\textbf{u}}^{n} \|^2  \\
&+ C ( \| \Delta \tilde{\textbf{u}}^{n} \|^2+ \| \nabla \tilde{\textbf{u}}^{n} \|^2 )C_0^4 (\Delta t)^5.
\endaligned
\end{equation}
Using \eqref{e_phi5}, we have
\begin{equation}\label{e_velocity11}
\aligned
 \Delta  \tilde{e} _{\phi}^{n+1} =    \gamma  \tilde{e} _{\phi}^{n+1} +  F^{\prime}( \phi^n ) -  F^{\prime}( \phi(t^{n+1}) ) -  \frac{1}{\lambda} \tilde{e} _{\mu}^{n+1} .
\endaligned
\end{equation}
Recalling \eqref{e_error_estimate4} leads to $  \| \mu^{n} \|_{H^1} \leq C $. In addition, we have
\begin{equation}\label{e_velocity12}
\aligned
  \|  e_{\phi}^{n} \|_{H^2}^2 \leq &   2 \| \tilde{e}_{\phi}^{n} \|_{H^2}^2 + 2  \|  \tilde{\phi}^{n} \|_{H^2}^2 C_0^4 (\Delta t)^4,
\endaligned
\end{equation}
and 
\begin{equation}\label{e_velocity13}
\aligned
  \|  \nabla e_{\mu}^{n} \|^2 \leq &   2 \| \nabla \tilde{e}_{\mu}^{n} \|^2 + 2  \| \nabla  \tilde{\mu}^{n} \|^2 C_0^4 (\Delta t)^4.
\endaligned
\end{equation}
Hence the third term on the right hand side of \eqref{e_velocity6} can be bounded by
\begin{equation}\label{e_velocity14}
\aligned
 & 4 \Delta t \| \Delta \tilde{e}_{\textbf{u}}^{n+1}  \|  (  \|  e_{ \mu }^{n} \|_{L^6} \| \nabla  \phi ^{n} \|_{L^3} + \| \mu(t^{n} ) \| _{L^6} \| \nabla e_{ \phi}^{n} \|_{L^3} ) \\
 \leq & 4 \Delta t \| \Delta \tilde{e}_{\textbf{u}}^{n+1}  \|  (  \|  e_{ \mu }^{n} \|_{H^1} \| \nabla  \phi ^{n} \|^{1/2} \| \nabla  \phi ^{n} \|_{H^1}^{1/2}+ \| \mu(t^{n} ) \| _{H^1} \| \nabla e_{ \phi}^{n} \|^{1/2} \| \nabla e_{ \phi}^{n} \|_{H^1}^{1/2} ) \\
 \leq & \frac{ (1-\alpha) \nu}{8 } \Delta t  \| \Delta \tilde{e}_{\textbf{u}}^{n +1} \|^2 + K_3 \Delta t  (\|  \tilde{e}_{\mu}^{n} \|^2+ \| \nabla \tilde{e}_{\mu}^{n} \|^2 )+ C \Delta t \| \tilde{e}_{\phi}^n \|^2  \\
 & + C \Delta t \| \nabla \tilde{e}_{\phi}^n \|^2 + C \|  \tilde{\phi}^{n} \|_{H^2}^2 C_0^4 (\Delta t)^5
 + C  \| \nabla  \tilde{\mu}^{n} \|^2 C_0^4 (\Delta t)^5 \\
 & + C\| \phi \|_{W^{1,\infty}(0,T; H^1(\Omega))}^2 \| \mu \|_{L^{\infty}(0,T; H^1(\Omega))}^2  (\Delta t)^3,
\endaligned
\end{equation}
where  $K_3$ is a positive constant which is independent of $\Delta t$ and $C_0$.

Using Cauchy-Schwarz inequality, the last two terms on the right hand side of \eqref{e_velocity6} can be bounded by
\begin{equation}\label{e_velocity15}
\aligned
&   2 \Delta t \| \Delta \tilde{e}_{\textbf{u}}^{n+1}  \| ( \|p(t^n)-p(t^{n+1}) \|+\| \textbf{R}_{\textbf{u}}^{n+1} \|+ \| (\textbf{u}(t^{n+1})\cdot \nabla)\textbf{u}(t^{n+1}) - ( \textbf{u}(t^{n})\cdot \nabla)\textbf{u}(t^{n})  \|  ) \\
&+ 2 \Delta t \| \Delta \tilde{e}_{\textbf{u}}^{n+1}  \|  \| \mu(t^{n+1}) \nabla \phi(t^{n+1} )  - \mu(t^{n}) \nabla \phi(t^{n} )  \| \\
\leq & \frac{ (1-\alpha) \nu}{8 }  \Delta t \| \Delta \tilde{e}_{\textbf{u}}^{n+1} \|^2 + C (\Delta t)^2 \left( \int_{t^n}^{t^{n+1}} \| p_t \|^2 dt + \int_{t^n}^{t^{n+1}}\|\textbf{u}_{tt}\|^2 dt 
\right) \\
& + C (\Delta t)^2 \left(  \int_{t^n}^{t^{n+1}}\| \nabla \textbf{u}_{t}\|^2 dt \| \textbf{u}(t^{n+1}) \|_{H^2}^2 +\| \textbf{u}(t^{n}) \|_{H^1}^2 \int_{t^n}^{t^{n+1}}\| \textbf{u}_{t}\|_{H^2}^2 dt \right) \\
& +  C (\Delta t)^2 \left(  \int_{t^n}^{t^{n+1}}\| \nabla \phi_{t}\|^2 dt \| \mu(t^{n+1}) \|_{H^2}^2 +\| \phi(t^{n}) \|_{H^1}^2 \int_{t^n}^{t^{n+1}}\| \mu_{t}\|_{H^2}^2 dt \right) .
\endaligned
\end{equation}
 Finally,  combining \eqref{e_velocity6} with \eqref{e_velocity7}-\eqref{e_velocity15}, we obtain
\begin{equation}\label{e_velocity16}
\aligned
& ( \| \nabla \tilde{e}_{\textbf{u}}^{n+1} \|^2 - \| \nabla \tilde{e}_{\textbf{u}}^{n} \|^2 + \| \nabla \tilde{e}_{\textbf{u}}^{n+1}- \nabla \tilde{e}_{\textbf{u}}^{n} \|^2) + (1- \frac{3  (1-\alpha) }{8 }  ) \nu \Delta t \| \Delta \tilde{e}_{\textbf{u}}^{n+1} \|^2 \\
\leq & (  \alpha +  \frac{ (1-\alpha) }{8 } ) \nu \Delta t \| \Delta \tilde{e}_\textbf{u} ^{n} \|^2+  C \Delta t \| \nabla  \tilde{e}_\textbf{u} ^{n} \|^2 + C \Delta t \| \tilde{e}_{\phi}^n \|^2 + C \Delta t \| \nabla \tilde{e}_{\phi}^n \|^2 \\
&  + K_3 \Delta t  (\|  \tilde{e}_{\mu}^{n} \|^2+ \| \nabla \tilde{e}_{\mu}^{n} \|^2 )  + C ( \| \Delta \tilde{\textbf{u}}^{n} \|^2+ \| \nabla \tilde{\textbf{u}}^{n} \|^2 )C_0^4 (\Delta t)^5 \\
 &  + C \|  \tilde{\phi}^{n} \|_{H^2}^2 C_0^4 (\Delta t)^5
 + C  \| \nabla  \tilde{\mu}^{n} \|^2 C_0^4 (\Delta t)^5 \\
& + C (  \| p \|_{W^{1,\infty}(0,T; L^2(\Omega)) }^2 + \| \textbf{u} \|_{W^{2,\infty}(0,T; L^2(\Omega)) }^2 ) (\Delta t)^3 \\
& + C   \| \textbf{u} \|_{W^{1,\infty}(0,T; H^1(\Omega)) }^2 \| \textbf{u} \|_{L^{\infty}(0,T; H^2(\Omega)) }^2 (\Delta t)^3 \\
& + C   \| \textbf{u} \|_{L^{\infty}(0,T; H^1(\Omega)) }^2 \| \textbf{u} \|_{W^{1,\infty}(0,T; H^2(\Omega)) }^2 (\Delta t)^3 \\
& +  C \| \phi \|_{W^{1,\infty}(0,T; H^1(\Omega))}^2 \| \mu \|_{L^{\infty}(0,T; H^2(\Omega))}^2  (\Delta t)^3 \\
 & +  C \| \phi \|_{L^{\infty}(0,T; H^1(\Omega))}^2 \| \mu \|_{W^{1,\infty}(0,T; H^2(\Omega))}^2  (\Delta t)^3,
\endaligned
\end{equation}
which leads to the desired result \eqref{e_error_estimate_u} by multiplying $ \frac{ M}{ 2 K_3 }$ on both sides of \eqref{e_velocity16}.

\end{proof}

Combining lemmas \ref{lem: error_estimate_phi} and \ref{lem: error_estimate_u}, we have 
\begin{equation} \label{e_error_final_utilde1}
\aligned
&  \lambda (
 \| \nabla \tilde{e} _{\phi}^{k+1} \|^2 -  \| \nabla \tilde{e} _{\phi}^{k} \|^2 +  \| \nabla \tilde{e} _{\phi}^{k+1}- \nabla \tilde{e} _{\phi}^{k} \|^2 ) + M \Delta t  \| \nabla \tilde{e} _{\mu}^{k+1} \|^2  \\
 &  +  \lambda \gamma ( \| \tilde{e} _{\phi}^{k+1}\|^2- \| \tilde{e} _{\phi}^{k}\|^2 +\| \tilde{e} _{\phi}^{k+1}- \tilde{e} _{\phi}^{k} \|^2 )  + M \Delta t   \|  \tilde{e} _{\mu}^{k+1} \|^2 \\
 & + \frac{M}{ 2 K_3} ( \| \nabla \tilde{e}_{\textbf{u}}^{k+1} \|^2 - \| \nabla \tilde{e}_{\textbf{u}}^{k} \|^2 + \| \nabla \tilde{e}_{\textbf{u}}^{k+1}- \nabla \tilde{e}_{\textbf{u}}^{k} \|^2) \\
& +\frac{M}{ 2 K_3} (1- \frac{3  (1-\alpha) }{8 }  ) \nu \Delta t \| \Delta \tilde{e}_{\textbf{u}}^{k+1} \|^2 \\
 \leq &  C_1 \Delta t \| \nabla \tilde{e} _{\phi}^{k+1} \|^2  + C\Delta t  \| \tilde{e} _{\phi}^{k} \|^2 + C\Delta t  \| \nabla \tilde{e}_{\phi}^{k} \|^2 + C\Delta t   \| \nabla \tilde{e}_{\textbf{u}}^{k} \|^2  \\
&    + C  ( \| \nabla \tilde{\textbf{u}}^{k} \|^2 + \|  \tilde{\phi}^{k} \|_{H^1}^2 ) 
C_0^4 (\Delta t)^5 + C (\Delta t)^3 \\
 & + \frac{M}{ 2 K_3} (  \alpha +  \frac{ (1-\alpha) }{8 } ) \nu \Delta t \| \Delta \tilde{e}_\textbf{u} ^{k} \|^2+  C \Delta t \| \nabla  \tilde{e}_\textbf{u} ^{k} \|^2 \\
& + C \Delta t \| \tilde{e}_{\phi}^k \|^2 + C \Delta t \| \nabla \tilde{e}_{\phi}^k \|^2  + \frac{M}{  2} \Delta t  (\|  \tilde{e}_{\mu}^{k} \|^2+ \| \nabla \tilde{e}_{\mu}^{k} \|^2 ) \\
&   + C ( \| \Delta \tilde{\textbf{u}}^{k} \|^2+ \| \nabla \tilde{\textbf{u}}^{k} \|^2 )C_0^4 (\Delta t)^5 \\
 &  + C \|  \tilde{\phi}^{k} \|_{H^2}^2 C_0^4 (\Delta t)^5
 + C  \| \nabla  \tilde{\mu}^{k} \|^2 C_0^4 (\Delta t)^5  + C  (\Delta t)^3.
\endaligned
\end{equation}
Summing \eqref{e_error_final_utilde1} over $k$, $k=0,1,2,\ldots,n$, using boundedness estimates \eqref{e_error_estimate4}, \eqref{e_error1}, lemma \ref{lem_phi_H2_boundness} and applying the discrete Gronwall lemma \ref{lem: gronwall1} under the condition that $ \Delta t \leq \frac{ \lambda }{ 2 C_1 } $, we can arrive at
\begin{equation}\label{e_error_final_utilde2}
\aligned
& \| \tilde{e} _{\phi}^{n+1} \|^2 +  \| \nabla \tilde{e} _{\phi}^{n+1} \|^2 +  \Delta t \sum\limits_{k=0}^{n} \| \tilde{e}_{\mu}^{k+1} \|^2 \\
&+ \Delta t \sum\limits_{k=0}^{n} \| \nabla \tilde{e}_{\mu}^{k+1} \|^2 
+ \|  \nabla \tilde{e}_{\textbf{u}}^{n+1} \|^2  + \Delta t \sum\limits_{k=0}^{n} \| \Delta \tilde{e}_{\textbf{u}}^{k+1} \|^2  \\
\leq & C_2 \left(1+C_0^4 (\Delta t)^2 \right) (\Delta t)^2, \ \forall n\leq T/\Delta t ,
\endaligned
\end{equation}  
where $C_2$ is independent of $C_0$ and $\Delta t$.

Next we finish the induction process by establishing the estimates for $| 1- \xi^{n+1} |$.  Let $S_{ r }^{k+1}$ be the truncation error defined by
\begin{equation}\label{e_error_R_1}
\aligned
S_{ r }^{k+1}=\frac{\partial r (t^{k+1})}{\partial t}- \frac{ R (t^{k+1})-R(t^{k})}{\Delta t}=\frac{1}{\Delta t}\int_{t^k }^{t^{k+1}}(t^k-t)\frac{\partial^2 r }{\partial t^2}dt.
\endaligned
\end{equation}
Subtracting \eqref{e_model_r_Modify} at $t^{k+1}$ from \eqref{e_model_semi4}, we obtain
\begin{equation}\label{e_error_R_2}
\aligned
\frac{ \tilde{e} _{R}^{k+1} -  e_{R}^{k} }{ \Delta t } = & - \frac{ \tilde{R}^{n+1} }{ E( \tilde{ \phi} ^{k+1} , \tilde{\textbf{u}}^{k+1} ) +\kappa_0 }
\left(   M \| \nabla \tilde{ \mu}^{k+1} \|^2 + \nu \|\nabla  \textbf{u}^k \|^2 \right) \\
& + \frac{ r( t^{n+1} ) }{ E( \phi (t ^{n+1} ), \textbf{u}(t^{n+1} ) ) +\kappa_0 }
\left(   M \| \nabla \mu( t^{n+1} ) \|^2 + \nu \|\nabla  \textbf{u}(t^{n+1} )  \|^2 \right) + S_{ r }^{k+1} .
\endaligned
\end{equation}
Thus we can obtain an error equation corresponding to \eqref{e_lem_relation_final} 
\begin{equation}\label{e_error_R_4}
\aligned
e_{R}^k = \sigma^k \tilde{e}_{R}^k + ( 1- \sigma^k ) \left( E( \phi ^{k} , \textbf{u}^{k} ) -  E( \phi (t^{k} ), \textbf{u}( t^{k} ) ) \right) .
\endaligned
\end{equation} 
Plugging \eqref{e_error_R_4} into \eqref{e_error_R_2} leads to
\begin{equation}\label{e_error_R_5}
\aligned
\tilde{e} _{R}^{k+1} - \sigma^{k} \tilde{e}_{R}^k = &  ( 1- \sigma^k ) \left( E( \phi ^{k} , \textbf{u}^{k} ) -  E( \phi (t^{k} ), \textbf{u}( t^{k} ) ) \right) \\
& - \Delta t  \frac{ \tilde{R}^{k+1} }{ E( \tilde{ \phi} ^{k+1} , \tilde{\textbf{u}}^{k+1} ) +\kappa_0 }
\left(   M \| \nabla \tilde{ \mu}^{k+1} \|^2 + \nu \|\nabla  \textbf{u}^k \|^2 \right) \\
& + \Delta t  \frac{ r( t^{k+1} ) }{ E( \phi (t ^{k+1} ), \textbf{u}(t^{k+1} ) ) +\kappa_0 }
\left(   M \| \nabla \mu( t^{k+1} ) \|^2 + \nu \|\nabla  \textbf{u}(t^{k+1} )  \|^2 \right) \\
&  + \Delta t  S_{ r }^{k+1} .
\endaligned
\end{equation} 
%By using the triangle inequality, we have
%\begin{equation}\label{e_error_R_6}
%\aligned
%& | \tilde{e} _{R}^{k+1} |  - \sigma^k | \tilde{e}_{R}^k | \leq | \tilde{e} _{R}^{k+1} - \sigma^k  \tilde{e}_{R}^k | \\
%\leq & \Delta t | \frac{ \tilde{R}^{k+1} }{ E( \tilde{ \phi} ^{k+1} , \tilde{\textbf{u}}^{k+1} ) +\kappa_0 }
%\left(   M \| \nabla \tilde{ \mu}^{k+1} \|^2 + \nu \|\nabla  \textbf{u}^k \|^2 \right) -  \left(   M \| \nabla \mu( t^{k+1} ) \|^2 + \nu \|\nabla  \textbf{u}(t^{k+1} )  \|^2 \right) | \\
%& +  ( 1- \sigma^k )  | \left( E( \phi ^{k} , \textbf{u}^{k} ) -  E( \phi (t^{k} ), \textbf{u}( t^{k} ) ) \right) | \\
%& +  \Delta t | S_{ r }^{k+1}  | .
%\endaligned
%\end{equation} 
Next we continue the error estimates in the following two cases $R^{k-1} \geq E( \phi ^{k} , \textbf{u}^{k} ) +\kappa_0$ and $R^{k-1} < E( \phi ^{k} , \textbf{u}^{k} ) +\kappa_0$.

\textbf{Case I: } Using Lemma \ref{lem: relation_R and tildeR}, we have $ \sigma^{k} =0 $ under the condition that $R^{k-1} \geq E( \phi ^{k} , \textbf{u}^{k} ) +\kappa_0$, thus multiplying \eqref{e_error_R_5} with $ \tilde{e} _{R}^{k+1} $ and 
 using the triangle inequality, we have
\begin{equation}\label{e_error_R_6}
\aligned
&  | \tilde{e} _{R}^{k+1} |^2  
=  \left( E( \phi ^{k} , \textbf{u}^{k} ) -  E( \phi (t^{k} ), \textbf{u}( t^{k} ) ),  \tilde{e} _{R}^{k+1}  \right)  +  \Delta t (  S_{ r }^{k+1}  ,\tilde{e} _{R}^{k+1} )  + H( \tilde{R}^{k+1}, \tilde{ \mu}^{k+1},   \textbf{u}^k ),
\endaligned
\end{equation} 
where 
\begin{equation*}
\aligned
 H( \tilde{R}^{k+1},  \tilde{ \mu}^{k+1},  \textbf{u}^k ) = & \Delta t \left( \frac{ \tilde{R}^{k+1} }{ E( \tilde{ \phi} ^{k+1} , \tilde{\textbf{u}}^{k+1} ) +\kappa_0 }
\left(   M \| \nabla \tilde{ \mu}^{k+1} \|^2 + \nu \| \nabla  \textbf{u}^k \|^2 \right)  \right. \\ 
& \left. - \left(   M \| \nabla \mu( t^{k+1} ) \|^2 + \nu \|\nabla  \textbf{u}(t^{k+1} )  \|^2 \right),  \tilde{e} _{R}^{k+1}  \right) . 
\endaligned
\end{equation*} 

Using \eqref{e_error1} and \eqref{e_error_final_utilde2}, we have
\begin{equation}\label{e_error_R_7}
\aligned
& | E( \phi ^{k} , \textbf{u}^{k} ) -  E( \phi (t^{k} ), \textbf{u}( t^{k} ) )  | \\
\leq & | E( \phi ^{k} , \textbf{u}^{k} ) -  E(  \tilde{\phi} ^{k} , \tilde{ \textbf{u} } ^{k})  | + |   E(  \tilde{\phi} ^{k} , \tilde{ \textbf{u} } ^{k}) -  E( \phi (t^{k} ), \textbf{u}( t^{k} ) )  | \\
\leq & C K_1 ( \| \nabla \phi^{k} - \nabla \tilde{\phi} ^{k} \| + \| \phi^{k} -  \tilde{\phi} ^{k} \| + \| \textbf{u}^{k} - \tilde{ \textbf{u} } ^{k} \| ) \\
& + C  ( \|  \nabla \tilde{\phi} ^{k} - \nabla  \phi (t^{k} ) \| + \|  \tilde{\phi} ^{k} -  \phi (t^{k} ) \|
 + \| \tilde{ \textbf{u} } ^{k} -  \textbf{u} (t^{k} ) \| ) \\
& + \lambda \int_{\Omega} | F( \phi ^{k} ) - F( \phi (t^{k} ) ) | d \textbf{x} \\
\leq & C | 1- \eta_1^k | ( \| \tilde{\phi} ^{k} \|_{H^1} + \| \tilde{ \textbf{u} } ^{k} \| ) + C \Delta t \\
\leq & C  ( \| \tilde{\phi} ^{k} \|_{H^1} + \| \tilde{ \textbf{u} } ^{k} \| )  C_0^2 (\Delta t)^2 + C \Delta t .
\endaligned
\end{equation} 
Then the first term on the right hand side of \eqref{e_error_R_6} can be estimated as 
\begin{equation}\label{e_error_R_8}
\aligned
& \left( E( \phi ^{k} , \textbf{u}^{k} ) -  E( \phi (t^{k} ), \textbf{u}( t^{k} ) ),  \tilde{e} _{R}^{k+1}  \right) \\
& \ \ \ \ \ \ \ \ \ \ \ 
\leq  C  ( \| \tilde{\phi} ^{k} \|_{H^1} ^2 + \| \tilde{ \textbf{u} } ^{k} \|^2 )  C_0^4 (\Delta t)^4 + C ( \Delta t )^2 + \frac{ 1 }{ 4 } | \tilde{e} _{R}^{k+1} |^2  .
\endaligned
\end{equation} 
Since
  \begin{equation}\label{e_error_R_9}
\aligned
&   \Delta t | \frac{ \tilde{R}^{k+1} }{ E( \tilde{ \phi} ^{k+1} , \tilde{\textbf{u}}^{k+1} ) +\kappa_0 }
\left(   M \| \nabla \tilde{ \mu}^{k+1} \|^2 + \nu \| \nabla  \textbf{u}^k \|^2 \right) 
 - \left(   M \| \nabla \mu( t^{k+1} ) \|^2 + \nu \|\nabla  \textbf{u}(t^{k+1} )  \|^2 \right) | \\
=	 &  \Delta t  \left|  \frac{ \tilde{R}^{k+1} }{ E( \tilde{ \phi} ^{k+1} , \tilde{\textbf{u}}^{k+1} ) +\kappa_0 } - 
\frac{ r( t^{k+1} ) }{ E( \phi (t ^{k+1} ), \textbf{u}(t^{k+1} ) ) +\kappa_0 } \right|  \left(   M \| \nabla \mu( t^{k+1} ) \|^2 + \nu \| \nabla   \textbf{u}(t^{k+1} ) \|^2 \right) \\
& +  M  \Delta t  \frac{ \tilde{R}^{k+1} }{ E( \tilde{ \phi} ^{k+1} , \tilde{\textbf{u}}^{k+1} ) +\kappa_0 } ( \| \nabla \tilde{ \mu}^{k+1} \|^2 - \| \nabla \mu( t^{k+1} ) \|^2 ) \\
& + \nu \Delta t   \frac{ \tilde{R}^{k+1} }{ E( \tilde{ \phi} ^{k+1} , \tilde{\textbf{u}}^{k+1} ) +\kappa_0 } ( \| \nabla  \textbf{u}^k \|^2 -  \|\nabla  \textbf{u}(t^{k+1} )  \|^2 ) \\
\leq & C  \Delta t | \tilde{e}_R^{k+1} | + C \Delta t ( \| \nabla \tilde{ \mu}^{k+1} \| +1 ) \|  \nabla \tilde{e}_{\mu}^{k+1} \|  \\
& + C  \Delta t \| \tilde{e}_{ \textbf{u}}^{k} \| + C \Delta t.
\endaligned
\end{equation} 
Thus using Cauchy-Schwarz inequality, the last term on the right hand side of \eqref{e_error_R_6} can be bounded by
  \begin{equation}\label{e_error_R_10}
\aligned
  H( \tilde{R}^{k+1},  & \tilde{ \mu}^{k+1},   \textbf{u}^k ) \\
\leq &  C  \Delta t | \tilde{e}_R^{k+1} |^2 + \frac{1}{4 K_2 } \Delta t  \| \nabla \tilde{ \mu}^{k+1} \|^2 | \tilde{e}_R^{k+1} |^2\\
&  + C \Delta t \|  \nabla \tilde{e}_{\mu}^{k+1} \|^2 + C  \Delta t \| \tilde{e}_{ \textbf{u}}^{k} \|^2 + C (\Delta t)^2 .
\endaligned
\end{equation} 
Combining \eqref{e_error_R_6} with \eqref{e_error_R_7}-\eqref{e_error_R_10}, we have
\begin{equation}\label{e_error_R_11}
\aligned
& | \tilde{e} _{R}^{k+1} |^2  
\leq  C  ( \| \tilde{\phi} ^{k} \|_{H^1} ^2 + \| \tilde{ \textbf{u} } ^{k} \|^2 )  C_0^4 (\Delta t)^4  + \frac{ 1 }{ 4 } | \tilde{e} _{R}^{k+1} |^2 \\
& + C  \Delta t | \tilde{e}_R^{k+1} |^2 +  \frac{1}{4 K_2 } \Delta t  \| \nabla \tilde{ \mu}^{k+1} \|^2 | \tilde{e}_R^{k+1} |^2\\
&  + C \Delta t \|  \nabla \tilde{e}_{\mu}^{k+1} \|^2 + C  \Delta t \| \tilde{e}_{ \textbf{u}}^{k} \|^2 + C (\Delta t)^2 .
\endaligned
\end{equation} 
By substituting $n$ for $k$ and adding some positive terms on the right hand side of \eqref{e_error_R_11}, we get
\begin{equation}\label{e_error_R_12}
\aligned
& | \tilde{e} _{R}^{n+1} |^2  
\leq  C  ( \| \tilde{\phi} ^{n} \|_{H^1} ^2 + \| \tilde{ \textbf{u} } ^{n} \|^2 )  C_0^4 (\Delta t)^4  + \frac{ 1 }{ 4 } | \tilde{e} _{R}^{n+1} |^2 \\
& + C_3  \Delta t \sum\limits_{k=0}^{n} | \tilde{e}_R^{k+1} |^2 +  \frac{1}{4 K_2 } | \tilde{e}_R^{n+1} |^2 \Delta t \sum\limits_{k=0}^{n} \| \nabla \tilde{ \mu}^{k+1} \|^2  \\
&  + C \Delta t \sum\limits_{k=0}^{n} \|   \nabla \tilde{e}_{\mu}^{k+1} \|^2 + C  \Delta t \sum\limits_{k=0}^{n} \| \tilde{e}_{ \textbf{u}}^{k} \|^2 + C (\Delta t)^2 .
\endaligned
\end{equation} 

\textbf{Case II: } In the following as specified in \eqref{e_error_sigmak}, we set $ \sigma^{k} = 1 - \frac{  \tilde{R}^{k} \left(   M \| \nabla \tilde{ \mu}^{k} \|^2 + \nu \|\nabla  \textbf{u}^{k-1} \|^2 \right) } {  (E( \tilde{ \phi} ^{k} , \tilde{\textbf{u}}^{k} ) +\kappa_0 ) ( E( \phi ^{k} , \textbf{u}^{k} ) +\kappa_0 -  \tilde{R}^{k} )  }  ( \Delta t )^2 $  under the condition that $R^{k-1} < E( \phi ^{k} , \textbf{u}^{k} ) +\kappa_0$, thus using \eqref{e_error_R_5}, we have
\begin{equation}\label{e_error_R_14}
\aligned
\tilde{e} _{R}^{k+1} -  \tilde{e}_{R}^k = &  - \Delta t  \frac{ \tilde{R}^{k+1} }{ E( \tilde{ \phi} ^{k+1} , \tilde{\textbf{u}}^{k+1} ) +\kappa_0 }
\left(   M \| \nabla \tilde{ \mu}^{k+1} \|^2 + \nu \|\nabla  \textbf{u}^k \|^2 \right) \\
& + \Delta t  \frac{ r( t^{k+1} ) }{ E( \phi (t ^{k+1} ), \textbf{u}(t^{k+1} ) ) +\kappa_0 }
\left(   M \| \nabla \mu( t^{k+1} ) \|^2 + \nu \|\nabla  \textbf{u}(t^{k+1} )  \|^2 \right) \\
&  + \Delta t  S_{ r }^{k+1} + \frac{  \tilde{R}^{k} \left(   M \| \nabla \tilde{ \mu}^{k} \|^2 + \nu \|\nabla  \textbf{u}^{k-1} \|^2 \right) } {  E( \tilde{ \phi} ^{k} , \tilde{\textbf{u}}^{k} ) +\kappa_0    }  ( \Delta t )^2  .
\endaligned
\end{equation} 
Multiplying \eqref{e_error_R_14} with $ \tilde{e} _{R}^{k+1} $ and using the triangle inequality, we have
\begin{equation}\label{e_error_R_15}
\aligned
& \frac{1}{2} | \tilde{e} _{R}^{k+1} |^2  - \frac{1}{2} | \tilde{e} _{R}^{k} |^2 + \frac{1}{2} |  \tilde{e} _{R}^{k+1}- \tilde{e} _{R}^{k} |^2 \\
= &  \Delta t \left(  \frac{  \tilde{R}^{k} \left(   M \| \nabla \tilde{ \mu}^{k} \|^2 + \nu \|\nabla  \textbf{u}^{k-1} \|^2 \right) } {  E( \tilde{ \phi} ^{k} , \tilde{\textbf{u}}^{k} ) +\kappa_0    }   \Delta t  ,  \tilde{e} _{R}^{k+1}  \right)  +  \Delta t (  S_{ r }^{k+1}  ,\tilde{e} _{R}^{k+1} )  \\
& + H( \tilde{R}^{k+1}, \tilde{ \mu}^{k+1},   \textbf{u}^k ),
\endaligned
\end{equation} 
Taking the sum of \eqref{e_error_R_15} from $q+1$ to $n$ with $ \sigma^{q} =0$ and using \eqref{e_error_R_6} result in
\begin{equation}\label{e_error_R_16}
\aligned
& \frac{1}{2} | \tilde{e} _{R}^{n+1} |^2 \leq  \Delta t  \sum\limits_{k=q+1}^{n} \left(  \frac{  \tilde{R}^{k} \left(   M \| \nabla \tilde{ \mu}^{k} \|^2 + \nu \|\nabla  \textbf{u}^{k-1} \|^2 \right) } {  E( \tilde{ \phi} ^{k} , \tilde{\textbf{u}}^{k} ) +\kappa_0    }   \Delta t  ,  \tilde{e} _{R}^{k+1}  \right) \\
& +  \frac{1}{2} \left( E( \phi ^{q} , \textbf{u}^{q} ) -  E( \phi (t^{q} ), \textbf{u}( t^{q} ) ),  \tilde{e} _{R}^{q+1}  \right)   \\
 & +  \frac{1}{2} \Delta t (  S_{ r }^{q+1}  ,\tilde{e} _{R}^{q+1} ) +  \Delta t  \sum\limits_{k=q+1}^{n} (  S_{ r }^{k+1}  ,\tilde{e} _{R}^{k+1} )  \\
&  + \frac{1}{2}  H( \tilde{R}^{q+1}, \tilde{ \mu}^{q+1},   \textbf{u}^q ) +  \sum\limits_{k=q+1}^{n} H( \tilde{R}^{k+1}, \tilde{ \mu}^{k+1},   \textbf{u}^k ) .
\endaligned
\end{equation} 
Using \eqref{e_error1} and lemma \ref{lem_phi_H2_boundness}, the first term on the right hand side of \eqref{e_error_R_16} can be estimated as
\begin{equation}\label{e_error_R_17}
\aligned
&  \Delta t  \sum\limits_{k=q+1}^{n} \left(  \frac{  \tilde{R}^{k} \left(   M \| \nabla \tilde{ \mu}^{k} \|^2 + \nu \|\nabla  \textbf{u}^{k-1} \|^2 \right) } {  E( \tilde{ \phi} ^{k} , \tilde{\textbf{u}}^{k} ) +\kappa_0    }   \Delta t  ,  \tilde{e} _{R}^{k+1}  \right) \\
\leq & C \Delta t \sum\limits_{k=q+1}^{n} | \tilde{e} _{R}^{k+1} |^2 +C( K_1 + K_2 ) ( \Delta t )^2 .
\endaligned
\end{equation} 
The other terms on the right hand side of \eqref{e_error_R_16} can be estimated by using exactly the same procedure as above in Case I , we can obtain that 
\begin{equation}\label{e_error_R_18}
\aligned
 \frac{1}{2} | \tilde{e} _{R}^{n+1} |^2
 \leq & C_4 \Delta t \sum\limits_{k=q+1}^{n} | \tilde{e} _{R}^{k+1} |^2 + C \Delta t \sum\limits_{k=0}^{n} \|   \nabla \tilde{e}_{\mu}^{k+1} \|^2 \\
&+   C  \Delta t \sum\limits_{k=0}^{n} \| \tilde{e}_{ \textbf{u}}^{k} \|^2 + C  ( \| \tilde{\phi} ^{q} \|_{H^1} ^2 + \| \tilde{ \textbf{u} } ^{q} \|^2 )  C_0^4 (\Delta t)^4 + C (\Delta t)^2 .
\endaligned
\end{equation} 
Applying the discrete Gronwall lemma \ref{lem: gronwall1} under the condition that $ \Delta t \leq \min\{ \frac{ 1 }{ 4 C_3 } , \frac{ 1 }{ 4 C_4 } \} $ and using \eqref{e_error_final_utilde2}, we can arrive at 
\begin{equation}\label{e_error_R_19}
\aligned
& | \tilde{e} _{R}^{n+1} |^2  
\leq  C_5 \left(1+C_0^4 (\Delta t)^2 \right) (\Delta t)^2, \ \forall n\leq T/\Delta t .
\endaligned
\end{equation} 
 Next we  finish the induction process as follows. Recalling \eqref{e_model_semi4}, we have
\begin{equation}\label{e_error_R_20}
\aligned
 | 1- \xi^{n+1} | =  & | \frac{R(t^{n+1} )}{ E(\phi(t^{n+1}), \textbf{u}(t^{n+1}) ) +\kappa_0 }  -\frac{ \tilde{R}^{n+1} }{ E( \tilde{ \phi} ^{n+1} , \tilde{\textbf{u}}^{n+1} ) +\kappa_0 }| \\
 \leq & C(  | \tilde{e}_{R}^{n+1} | + \| \tilde{e}_{\textbf{u}}^{n+1} \| + \| \tilde{e}_{\phi}^{n+1} \| +   \| \nabla \tilde{e}_{\phi}^{n+1} \|  ) \\
  \leq & C_6 \Delta t \sqrt{  1+C_0^4 (\Delta t)^2 }, \ \  \ \forall n\leq T/\Delta t .
\endaligned
\end{equation}
where $C_6$ is independent of $C_0$ and $\Delta t$.

Let $C_0=\max\{ 2 C_6, 2 \sqrt{C_4 }, 2\sqrt{ C_3 },  (2C_2)^{ \frac{3}{2} } , \sqrt{ \frac{2C_1}{ \lambda } }, 4 \}$ and $\Delta t \leq \frac{1}{ 1+C_0^2 }$,  we can obtain
\begin{equation}\label{e_error_R_21}
\aligned
 C_6 \sqrt{  1+C_0^4 (\Delta t)^2 } \leq C_6 (1+C_0^2 \Delta t ) \leq C_0.
\endaligned
\end{equation}
Then combining \eqref{e_error_R_20} with \eqref{e_error_R_21} results in
\begin{equation}\label{e_error_R_22}
\aligned
 | 1- \xi^{n+1} | 
  \leq &  C_0 \Delta t,  \ \forall n\leq T/\Delta t.
\endaligned
\end{equation}
Recalling \eqref{e_error_final_utilde2}, we have
\begin{equation}\label{e_error_R_23}
\aligned
& \| \tilde{e}_{\textbf{u}}^{k}  \|_{H^2} + \| \tilde{e}_{\mu}^{k}  \|_{H^1} 
\leq  C_2 \left(1+C_0^4 (\Delta t)^2 \right) (\Delta t)^{1/2} \leq  (\Delta t)^{1/6},  
\endaligned
\end{equation}  
which completes the induction process \eqref{e_error_estimate1} and \eqref{e_error_estimate2}.

Then combining \eqref{e_error_final_utilde2} with \eqref{e_velocity8}, \eqref{e_velocity12} and \eqref{e_velocity13}, we obtain 
\begin{equation}\label{e_error_final}
\aligned
& \| e _{\phi}^{n+1} \|^2 +  \| \nabla e _{\phi}^{n+1} \|^2 +  \Delta t \sum\limits_{k=0}^{n} \| e _{\mu}^{k+1} \|^2 + \Delta t \sum\limits_{k=0}^{n} \| \nabla e_{\mu}^{k+1} \|^2  \\
& + \|  \nabla e_{\textbf{u}}^{n+1} \|^2  + \Delta t \sum\limits_{k=0}^{n} \| \Delta e_{\textbf{u}}^{k+1} \|^2  
\leq  C (\Delta t)^2, \ \forall n\leq T/\Delta t .
\endaligned
\end{equation}  

Now it remains to  estimate the  pressure error. Recalling \eqref{e_velocity5}, we can estimate \eqref{e_velocity4} into the following:
\begin{equation}\label{e_error_final_p}
\aligned
 \Delta t \sum\limits_{k=0}^{n} \| \nabla e_p^{k+1} \|^2 \leq &  C \Delta t \sum\limits_{k=0}^{n}  \| \Delta \tilde{e}_\textbf{u} ^{k+1} \|^2 + C \Delta t \sum\limits_{k=0}^{n}  \| \nabla  \tilde{e}_\textbf{u} ^{k+1} \|^2 \\
 & +  C \Delta t \sum\limits_{k=0}^{n}  \| \nabla e_{ \textbf{u} }^{k+1} \|^2 \| \nabla \textbf{u}(t^{k+1} ) \|  \| \nabla \textbf{u}(t^{k+1})  \|_{H^1} \\
 & + C \Delta t \sum\limits_{k=0}^{n} \Delta t \| \nabla \textbf{u}^{k+1} \|^2 \| \nabla e_{ \textbf{u} }^{k+1}  \|  \| \nabla e_{ \textbf{u} }^{k+1}  \|_{H^1} \\
& + C \Delta t \sum\limits_{k=0}^{n} \Delta t \| \nabla e_{ \mu }^{n+1} \|^2 \|  \nabla  \phi ^{n+1}  \|  \|  \nabla  \phi ^{n+1}  \|_{H^1} \\ 
 & + C \Delta t \sum\limits_{k=0}^{n} \Delta t \| \nabla \mu(t^{n+1} ) \|^2 \|  \nabla e_{ \phi}^{n+1}  \|  \| \nabla e_{ \phi}^{n+1} \|_{H^1} \\
\leq & C (\Delta t)^2,  \ \forall n\leq T/\Delta t ,
\endaligned
\end{equation} 
which completes the final results.

\end{proof}

% ================================================================

 \section{Numerical experiments}
In this section, we provide some numerical experiments to verify our theoretical results of the constructed high-order EOP-GSAV schemes for the Cahn-Hilliard-Navier-Stokes model.

\subsection{Convergence Tests}
We first verify the accuracy of the proposed numerical schemes.
We choose the coefficients
\begin{equation}
\lambda=1,\quad M=1\times 10^{-3},\quad \epsilon=1,
\quad  \gamma=0, \quad  \nu=0.05,\quad C_0=1,
\end{equation}
and solve \eqref{e_model}
with right hand sides chosen so that the exact solution is
\begin{equation}
	\begin{aligned}
		\phi(x,y,t)&=\cos(t)\cos(\pi x)\cos(\pi y),\\
		\textbf{u}(x,y,t)&=\pi \sin(t)(\sin^2(\pi x)\sin(2\pi y),-\sin(2\pi x)\sin^2(\pi y))^T,\\
		p(x,y,t)&=\sin(t)\cos(\pi x)\sin(\pi y).
	\end{aligned}
\end{equation}

We set $\Omega= [-1, 1]^2$ and use $50\times50$ modes to discretize the space variables, so the spatial discretization
error is negligible compared to the time discretization.  In Figures \ref{error1a}-\ref{error1c}, we list the errors between the numerical solution and the exact solution at $T=0.2$. We observe that all schemes achieve the expected accuracy in time, which is consistent with the error analysis in Theorem \ref{thm: error_estimate_final}.

\begin{figure}[htp]
	\centering
	\subfigure[BDF1 vs. errors]{
		\begin{minipage}[c]{0.3\textwidth}
			\includegraphics[width=1\textwidth]{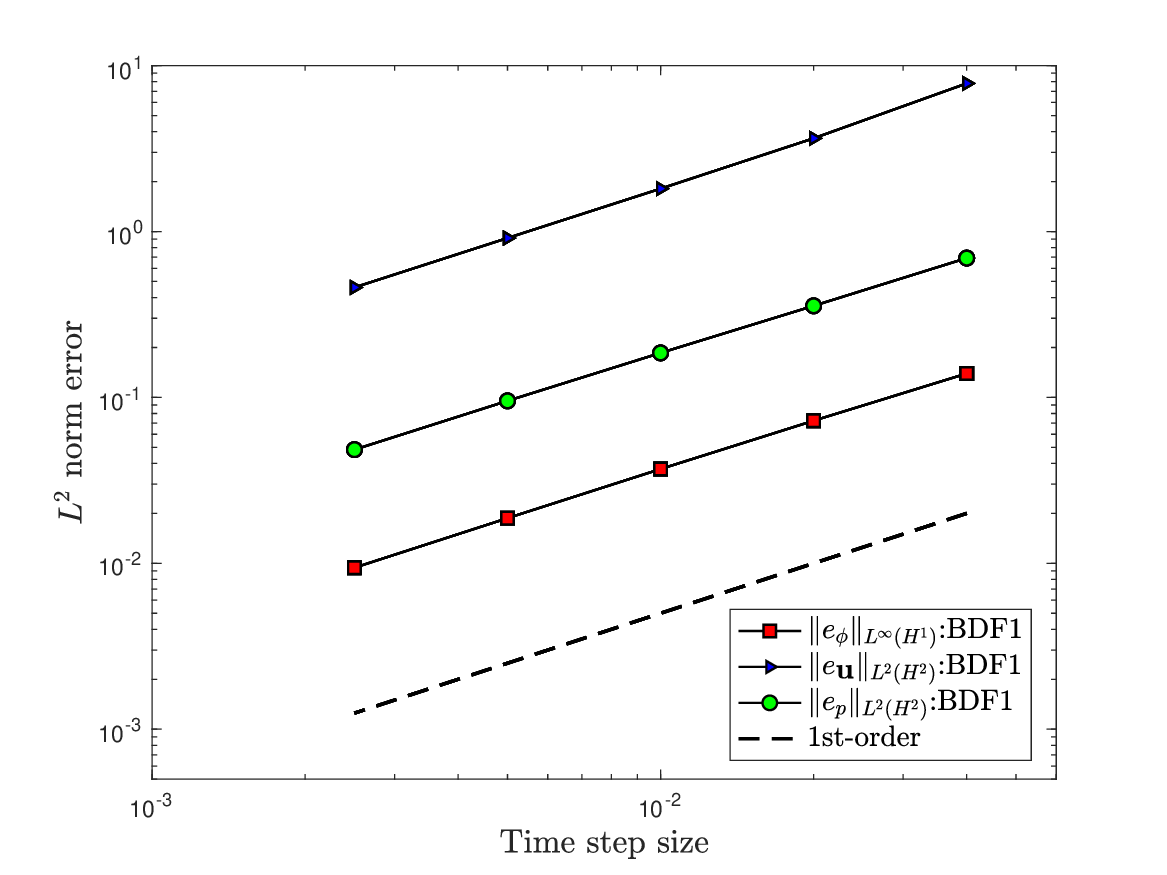}\label{error1a}
		\end{minipage}
	}
	\subfigure[BDF2 vs. errors]{
		\begin{minipage}[c]{0.3\textwidth}
			\includegraphics[width=1\textwidth]{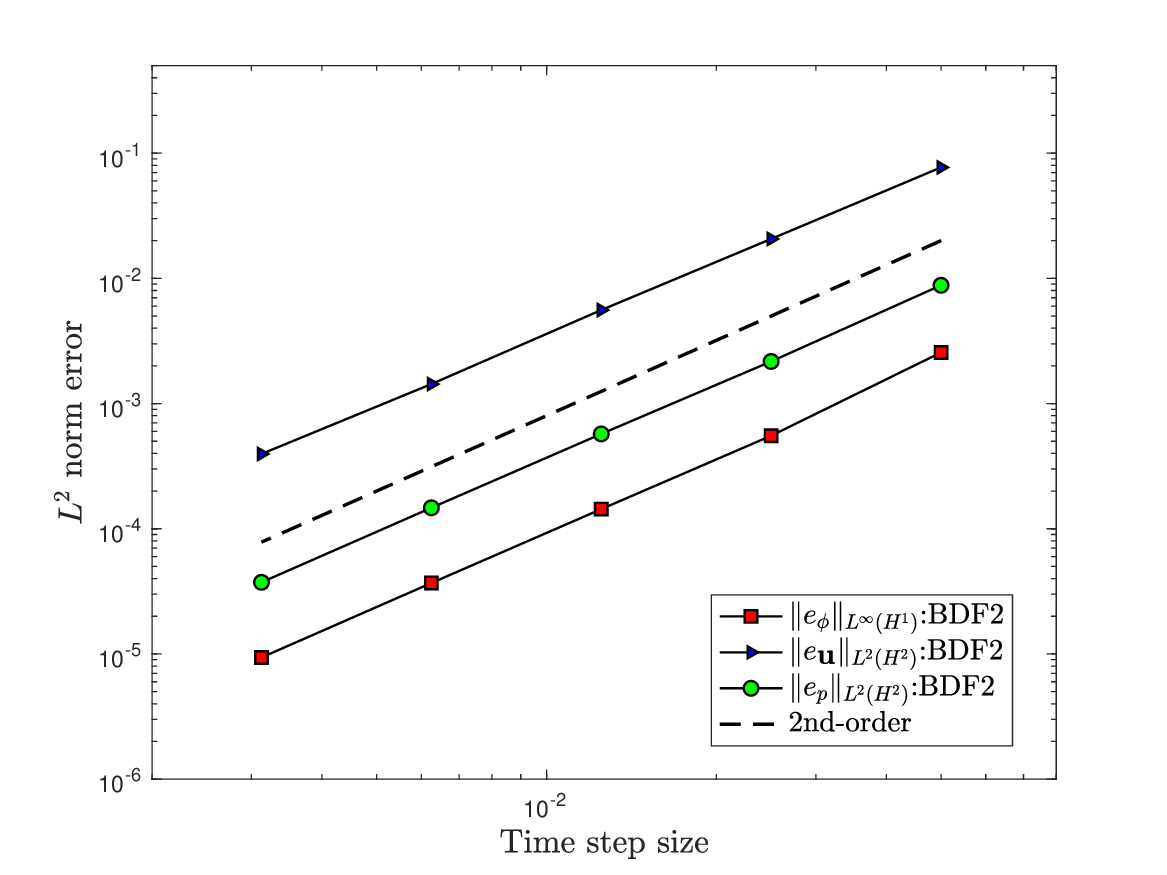}\label{error1b}
		\end{minipage}
	}
	%\rule{15cm}{0.05em}\\
	\subfigure[BDF3 vs. errors]{
		\begin{minipage}[c]{0.3\textwidth}
			\includegraphics[width=1\textwidth]{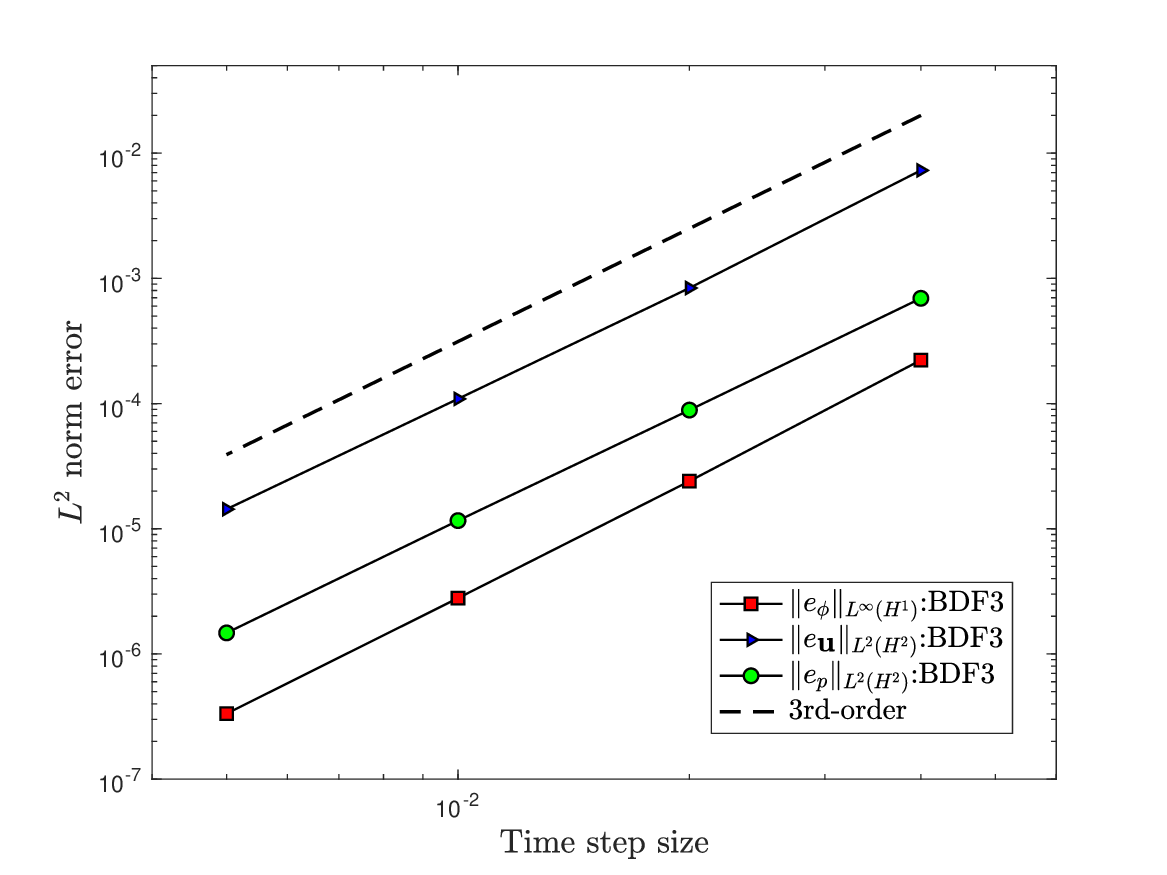}\label{error1c}
		\end{minipage}
	}
	{\caption{Numerical convergence rate of the first- to third-order schemes.}\label{error1}}
\end{figure}

\subsection{Shape relaxation}\label{Shape relaxation}
In this simulation, we set $\Omega= [0, 1]^2$ and use $128\times128$ modes to discretize the space variables. The parameters are
\begin{equation}
	\begin{aligned}
		\Delta t=&5\times 10^{-4},\quad \lambda=1\times 10^{-2},
		\quad M=1\times 10^{-3},\quad \epsilon=1\times 10^{-2},
		\\
		\gamma= &2\times 10^4,
		\quad 
		\nu=1,  \quad  \chi=6\times10, \quad
		C_0=10^5.
	\end{aligned}
\end{equation}
In Figure \ref{figure9}, we depict the dynamic process of shape relaxation towards a disk, considering various initial values. It can be clearly observed that the energy dissipation law holds for Cases 1 and 2 in Figure \ref{energy_evolu}. Furthermore, we can also find that the more vertices, the faster the evolution by comparing the energy evolutions of two cases.

 \begin{figure}[htp]
	\centering 
	%\vspace{-0.5cm}
	Case 1
	\subfigure[$T$=0]{
		\begin{minipage}[c]{0.21\textwidth}
			%	\vspace{-0.15cm}
			\includegraphics[width=1.2\textwidth]{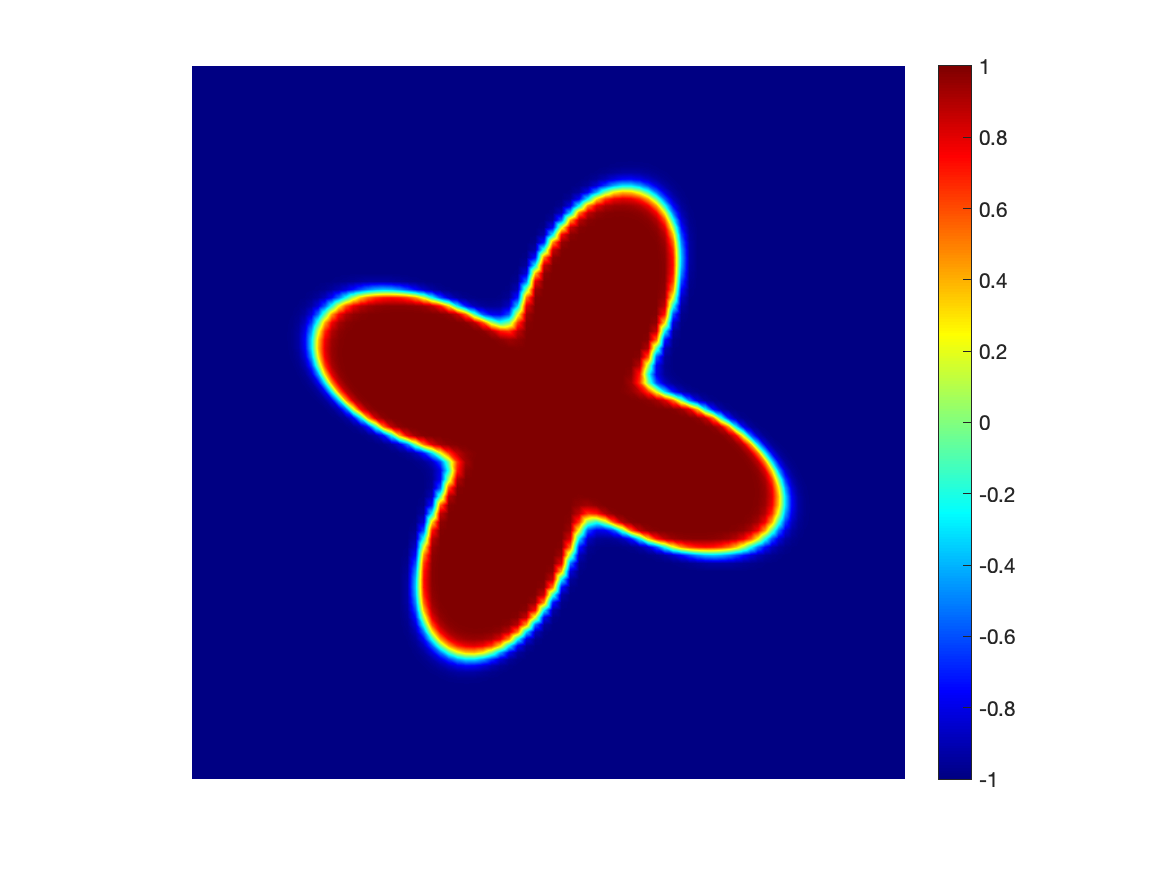}
			%	\vspace{-1cm}
		\end{minipage}
	}
	\subfigure[$T$=0.2]{
		\begin{minipage}[c]{0.21\textwidth}
			%	\vspace{-0.15cm}
			\includegraphics[width=1.2\textwidth]{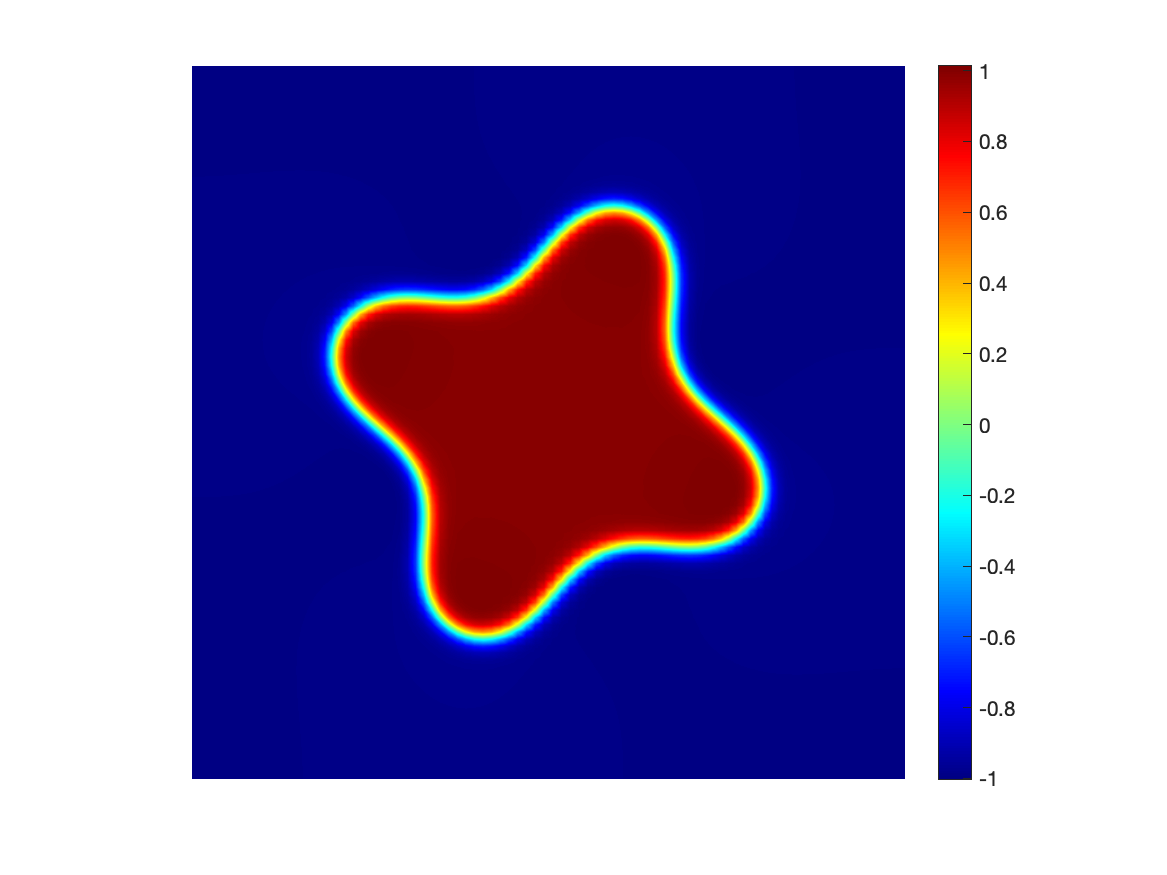}
			%	\vspace{-1cm}
		\end{minipage}
	}
	\subfigure[$T$=0.5]{
		\begin{minipage}[c]{0.21\textwidth}
			%	\vspace{-0.15cm}
			\includegraphics[width=1.2\textwidth]{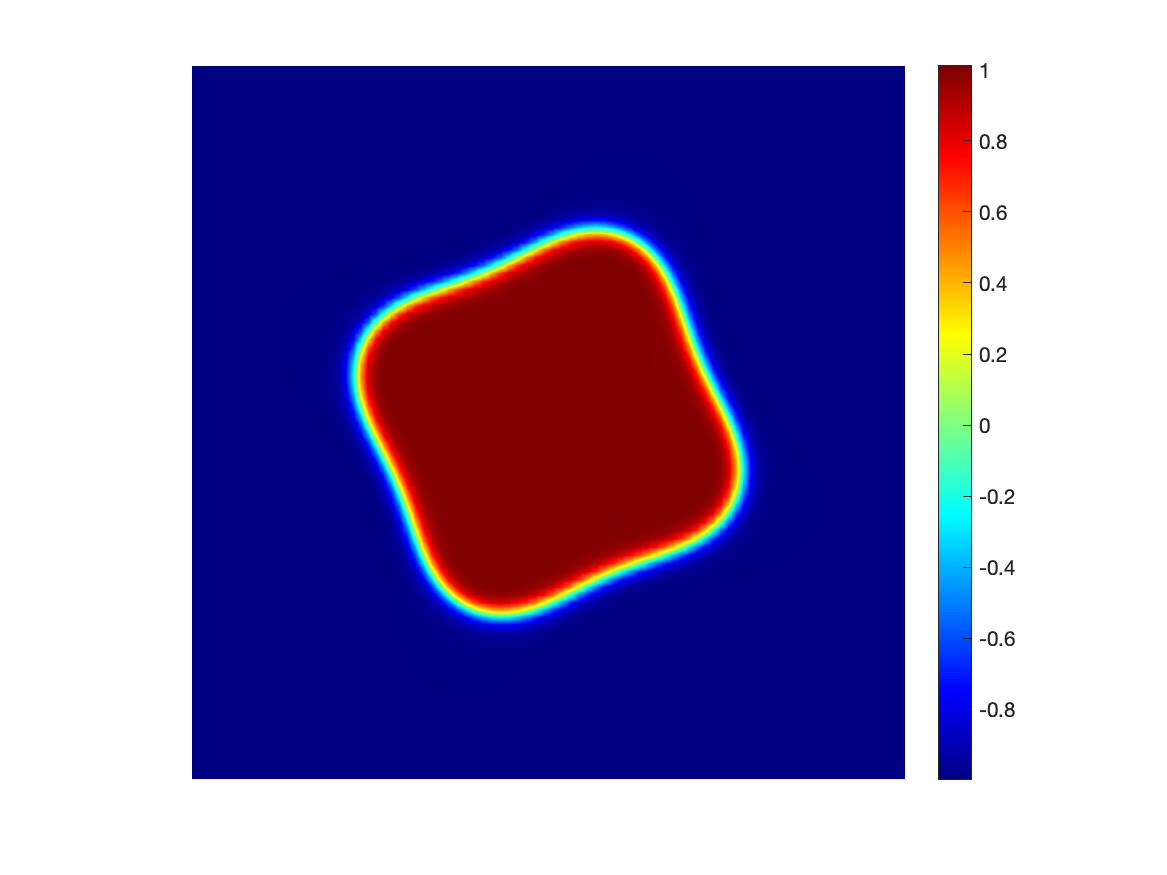}
			%	\vspace{-1cm}
		\end{minipage}
	}
	%\rule{15cm}{0.05em}\\
	\subfigure[$T$=1.5]{
		\begin{minipage}[c]{0.21\textwidth}
			\includegraphics[width=1.2\textwidth]{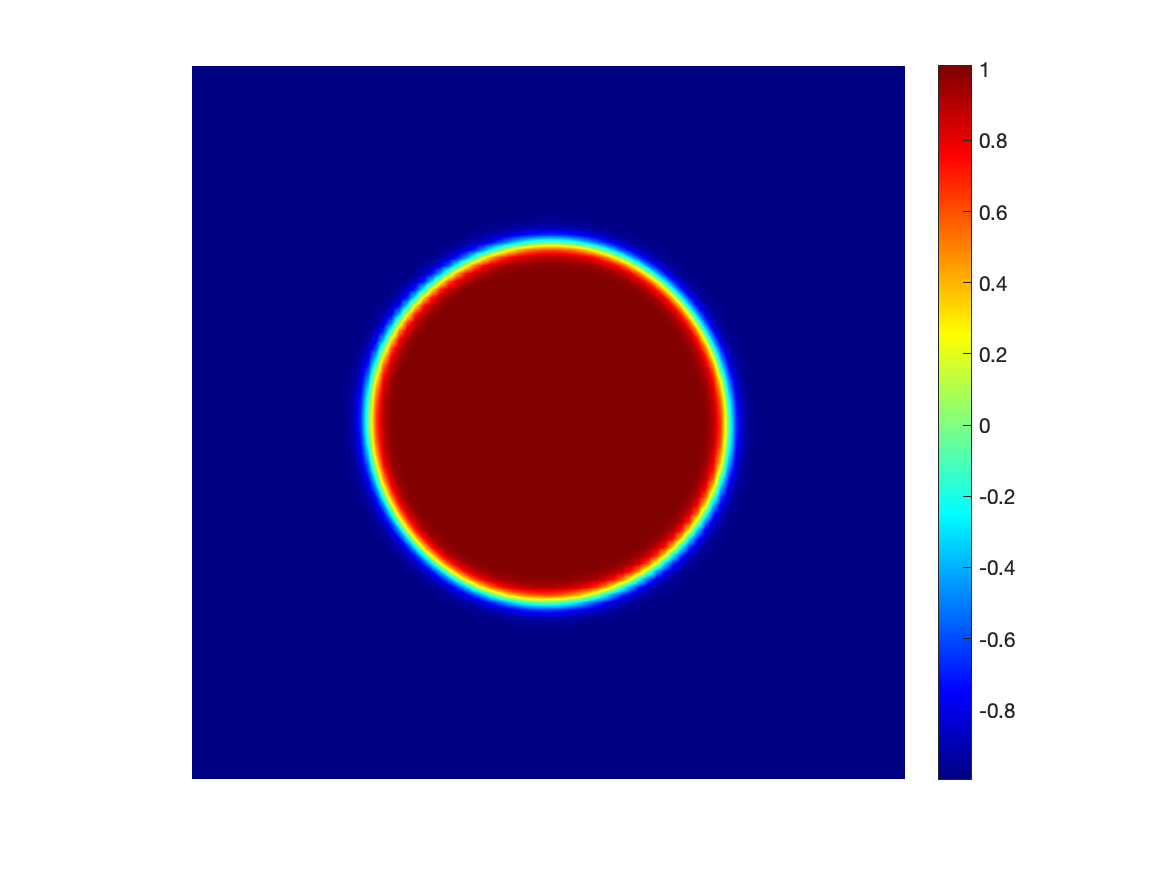}
			%	\vspace{-1cm}
		\end{minipage}
	}
	Case 2
	\subfigure[$T$=0]{
	\begin{minipage}[c]{0.21\textwidth}
		\includegraphics[width=1.2\textwidth]{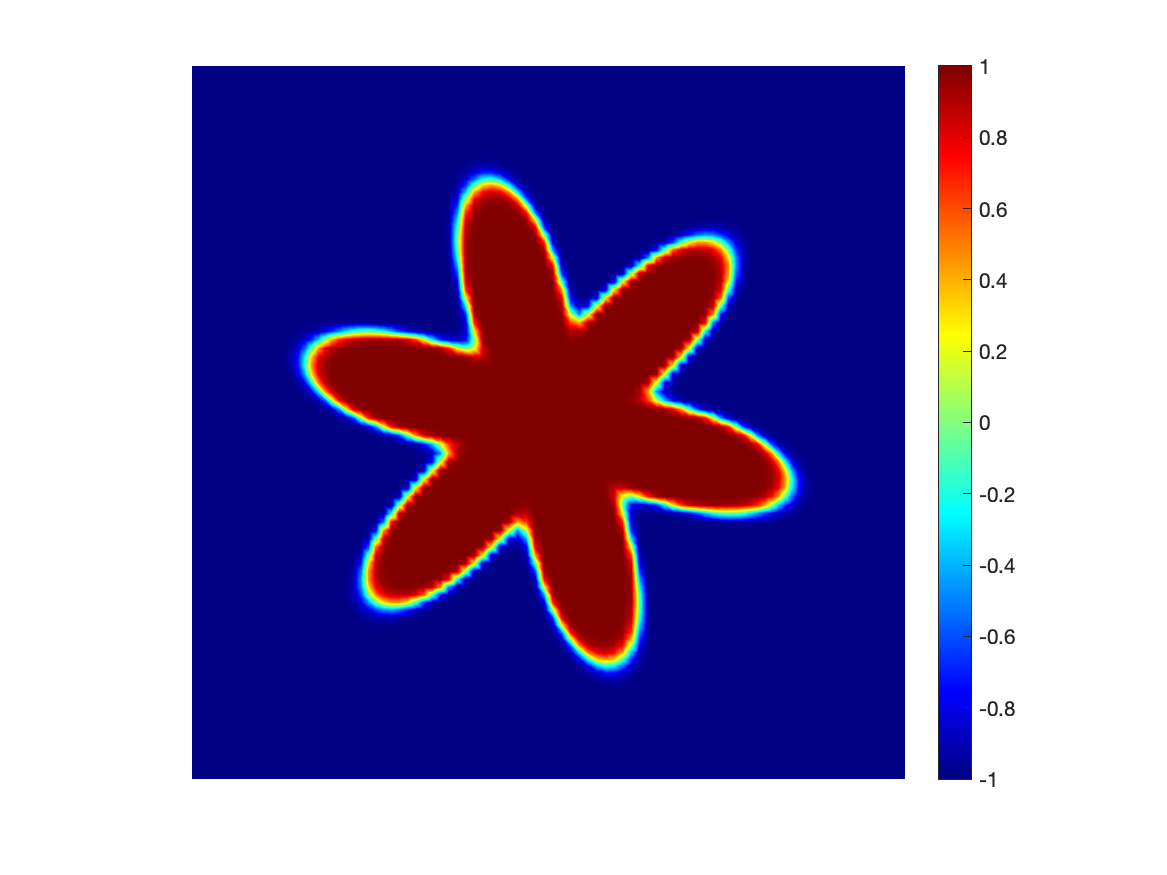}
		%	\vspace{-1cm}
	\end{minipage}
}
	\subfigure[$T$=0.2]{
		\begin{minipage}[c]{0.21\textwidth}
			\includegraphics[width=1.2\textwidth]{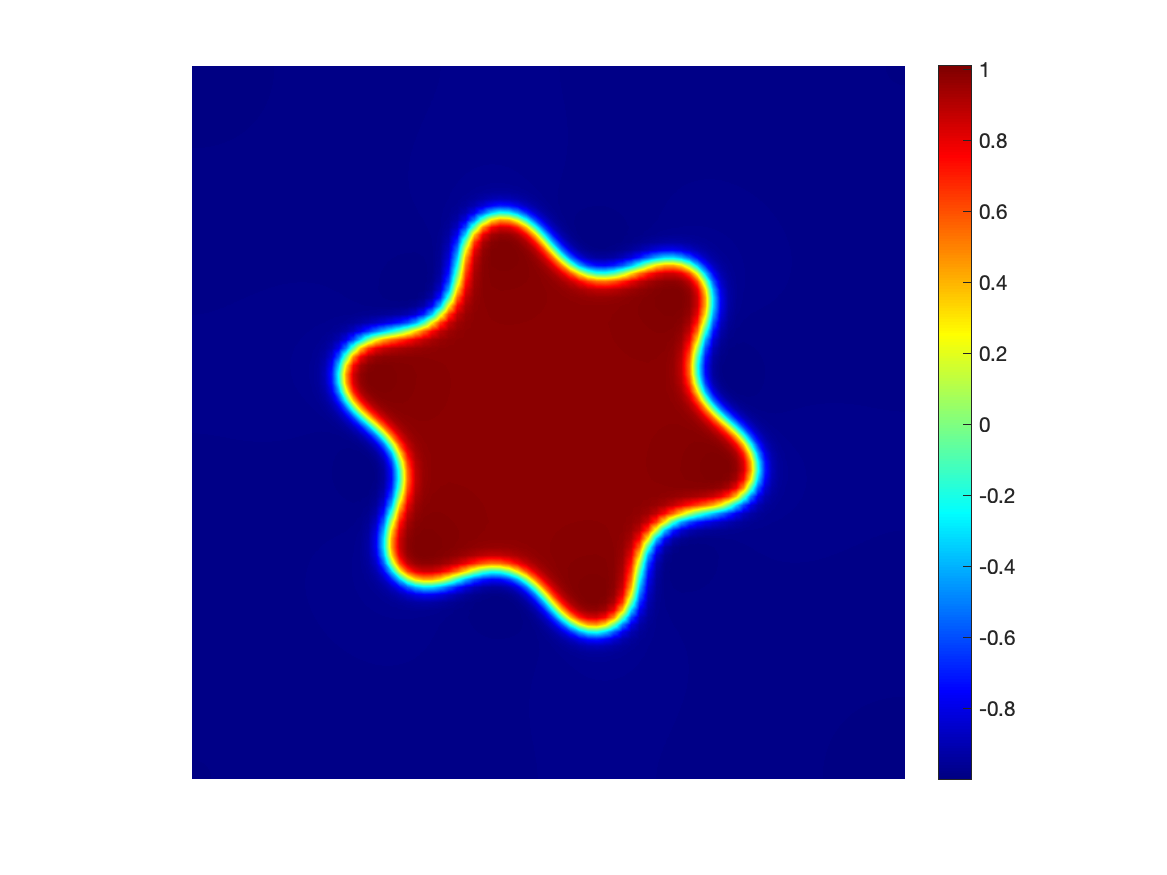}
			%	\vspace{-1cm}
		\end{minipage}
	}
	\subfigure[$T$=0.5]{
	\begin{minipage}[c]{0.21\textwidth}
		\includegraphics[width=1.2\textwidth]{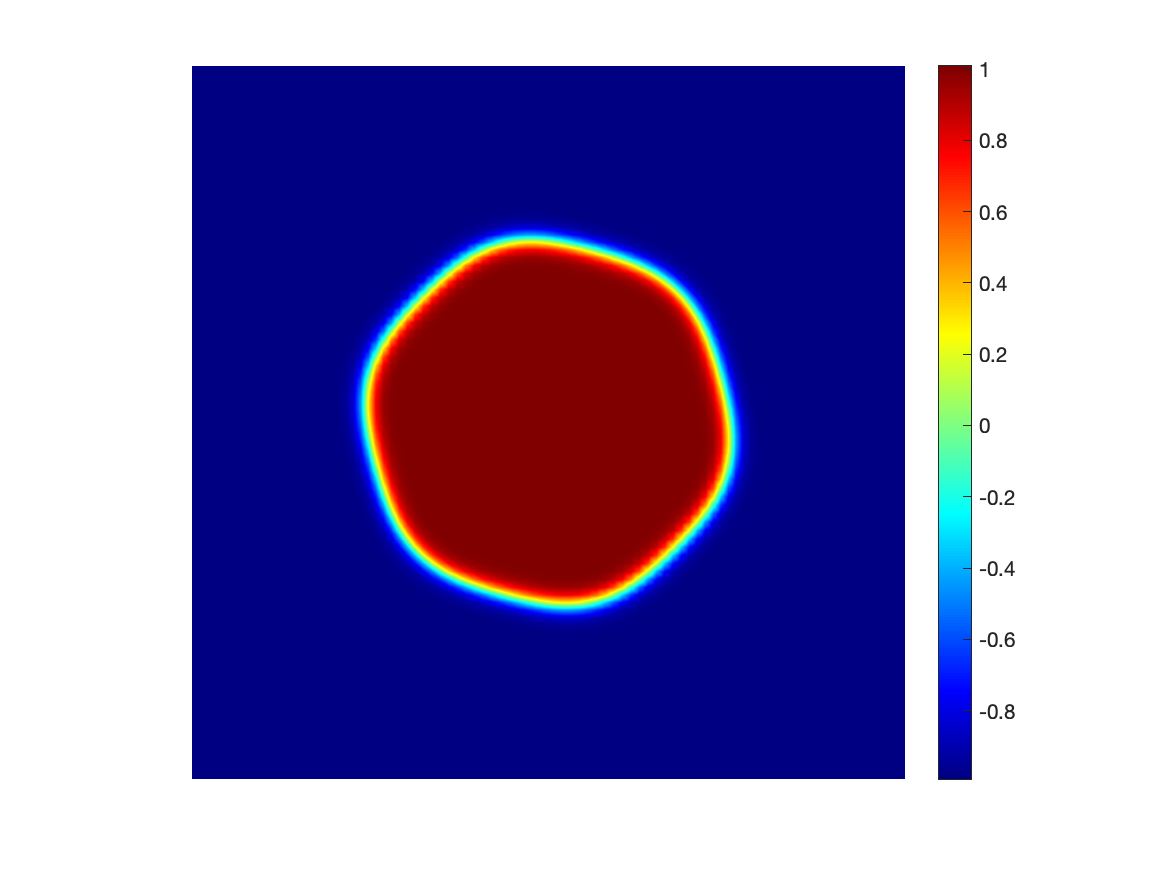}
		%	\vspace{-1cm}
	\end{minipage}
}
	\subfigure[$T$=1.5]{
		\begin{minipage}[c]{0.21\textwidth}
			\includegraphics[width=1.2\textwidth]{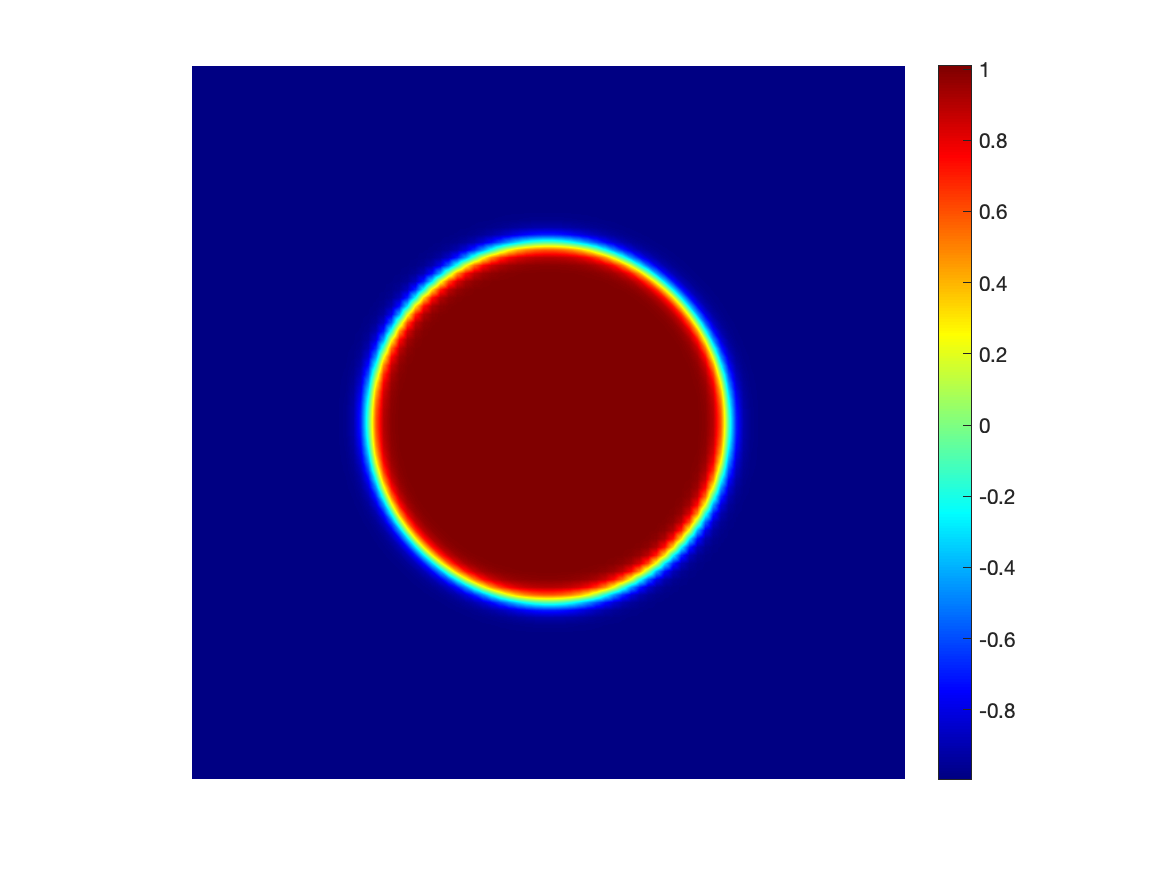}
			%	\vspace{-1cm}
		\end{minipage}
	}
		\vspace{-0.5cm}
	{\caption{Snapshots of the phase function $\phi$ at different $T$.} \label{figure9}}
\end{figure}

	\begin{figure}[htp]
	\centering
	\subfigure[]{
	\begin{minipage}[c]{0.45\textwidth}
		\includegraphics[width=1\textwidth]{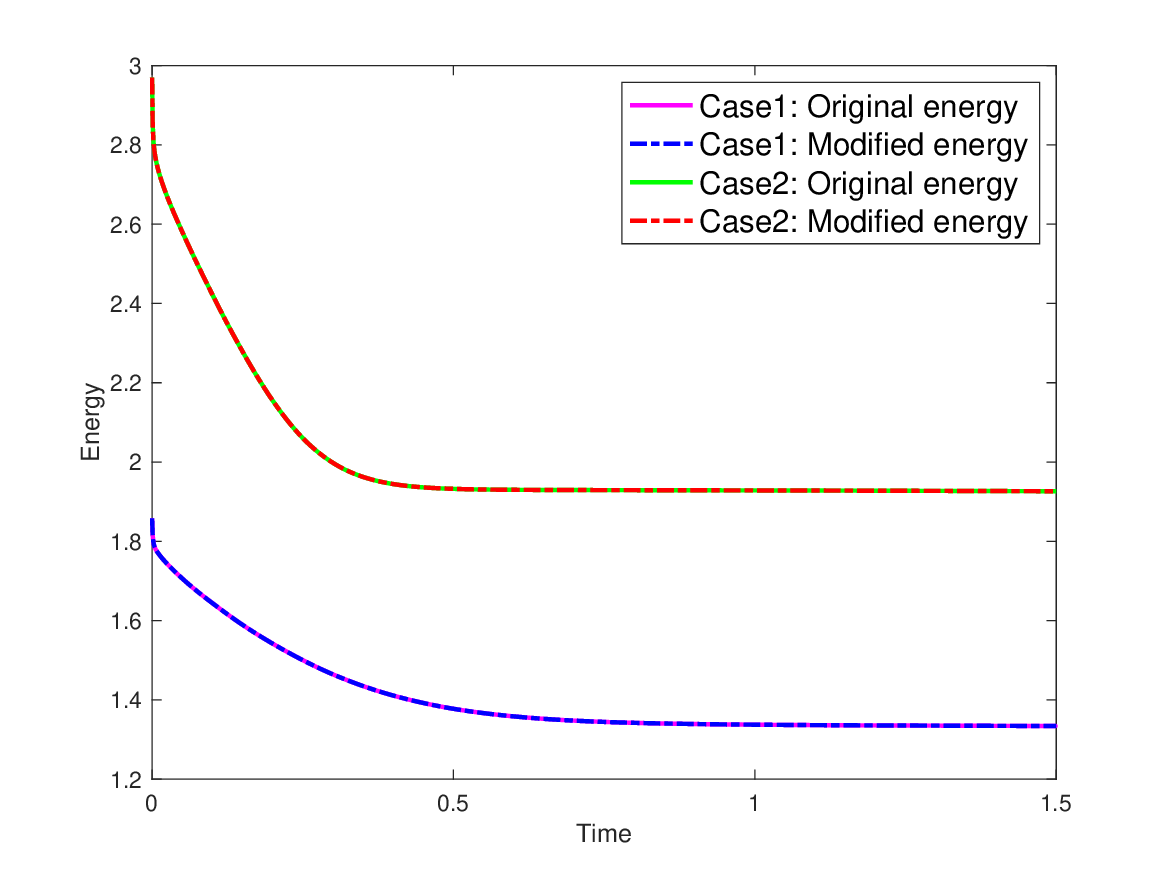} \label{energy_evolu}
	\end{minipage}
}
	\subfigure[]{
		\begin{minipage}[c]{0.45\textwidth}
			\includegraphics[width=1\textwidth]{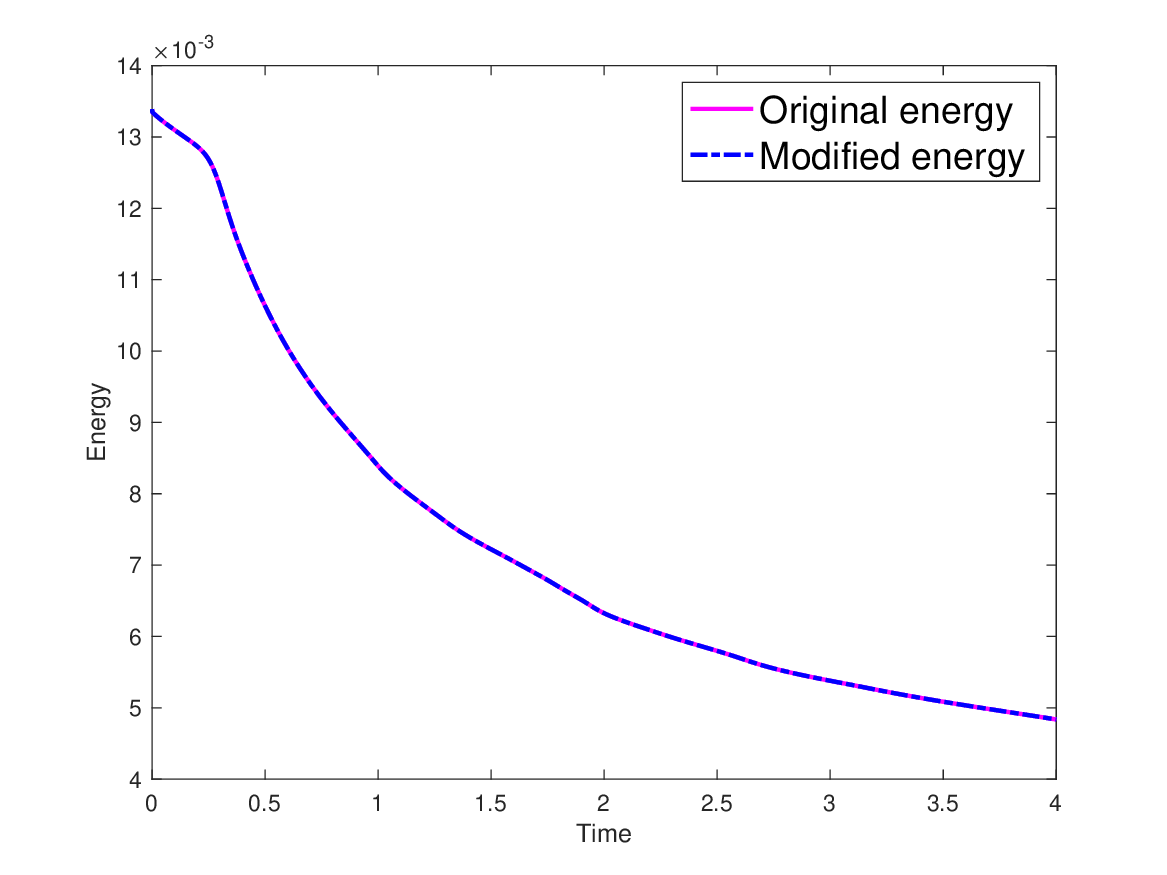} \label{energy_sep}
		\end{minipage}
	}
	{\caption{Evolutions of original and modified energy curves with respect to shape relaxation (left) and nucleation process (right) which are shown in subsection \ref{Shape relaxation} and subsection \ref{nucleation}, respectively.
}}
\end{figure}

%\begin{figure} 
%\centering 
%\includegraphics[width=0.4\textwidth]{energy.eps}
%	{\caption{Original and modified energy evolutions for cases 1 and 2.} \label{energy_evolu} } 
%\end{figure}
\subsection{Flow-coupled phase separation}\label{nucleation}
In this subsection, the process of flow-coupled nucleation is considered. The initial condition of $phi$ is as follows
\begin{equation*}
	\begin{aligned}
		&\phi(x,y,0)=2y-1+0.01\text{rand}(x,y),
	\end{aligned}
\end{equation*}
where $\text{rand}(x,y)$ represents the random distribution between $-1$ and $1$. 

The parameters are set as follows:
\begin{equation}
	\begin{aligned}
		\Delta t=&1\times 10^{-3},\quad \lambda=1\times 10^{-5},
		\quad M=1\times 10^{-1},
		\\
	\epsilon=&1\times 10^{-2},	\quad 
		\gamma= 2\times 10^4,
		\quad 
		\nu=1, \quad
		C_0=10^5.
	\end{aligned}
\end{equation}
The magnitude of $\phi$ has a  larger value close to the upper and lower boundaries and a smaller value close to the domain center with the specified initially condition.
In Figrue \ref{rr}, it can be noticed that the phase separation occurs close to the domain center.
The generated droplets eventually disappear over time since the interfacial length as a whole shrank as a result of energy decays.
We can find that the energy decays monotonically in Figure \ref{energy_sep}.
 \begin{figure}[htp]
	\centering 
	\subfigure[$T$=1]{
		\begin{minipage}[c]{0.22\textwidth}
			%	\vspace{-0.15cm}
			\includegraphics[width=1.3\textwidth]{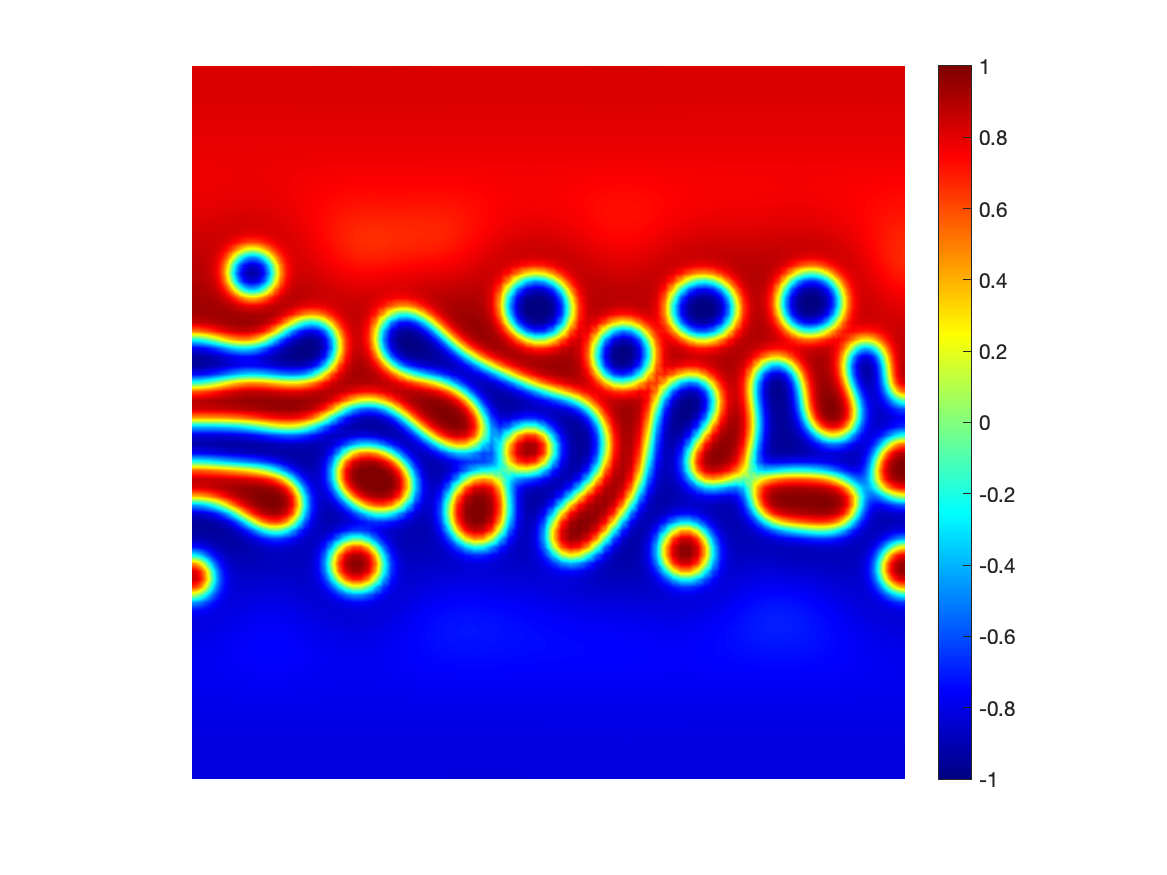}
			%	\vspace{-1cm}
		\end{minipage}
	}
	\subfigure[$T$=2]{
		\begin{minipage}[c]{0.22\textwidth}
			%	\vspace{-0.15cm}
			\includegraphics[width=1.3\textwidth]{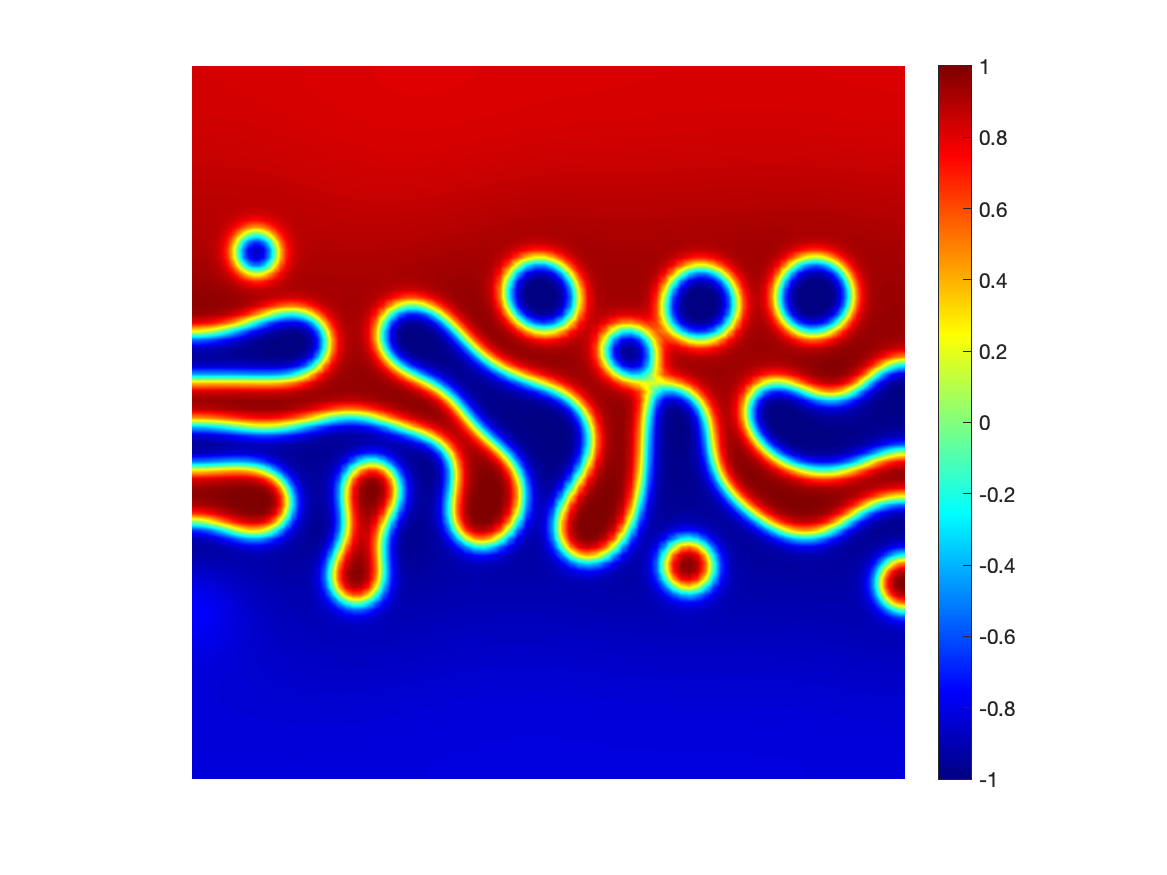}
			%	\vspace{-1cm}
		\end{minipage}
	}
	\subfigure[$T$=3]{
		\begin{minipage}[c]{0.22\textwidth}
			%	\vspace{-0.15cm}
			\includegraphics[width=1.3\textwidth]{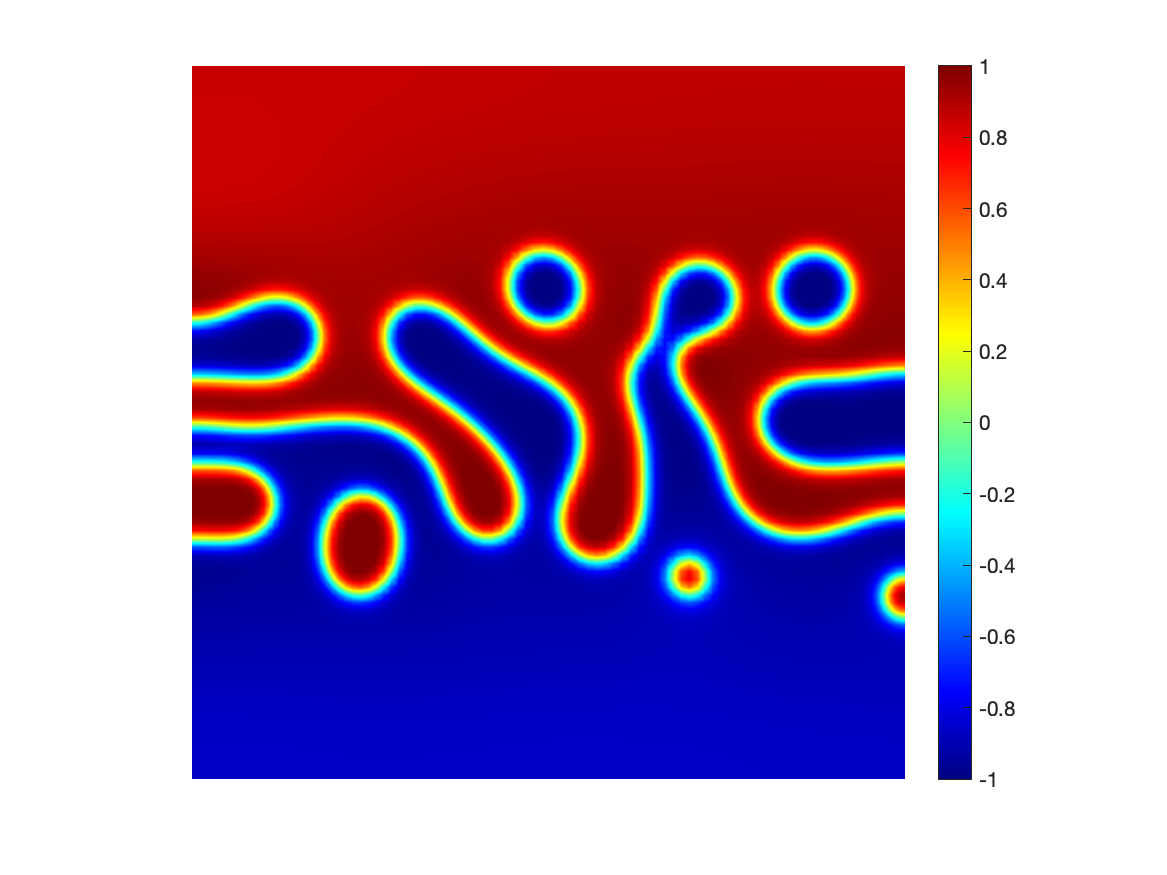}
			%	\vspace{-1cm}
		\end{minipage}
	}
	%\rule{15cm}{0.05em}\\
	\subfigure[$T$=4]{
		\begin{minipage}[c]{0.22\textwidth}
			\includegraphics[width=1.3\textwidth]{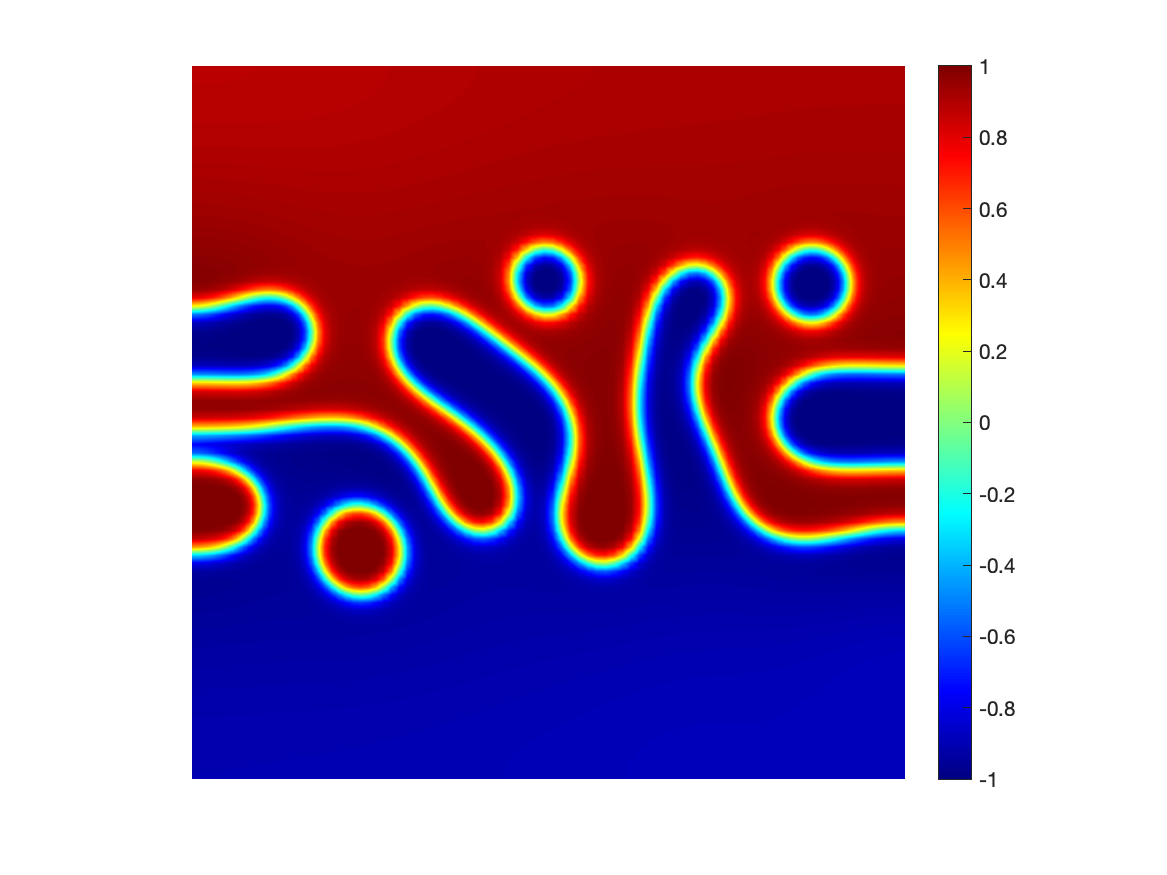}
			%	\vspace{-1cm}
		\end{minipage}
	}
	\subfigure[$T$=1]{
		\begin{minipage}[c]{0.22\textwidth}
			\includegraphics[width=1.3\textwidth]{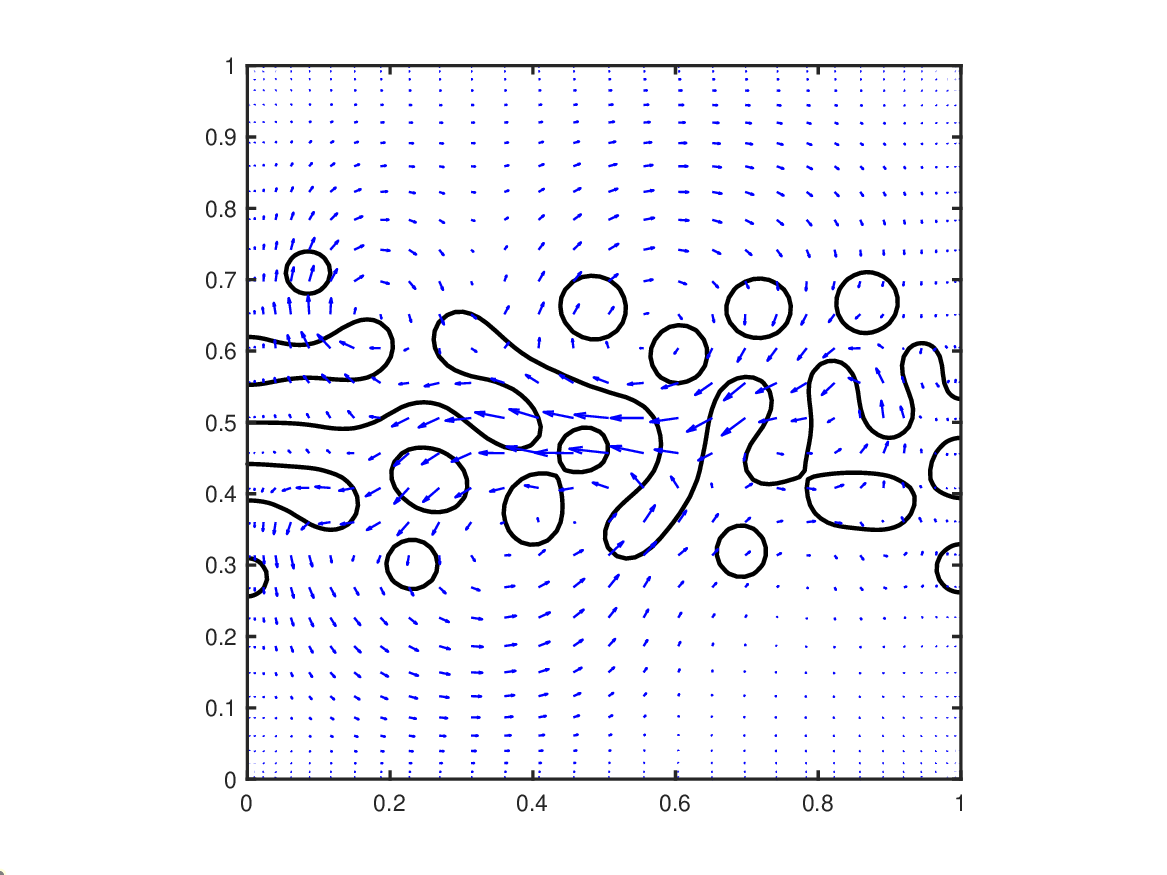}
			%	\vspace{-1cm}
		\end{minipage}
	}
	\subfigure[$T$=2]{
		\begin{minipage}[c]{0.22\textwidth}
			\includegraphics[width=1.3\textwidth]{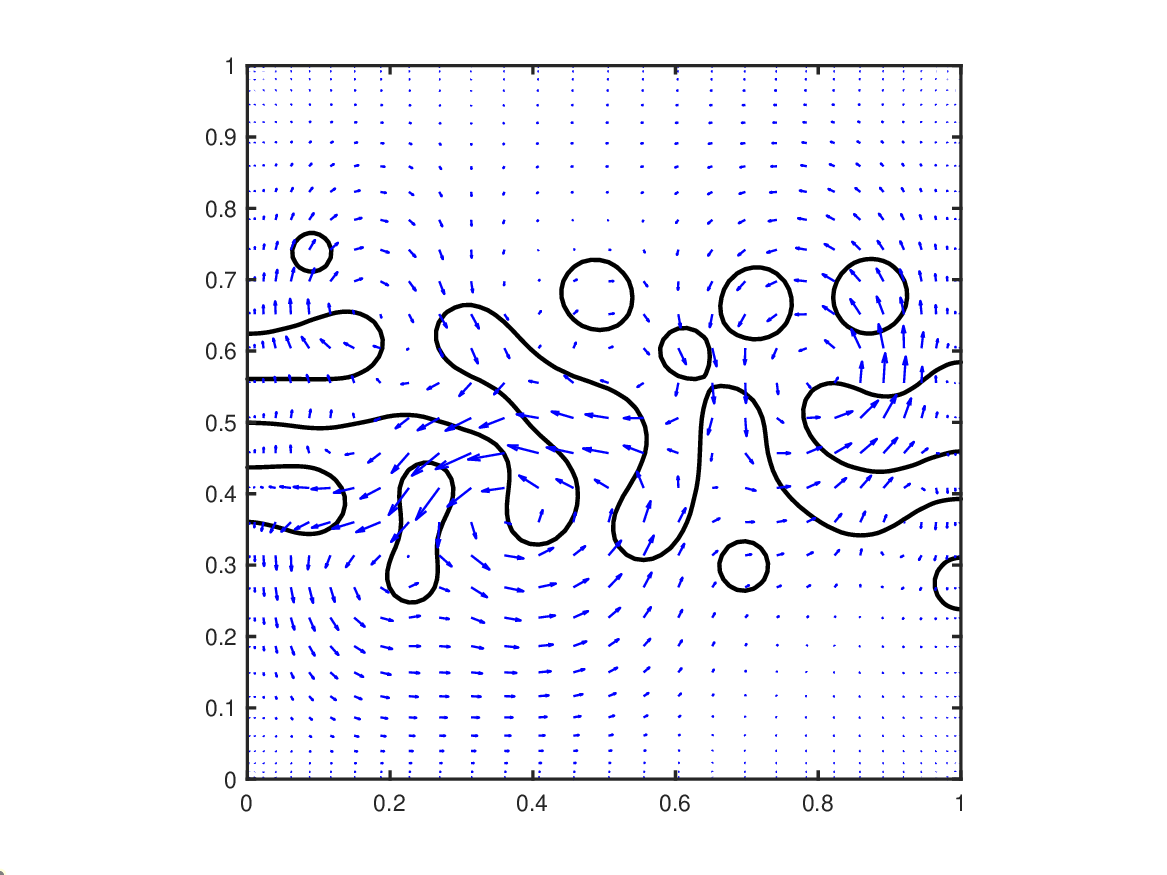}
			%	\vspace{-1cm}
		\end{minipage}
	}
	\subfigure[$T$=3]{
		\begin{minipage}[c]{0.22\textwidth}
			\includegraphics[width=1.3\textwidth]{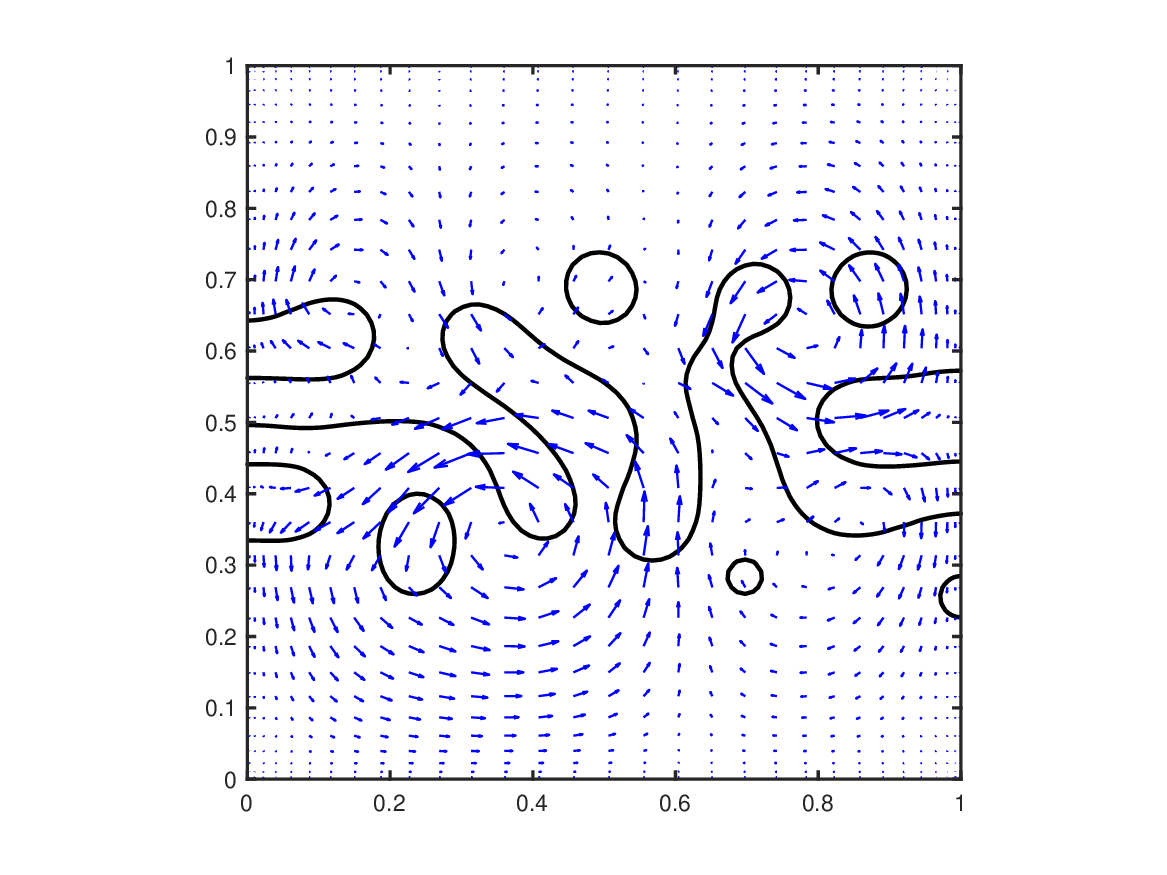}
			%	\vspace{-1cm}
		\end{minipage}
	}
	\subfigure[$T$=4]{
		\begin{minipage}[c]{0.22\textwidth}
			\includegraphics[width=1.3\textwidth]{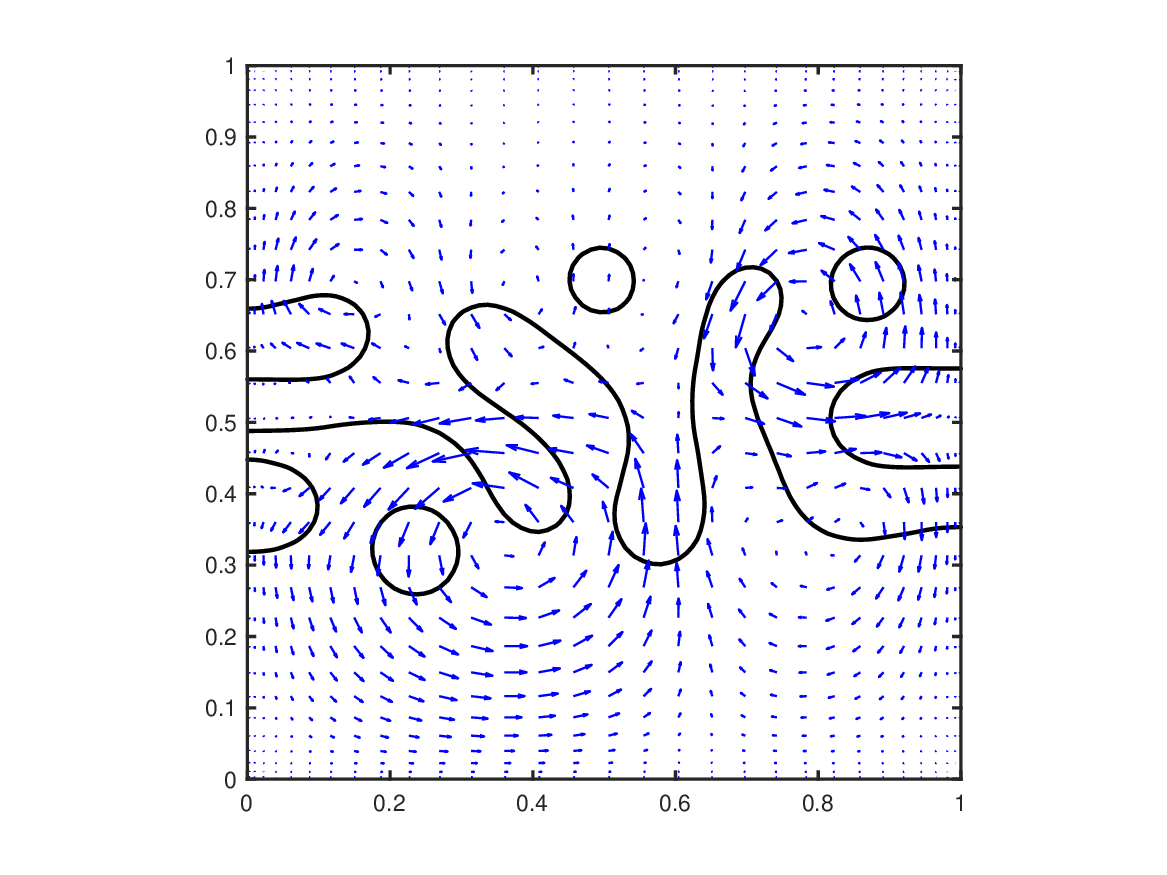}
			%	\vspace{-1cm}
		\end{minipage}
	}
	{\caption{The top and bottom rows display the snapshots of the phase function $\phi$ and velocity field for the flow-coupled nucleation process at different time $T$.} \label{rr}}
\end{figure}

\subsection{Buoyancy-driven flow}
In this example, we reformulate the momentum equation as follows:
\begin{equation}
	  \frac{\partial \textbf{u}}{\partial t}+ \textbf{u}\cdot \nabla\textbf{u}
	-\nu\Delta\textbf{u}+\nabla p= \mu \nabla \phi+\chi\rho(\phi)\textbf{g},
\end{equation}
where $\chi\rho(\phi)\textbf{g}$ is a buoyancy term
$\chi\rho(\phi)=\chi(\phi-\bar{\phi})$, and 
$\bar{\phi}$ is spatially averaged order parameter.
The computational domain is $\Omega= [0, 1]^2$ and we use $128\times128$ modes to discretize the space variables.
\subsubsection{Bubble rising}
The numerical and physical parameters are provided as follows:
\begin{equation}
	\begin{aligned}
\Delta t=&5\times 10^{-4},\quad \lambda=1\times 10^{-3},
\quad M=1\times 10^{-2},\quad \epsilon=1\times 10^{-2},
\\
 \gamma= &2\times 10^4,
 \quad 
 \nu=1, \quad 
 g=(0,-1)^T, \quad  \chi=5\times10, \quad
 C_0=10^5.
 	\end{aligned}
\end{equation}
We set the initial condition for the phase function as a circular bubble centered at $(0.5, 0.25)$, and the initial velocity is initialized as $u_0 = 0$.
 In Figure \ref{figure3}, snapshots of the phase evolution at various time points ($t = 1, 1.75, 2.55, 3.25, 3.35, 4$) are displayed. 
 Initially, the bubble appears as a circular shape near the bottom of the domain. The bubble, which is lighter compared to the surrounding fluid, rises gradually, transitioning into an elliptical shape, and eventually deforms as it approaches the upper boundary, as expected.

 \begin{figure}[htp]
 	\centering 
 	%\vspace{-0.5cm}
 	\subfigure[$T$=1]{
 		\begin{minipage}[c]{0.3\textwidth}
 			%	\vspace{-0.15cm}
 			\includegraphics[width=1.1\textwidth]{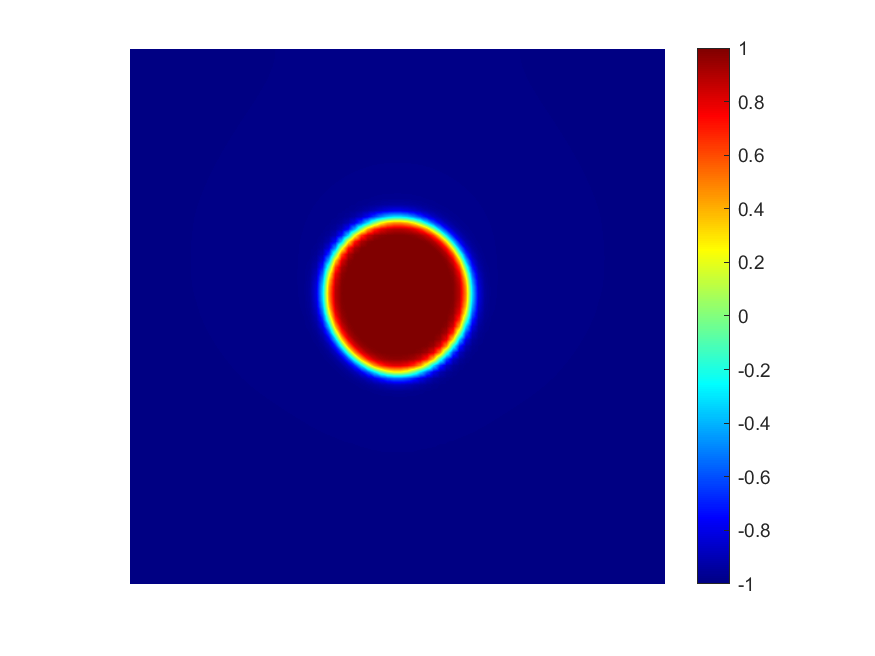}
 			%	\vspace{-1cm}
 		\end{minipage}
 	}
 	\subfigure[$T$=1.75]{
 		\begin{minipage}[c]{0.3\textwidth}
 			%	\vspace{-0.15cm}
 			\includegraphics[width=1.1\textwidth]{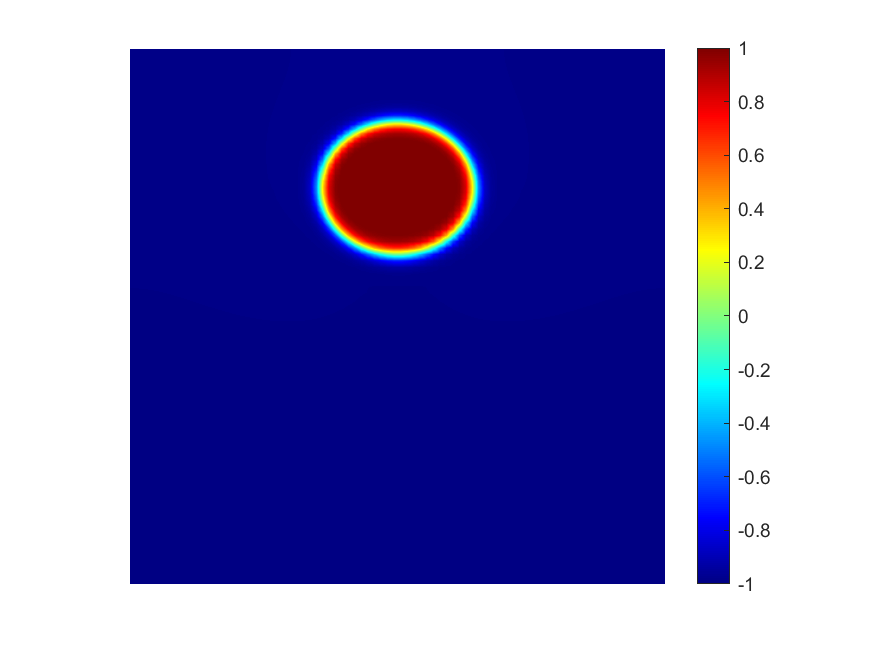}
 			%	\vspace{-1cm}
 		\end{minipage}
 	}
 	\subfigure[$T$=2.55]{
 		\begin{minipage}[c]{0.3\textwidth}
 			%	\vspace{-0.15cm}
 			\includegraphics[width=1.1\textwidth]{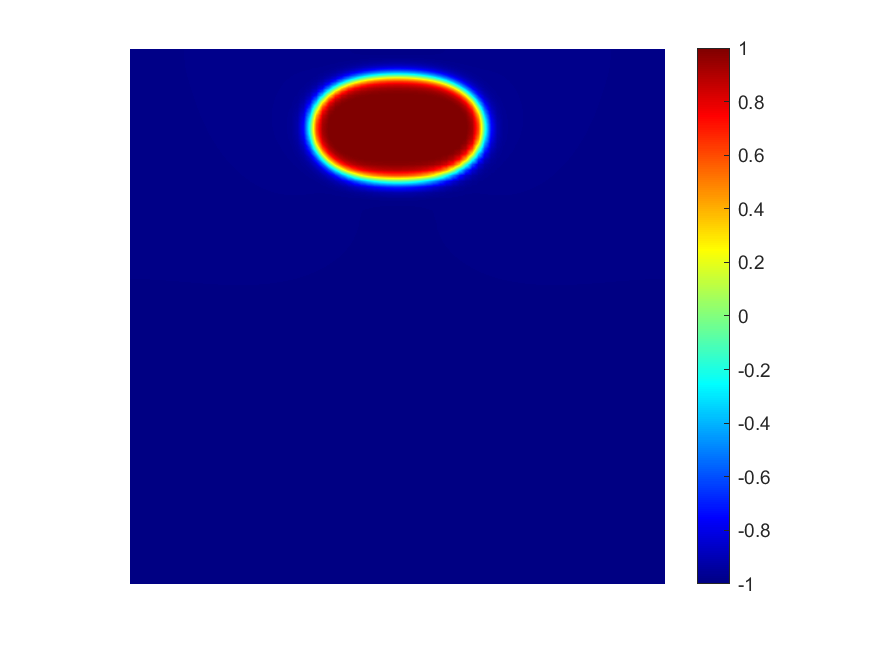}
 			%	\vspace{-1cm}
 		\end{minipage}
 	}
 	%\rule{15cm}{0.05em}\\
 	\subfigure[$T$=3.25]{
 		\begin{minipage}[c]{0.3\textwidth}
 			\includegraphics[width=1.1\textwidth]{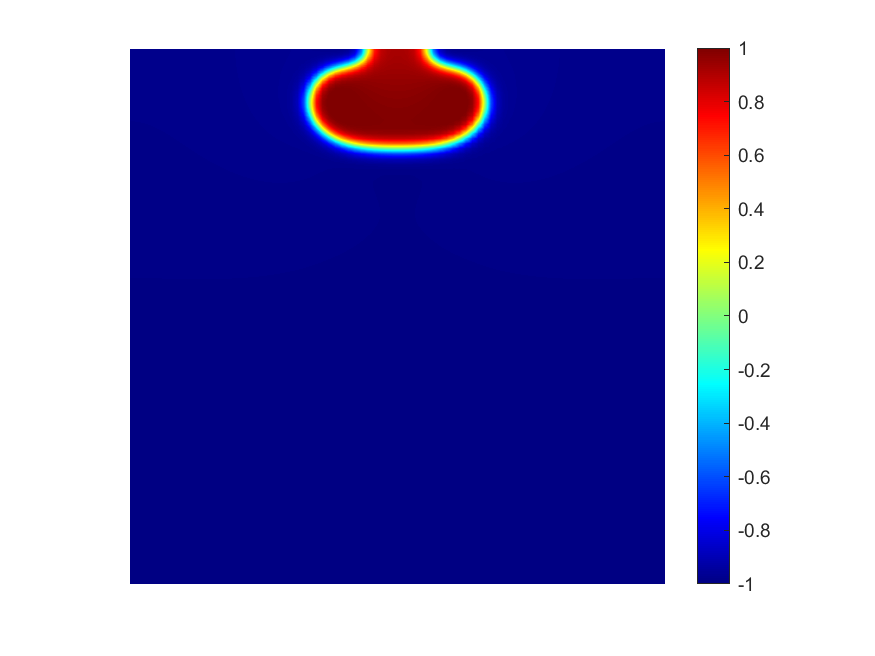}
 			%	\vspace{-1cm}
 		\end{minipage}
 	}
 	\subfigure[$T$=3.35]{
 		\begin{minipage}[c]{0.3\textwidth}
 			\includegraphics[width=1.1\textwidth]{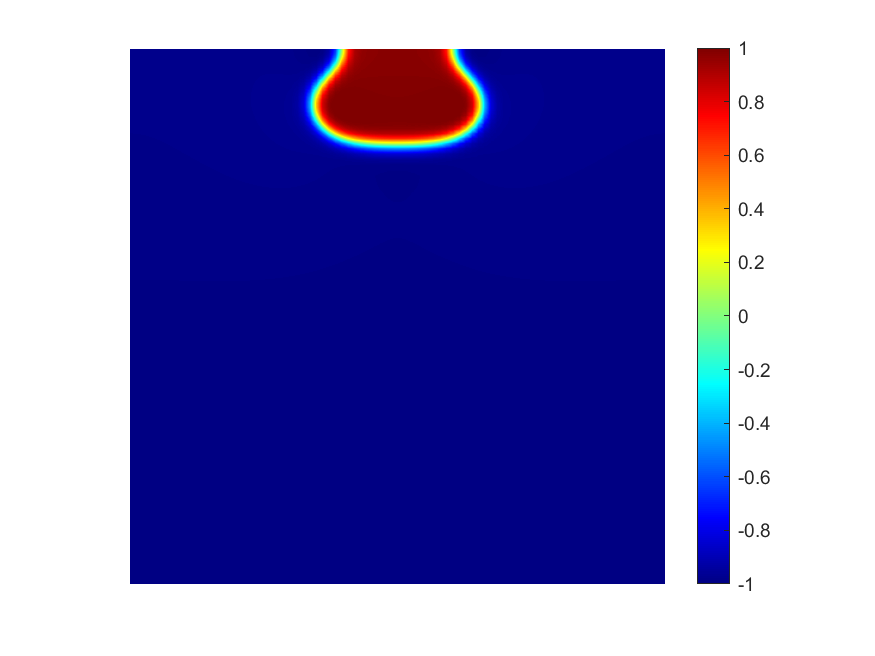}
 			%	\vspace{-1cm}
 		\end{minipage}
 	}
 	\subfigure[$T$=4]{
 		\begin{minipage}[c]{0.3\textwidth}
 			\includegraphics[width=1.1\textwidth]{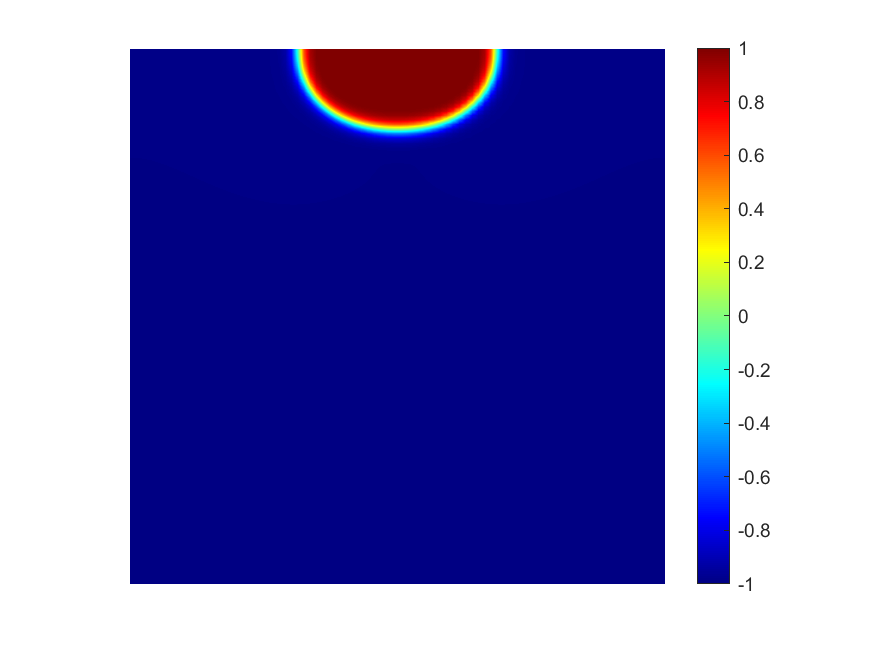}
 			%	\vspace{-1cm}
 		\end{minipage}
 	}
 	%	\vspace{-0.5cm}
 	{\caption{Snapshots of the phase function $\phi$ at different $T$.}\label{figure3}}
 \end{figure}

 \subsubsection{Dripping droplet}
 We first conduct simulations to observe the evolving behavior of a dripping droplet under different Reynolds numbers: $\nu= 1/10, 1/50$, and $1/100$.
Set the parameters as
\begin{equation}
	\begin{aligned}
			\Delta t=&1\times 10^{-3},
		\quad \lambda=1\times10^{-3},
		\quad M=1\times10^{-2},
		\quad \epsilon=1\times 10^{-2},\\
		&	\gamma=2\times10^4,
		\quad 
		g=(0,1)^T, \quad  \chi=10,
		\quad
		C_0=10^5.
	\end{aligned}
\end{equation}
 Initially, the droplet with heavier density is attached to the upper solid wall.
Due to the influence of gravity, the droplet gradually descends over time.
In Figure \ref{figure6}, the top row presents a snapshot of the droplet at $\nu = 1/10$, while the second and third rows depict the evolution results for $\nu = 1/50$ and $\nu = 1/100$ respectively. It is evident that as the Reynolds number $1/\nu$ increases, the pinch off occurs more rapidly and the droplet descends at a faster rate.

Next, we choose the parameters as
%\begin{equation}
%	\begin{aligned}
%			\Delta t=&1\times 10^{-3},
%		\quad \lambda=1.125\times10^{-6},
%	\quad M=1.875\times10^{-1},
%	\quad \epsilon=7.5\times 10^{-3},\\
%		&\nu=1\times 10^{-1},\quad
%			\gamma=\frac{2}{\epsilon^2}, \quad 
%		g=(0,1)^T, \quad  \chi=10,
%		 \quad
%		C_0=10^5.
%	\end{aligned}
%\end{equation}
\begin{equation}
	\begin{aligned}
		\Delta t=&1\times 10^{-3},
		\quad \lambda=1\times10^{-5},
		\quad M=1\times10^{-1},
		\quad \epsilon=7.5\times 10^{-3},\\
		&\gamma=\frac{2}{\epsilon^2}, \quad 
		\nu=5\times 10^{-2},\quad
		g=(0,1)^T, \quad  \chi=10,
		\quad
		C_0=10^5.
	\end{aligned}
\end{equation}
 The snapshot of the droplet is shown in Figures \ref{figure5}-\ref{figure7}  with the initial state unchanged. The formation of spike structures becomes evident, particularly when the liquid filament is extremely elongated.

 \begin{figure}[!t]
 	\centering 
 	%\vspace{-0.5cm}
 	\subfigure[$\nu=0.1$, \ $T$=0.8]{
 		\begin{minipage}[c]{0.3\textwidth}
 			%	\vspace{-0.15cm}
 			\includegraphics[width=1.1\textwidth]{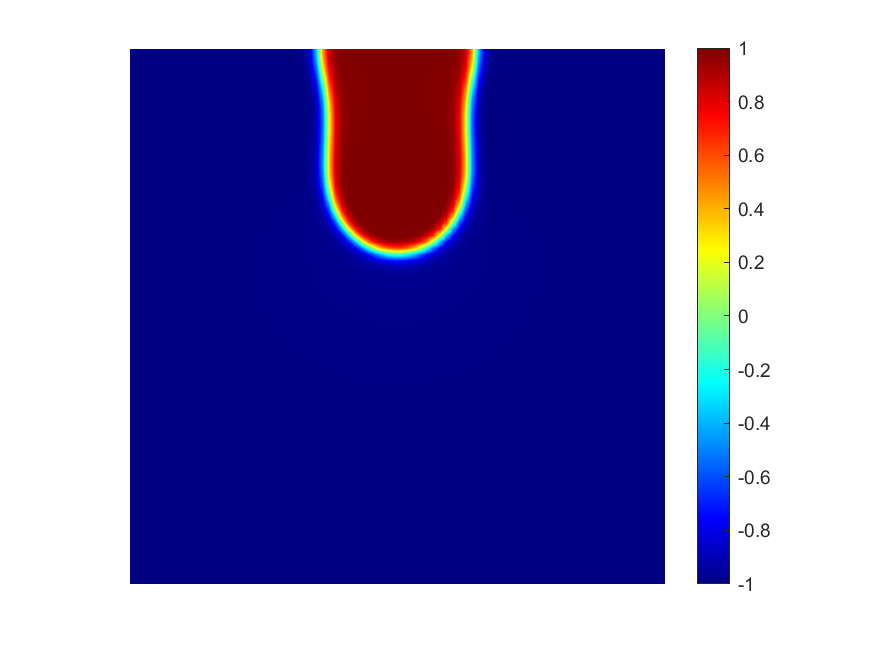}
 			%	\vspace{-1cm}
 		\end{minipage}
 	}
 	\subfigure[$\nu=0.1$, \ $T$=1.2]{
 		\begin{minipage}[c]{0.3\textwidth}
 			%	\vspace{-0.15cm}
 			\includegraphics[width=1.1\textwidth]{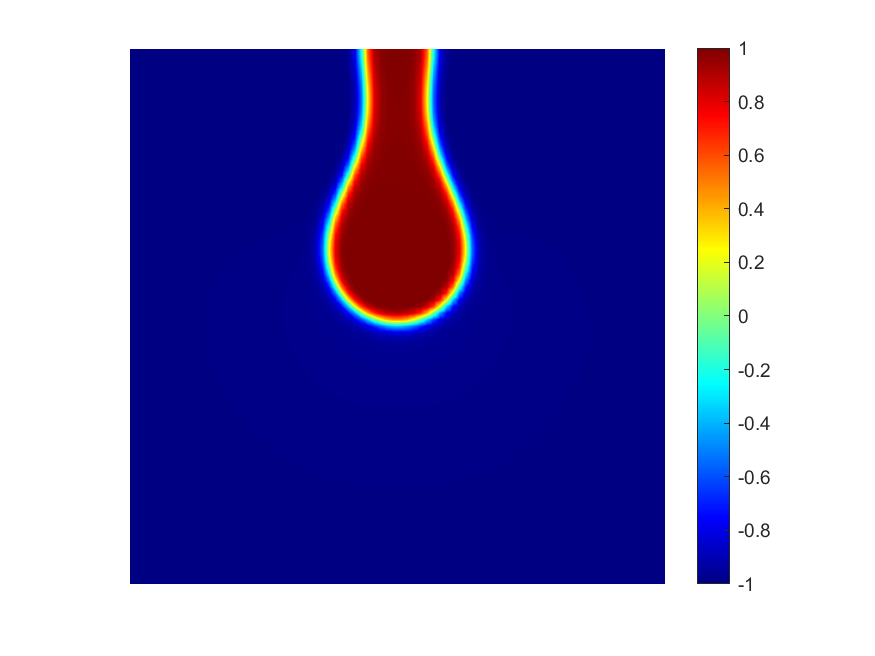}
 			%	\vspace{-1cm}
 		\end{minipage}
 	}
 	\subfigure[$\nu=0.1$, \ $T$=2]{
 		\begin{minipage}[c]{0.3\textwidth}
 			%	\vspace{-0.15cm}
 			\includegraphics[width=1.1\textwidth]{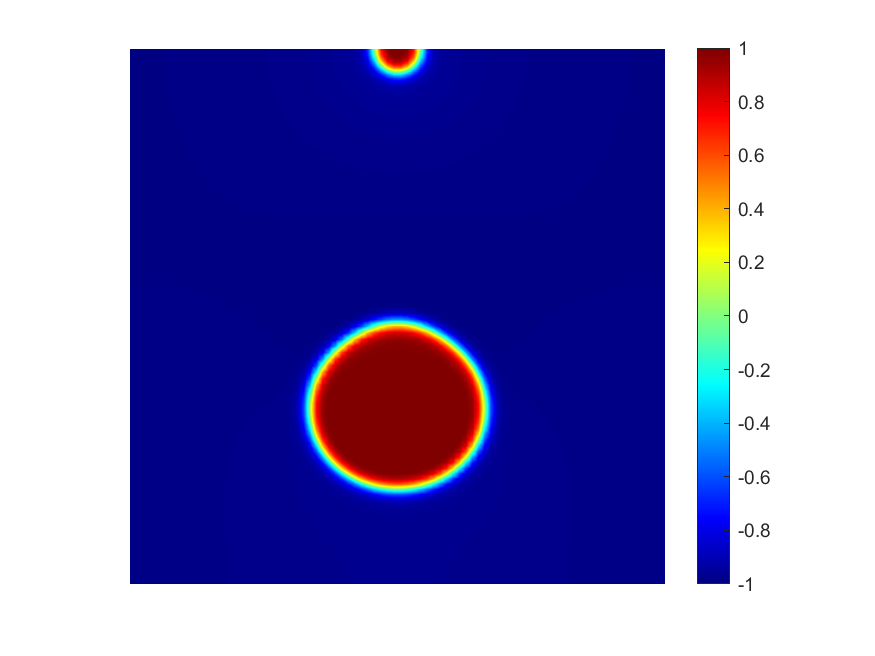}
 			%	\vspace{-1cm}
 		\end{minipage}
 	}
 	%\rule{15cm}{0.05em}\\
 	\subfigure[$\nu=0.01$,\ $T$=0.2]{
 		\begin{minipage}[c]{0.3\textwidth}
 			\includegraphics[width=1.1\textwidth]{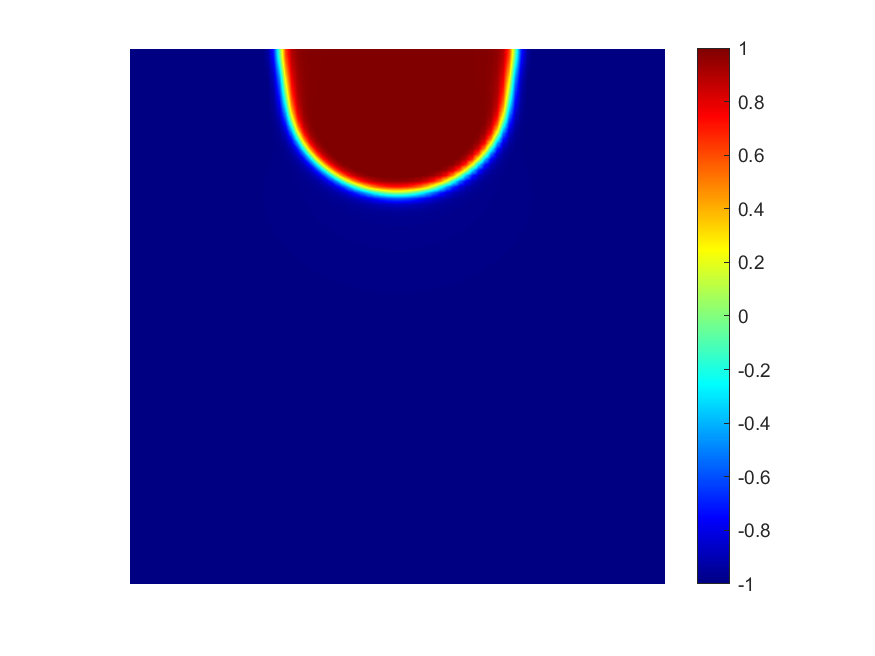}
 			%	\vspace{-1cm}
 		\end{minipage}
 	}
 	\subfigure[$\nu=0.01$,\ $T$=0.4]{
 		\begin{minipage}[c]{0.3\textwidth}
 			\includegraphics[width=1.1\textwidth]{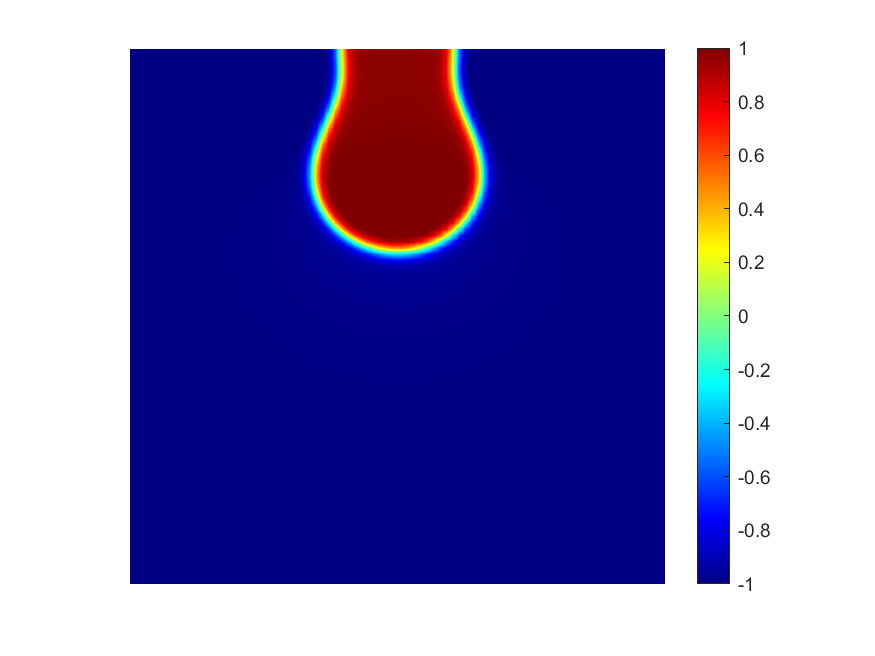}
 			%	\vspace{-1cm}
 		\end{minipage}
 	}
 	\subfigure[$\nu=0.01$,\ $T$=0.6]{
 		\begin{minipage}[c]{0.3\textwidth}
 			\includegraphics[width=1.1\textwidth]{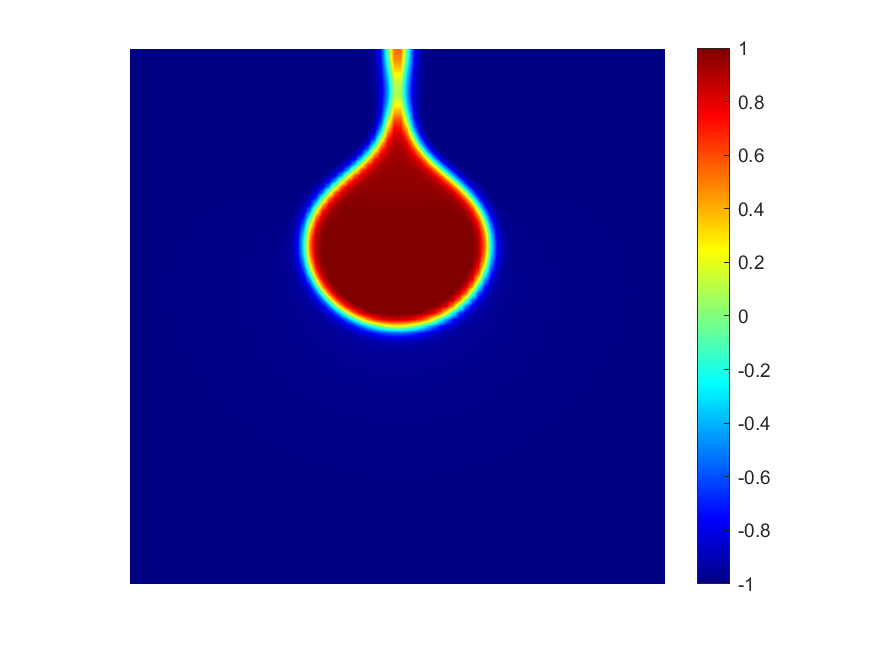}
 			%	\vspace{-1cm}
 		\end{minipage}
 	}
 	\subfigure[$\nu=0.005$,\ $T$=0.2]{
 	\begin{minipage}[c]{0.3\textwidth}
 		\includegraphics[width=1.1\textwidth]{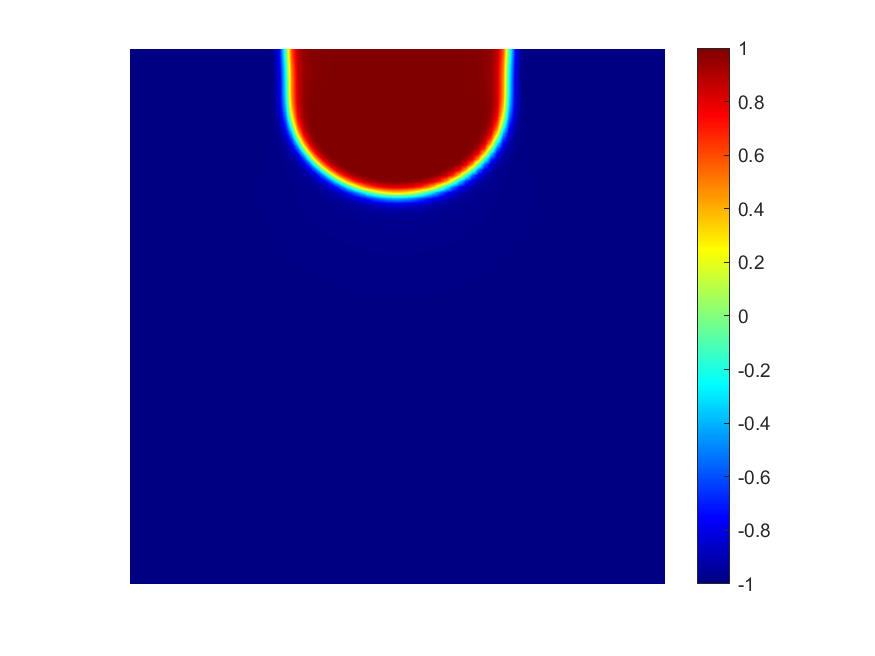}
 		%	\vspace{-1cm}
 	\end{minipage}
 }
 \subfigure[$\nu=0.005$,\ $T$=0.4]{
 	\begin{minipage}[c]{0.3\textwidth}
 		\includegraphics[width=1.1\textwidth]{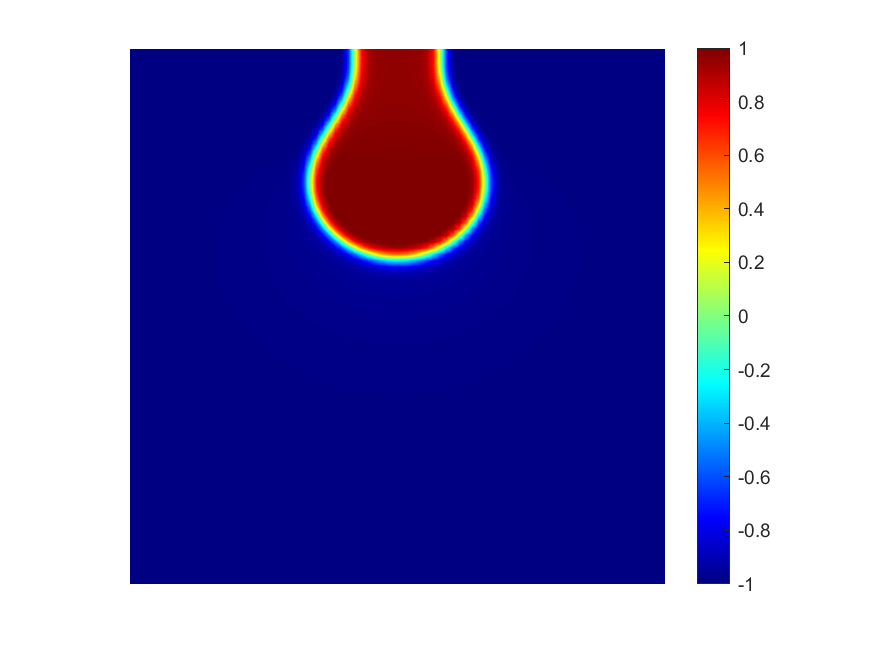}
 		%	\vspace{-1cm}
 	\end{minipage}
 }
 \subfigure[$\nu=0.005$,\ $T$=0.6]{
 	\begin{minipage}[c]{0.3\textwidth}
 		\includegraphics[width=1.1\textwidth]{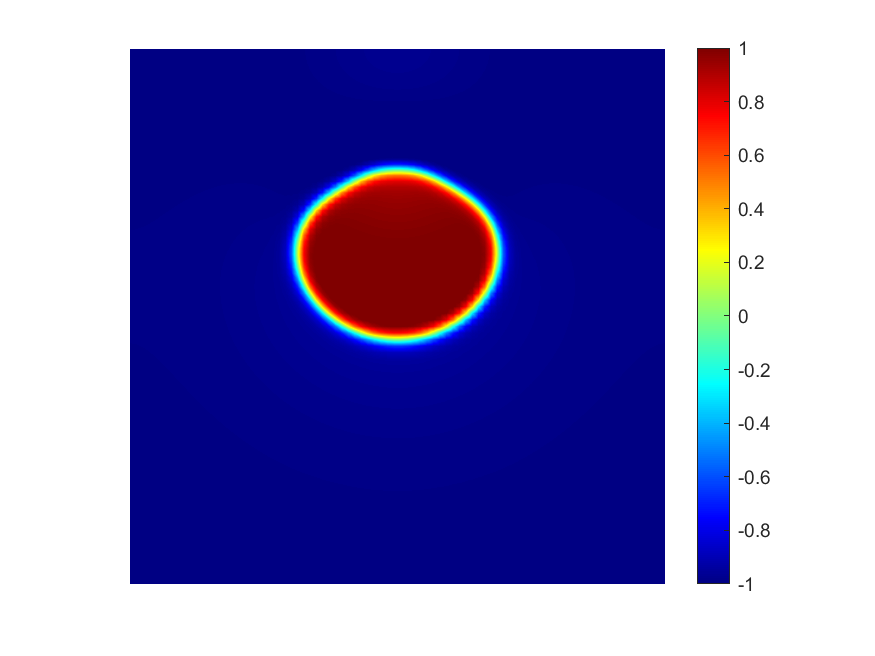}
 		%	\vspace{-1cm}
 	\end{minipage}
 }
 	%	\vspace{-0.5cm}
 	{\caption{Snapshots of the phase function $\phi$ at different $T$.}\label{figure6}}
 \end{figure}

% \begin{figure}[htp]
%	\centering 
%	%\vspace{-0.5cm}
%	\subfigure[$T$=0.2]{
%		\begin{minipage}[c]{0.3\textwidth}
%			%	\vspace{-0.15cm}
%			\includegraphics[width=1.1\textwidth]{c1.eps}
%			%	\vspace{-1cm}
%		\end{minipage}
%	}
%	\subfigure[$T$=0.4]{
%		\begin{minipage}[c]{0.3\textwidth}
%			%	\vspace{-0.15cm}
%			\includegraphics[width=1.1\textwidth]{c2.eps}
%			%	\vspace{-1cm}
%		\end{minipage}
%	}
%	\subfigure[$T$=0.6]{
%		\begin{minipage}[c]{0.3\textwidth}
%			%	\vspace{-0.15cm}
%			\includegraphics[width=1.1\textwidth]{c3.eps}
%			%	\vspace{-1cm}
%		\end{minipage}
%	}
%	%\rule{15cm}{0.05em}\\
%	\subfigure[$T$=0.7]{
%		\begin{minipage}[c]{0.3\textwidth}
%			\includegraphics[width=1.1\textwidth]{c4.eps}
%			%	\vspace{-1cm}
%		\end{minipage}
%	}
%	\subfigure[$T$=0.9]{
%		\begin{minipage}[c]{0.3\textwidth}
%			\includegraphics[width=1.1\textwidth]{c5.eps}
%			%	\vspace{-1cm}
%		\end{minipage}
%	}
%	\subfigure[$T$=1]{
%		\begin{minipage}[c]{0.3\textwidth}
%			\includegraphics[width=1.1\textwidth]{c6.eps}
%			%	\vspace{-1cm}
%		\end{minipage}
%	}
%	%	\vspace{-0.5cm}
%	{\caption{Snapshots of the phase function $\phi$ at different $T$.}\label{figure5}}
%\end{figure}

 \begin{figure}[!t]
	\centering 
	%\vspace{-0.5cm}
	\subfigure[$T$=0.2]{
		\begin{minipage}[c]{0.3\textwidth}
			%	\vspace{-0.15cm}
			\includegraphics[width=1.1\textwidth]{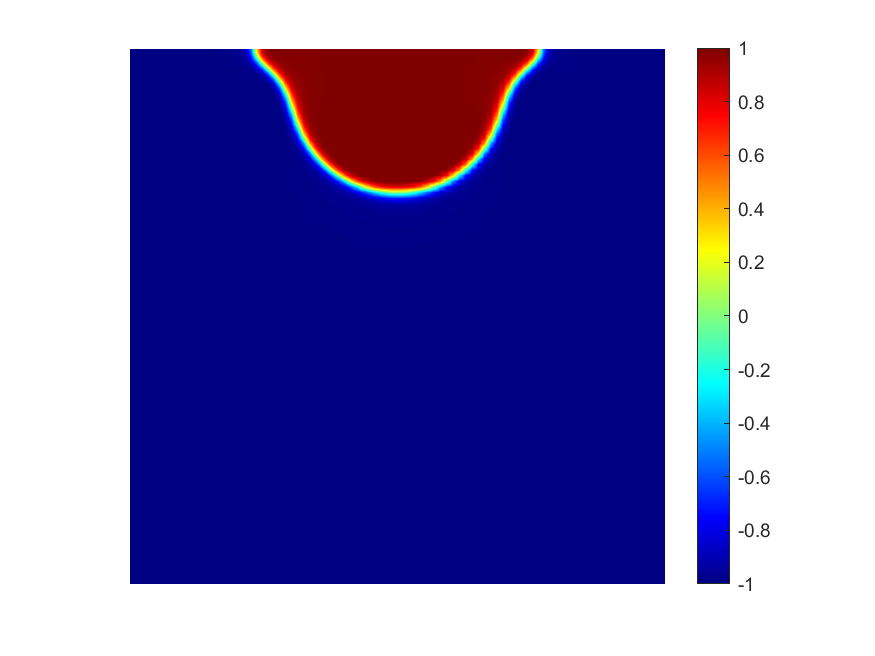}
			%	\vspace{-1cm}
		\end{minipage}
	}
	\subfigure[$T$=0.4]{
		\begin{minipage}[c]{0.3\textwidth}
			%	\vspace{-0.15cm}
			\includegraphics[width=1.1\textwidth]{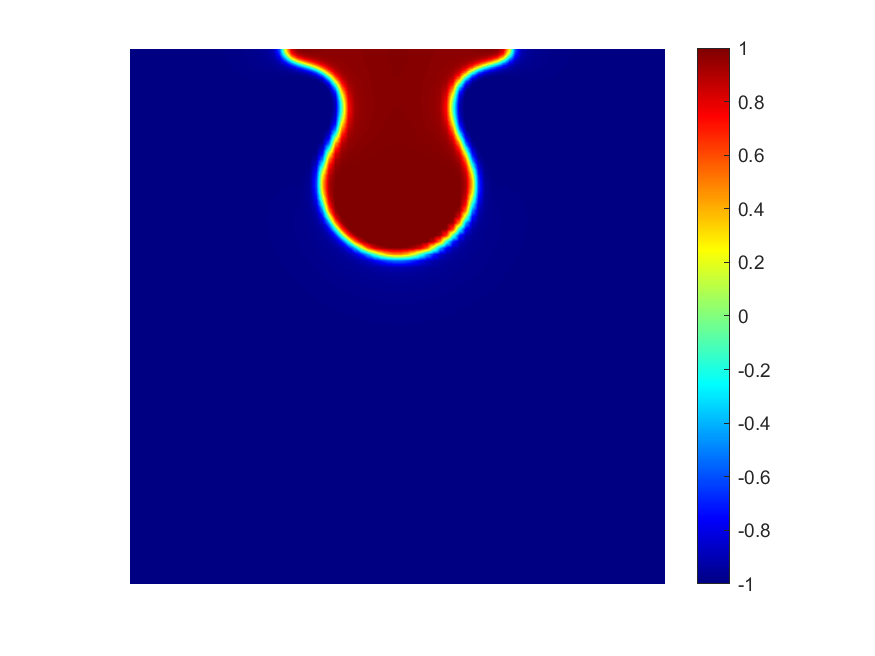}
			%	\vspace{-1cm}
		\end{minipage}
	}
	\subfigure[$T$=0.6]{
		\begin{minipage}[c]{0.3\textwidth}
			%	\vspace{-0.15cm}
			\includegraphics[width=1.1\textwidth]{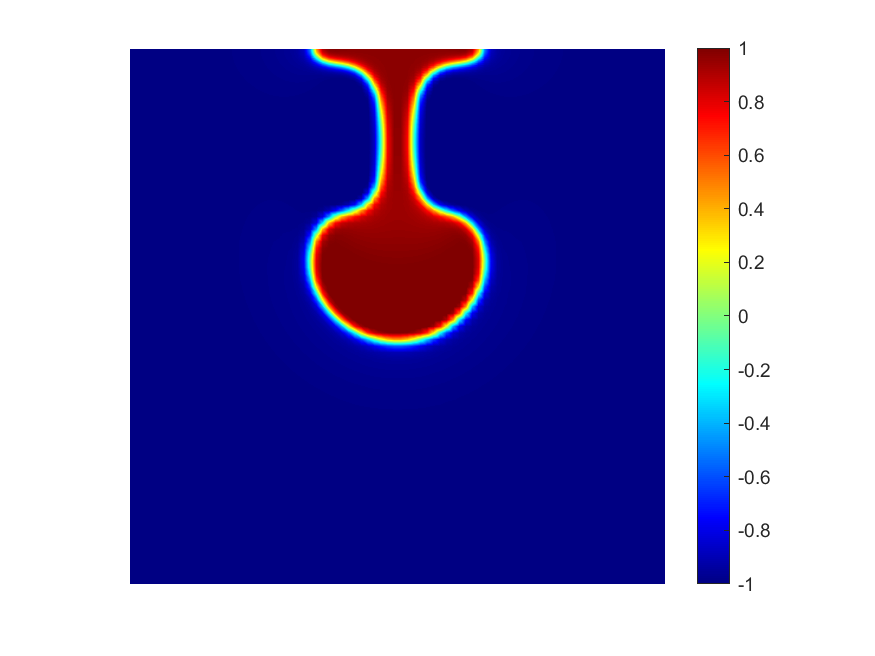}
			%	\vspace{-1cm}
		\end{minipage}
	}
	%\rule{15cm}{0.05em}\\
	\subfigure[$T$=0.7]{
		\begin{minipage}[c]{0.3\textwidth}
			\includegraphics[width=1.1\textwidth]{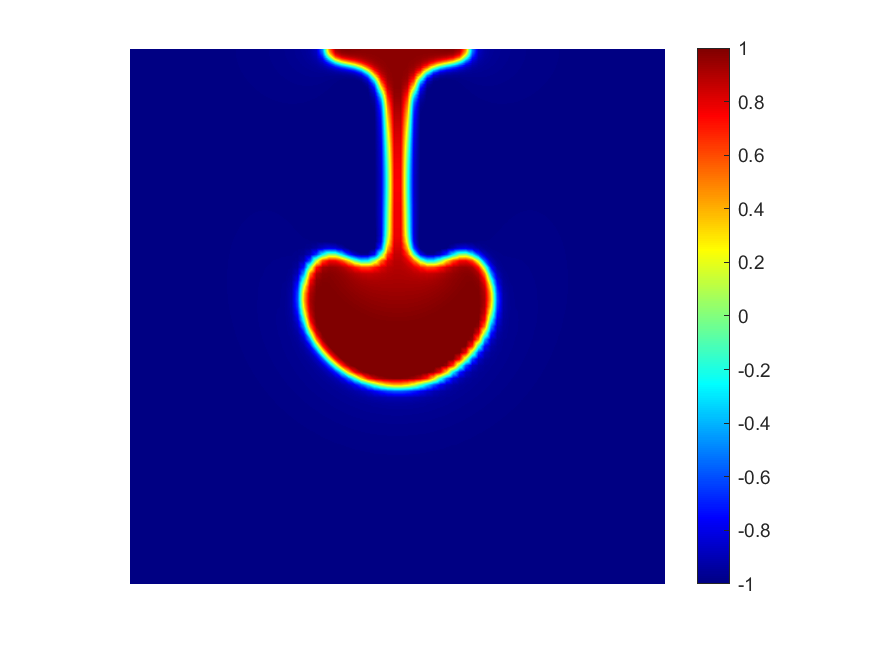}
			%	\vspace{-1cm}
		\end{minipage}
	}
	\subfigure[$T$=0.8]{
		\begin{minipage}[c]{0.3\textwidth}
			\includegraphics[width=1.1\textwidth]{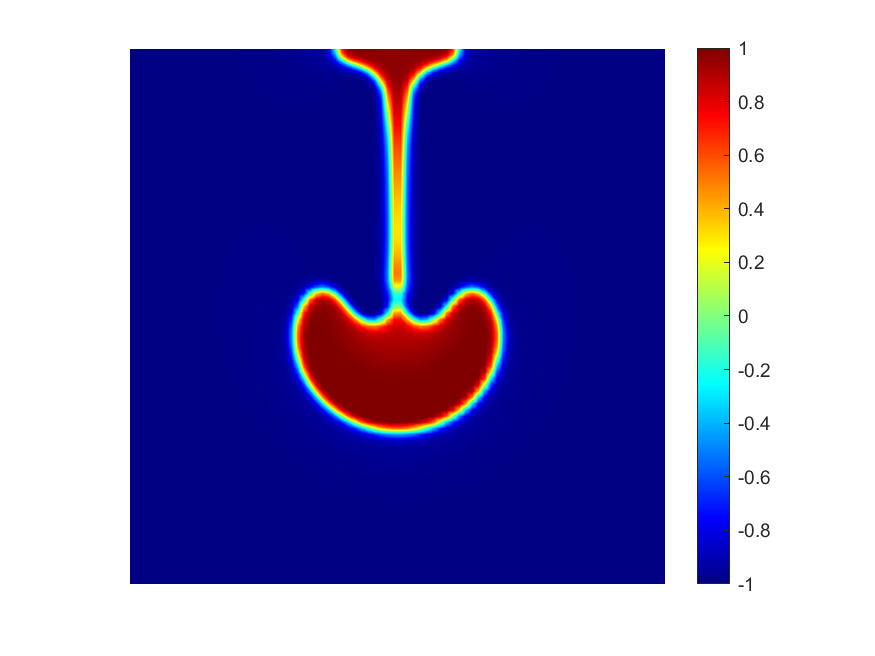}
			%	\vspace{-1cm}
		\end{minipage}
	}
	\subfigure[$T$=0.9]{
		\begin{minipage}[c]{0.3\textwidth}
			\includegraphics[width=1.1\textwidth]{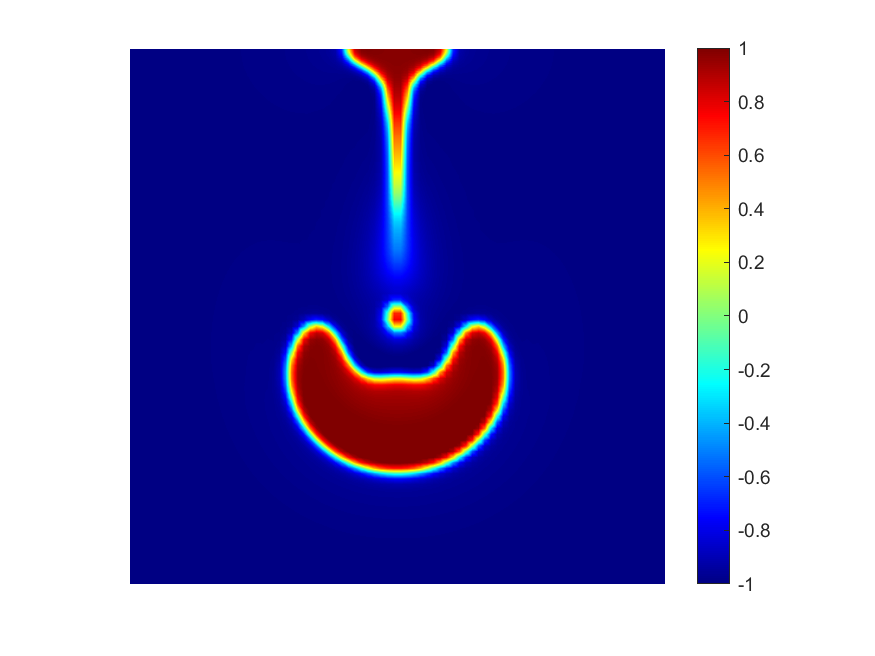}
			%	\vspace{-1cm}
		\end{minipage}
	}
	%	\vspace{-0.5cm}
	{\caption{Snapshots of the phase function $\phi$ at different $T$.}\label{figure5}}
\end{figure}

 \begin{figure}[!t]
	\centering 
	%\vspace{-0.5cm}
	\subfigure[$T$=0.2]{
		\begin{minipage}[c]{0.3\textwidth}
			%	\vspace{-0.15cm}
			\includegraphics[width=1.1\textwidth]{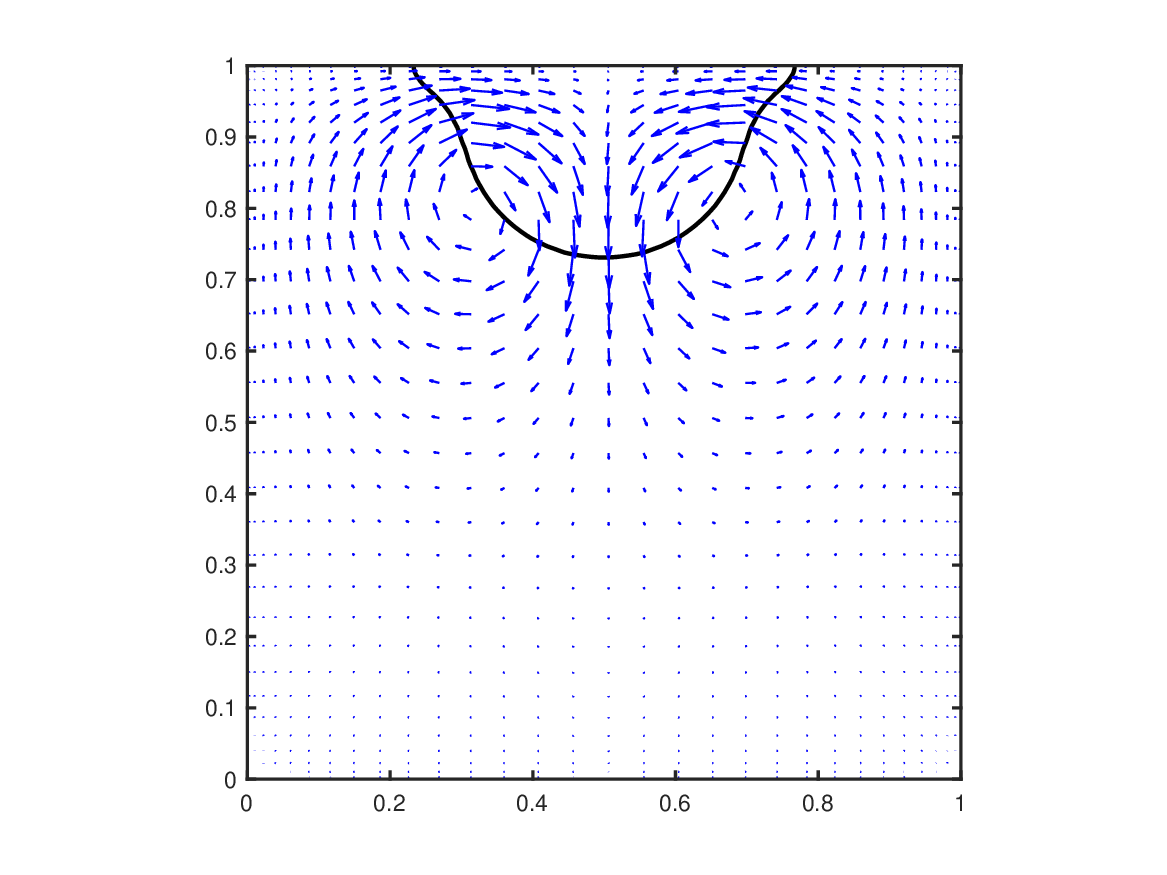}
			%	\vspace{-1cm}
		\end{minipage}
	}
	\subfigure[$T$=0.4]{
		\begin{minipage}[c]{0.3\textwidth}
			%	\vspace{-0.15cm}
			\includegraphics[width=1.1\textwidth]{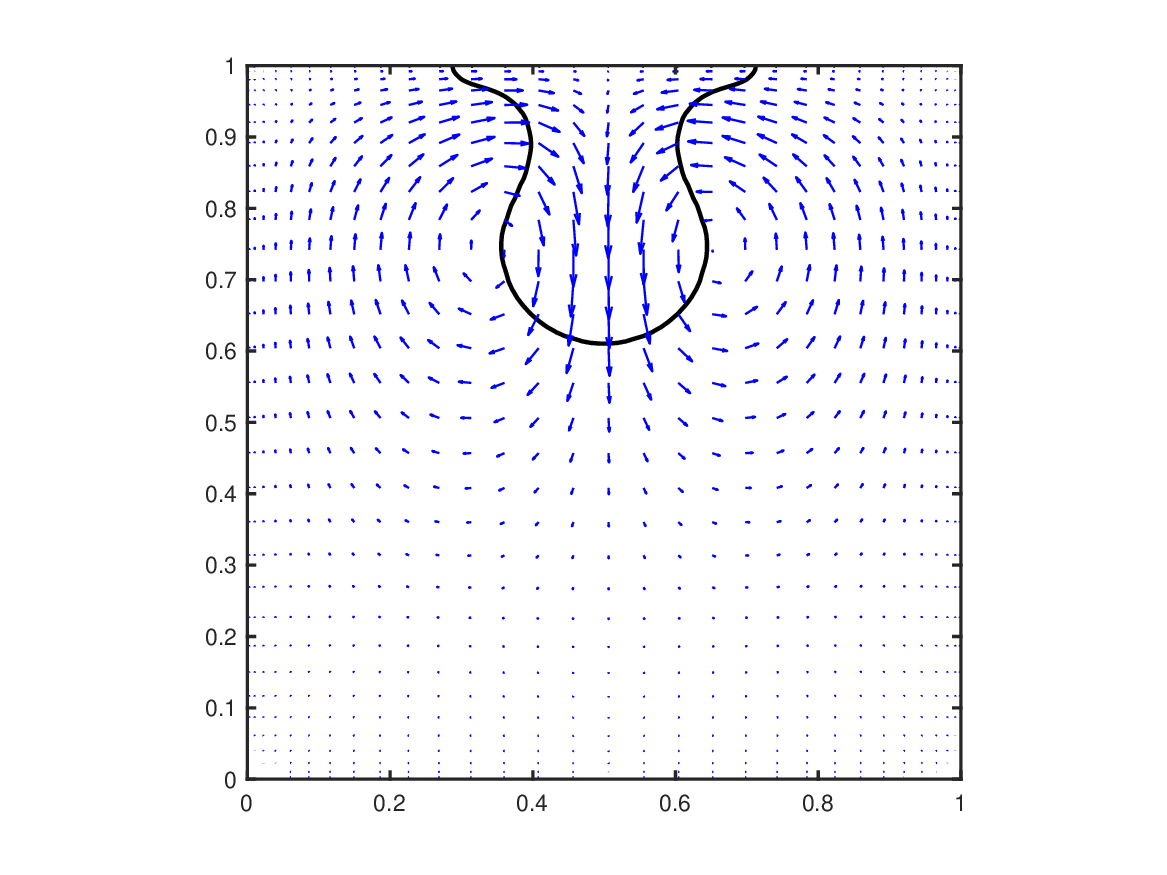}
			%	\vspace{-1cm}
		\end{minipage}
	}
	\subfigure[$T$=0.6]{
		\begin{minipage}[c]{0.3\textwidth}
			%	\vspace{-0.15cm}
			\includegraphics[width=1.1\textwidth]{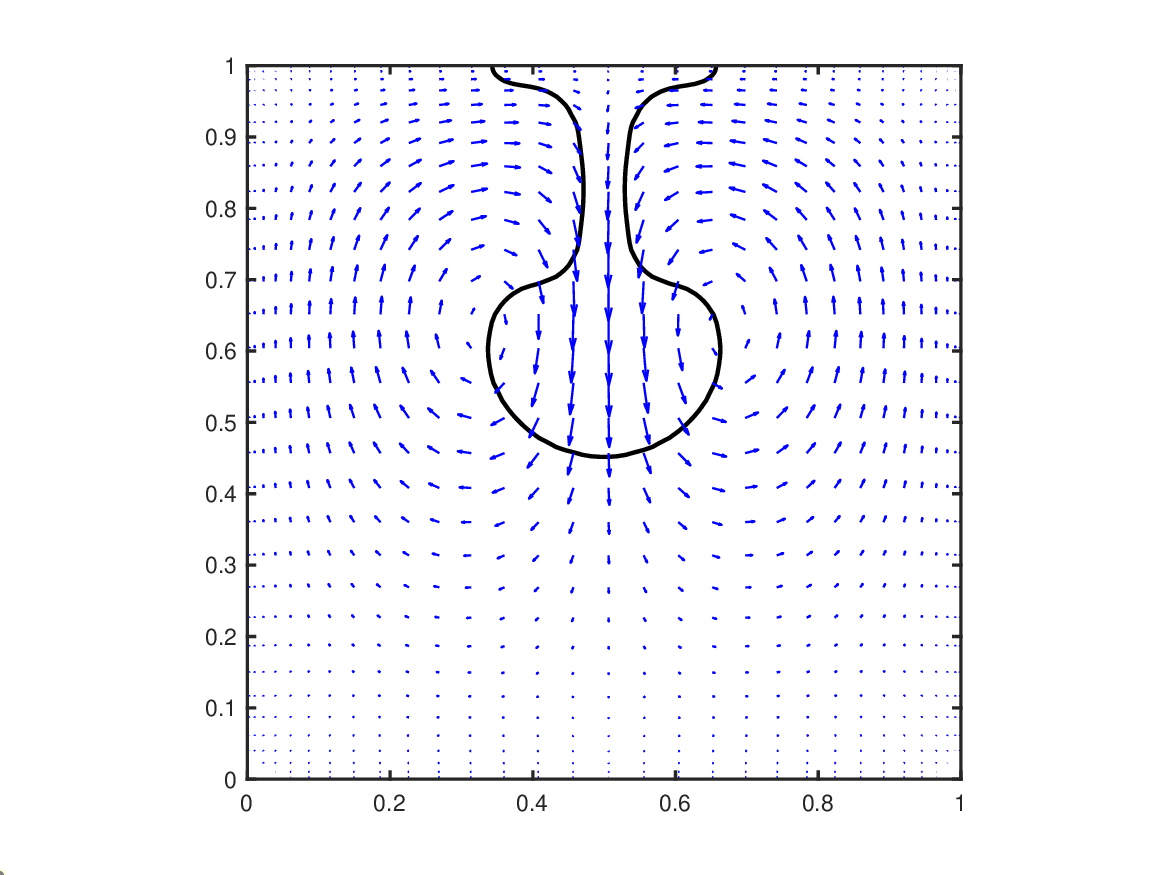}
			%	\vspace{-1cm}
		\end{minipage}
	}
	%\rule{15cm}{0.05em}\\
	\subfigure[$T$=0.7]{
		\begin{minipage}[c]{0.3\textwidth}
			\includegraphics[width=1.1\textwidth]{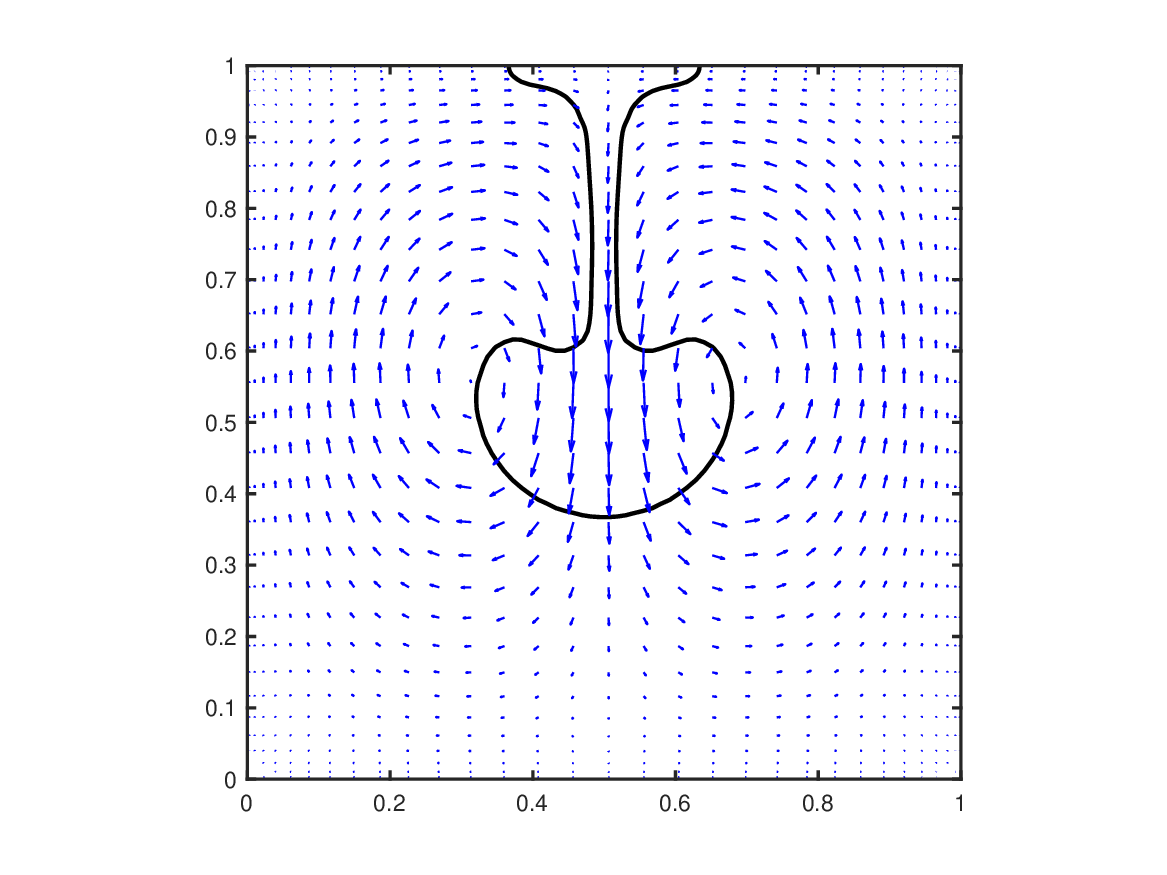}
			%	\vspace{-1cm}
		\end{minipage}
	}
	\subfigure[$T$=0.8]{
		\begin{minipage}[c]{0.3\textwidth}
			\includegraphics[width=1.1\textwidth]{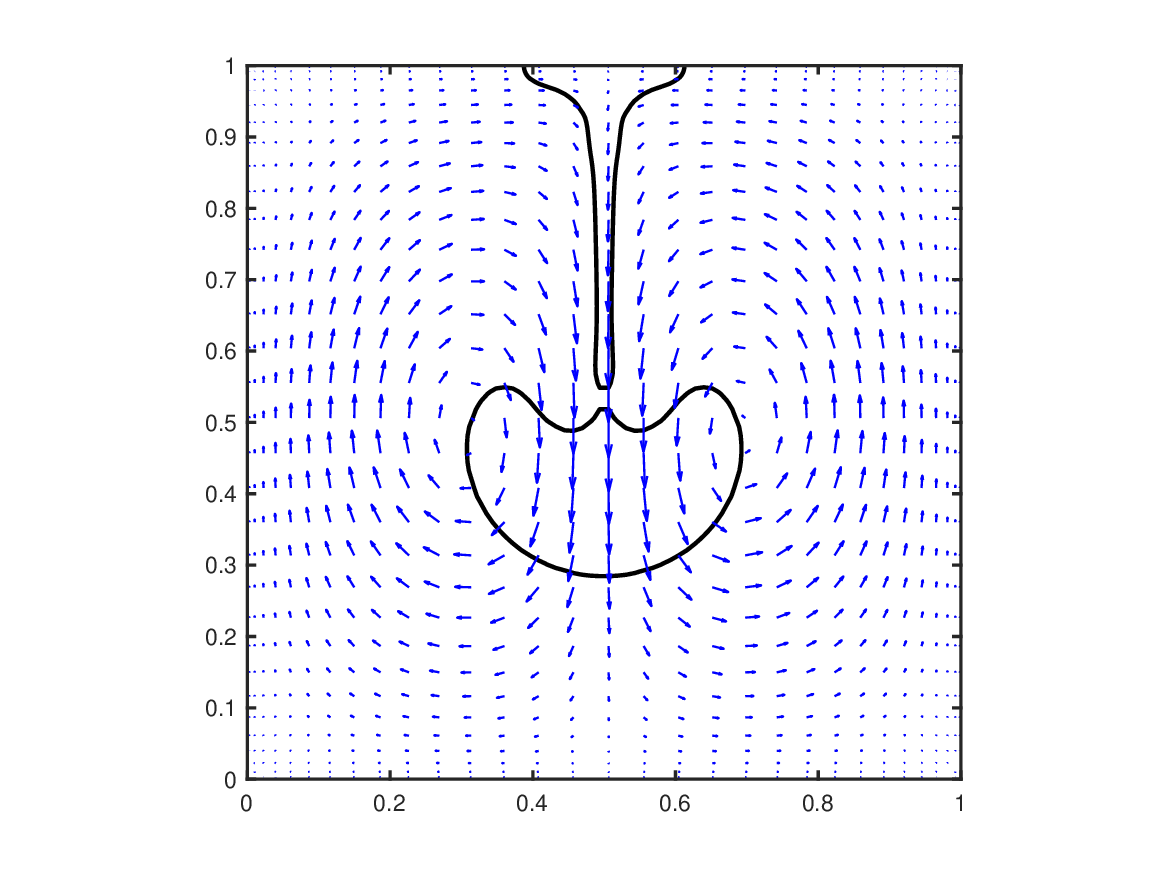}
			%	\vspace{-1cm}
		\end{minipage}
	}
	\subfigure[$T$=0.9]{
		\begin{minipage}[c]{0.3\textwidth}
			\includegraphics[width=1.1\textwidth]{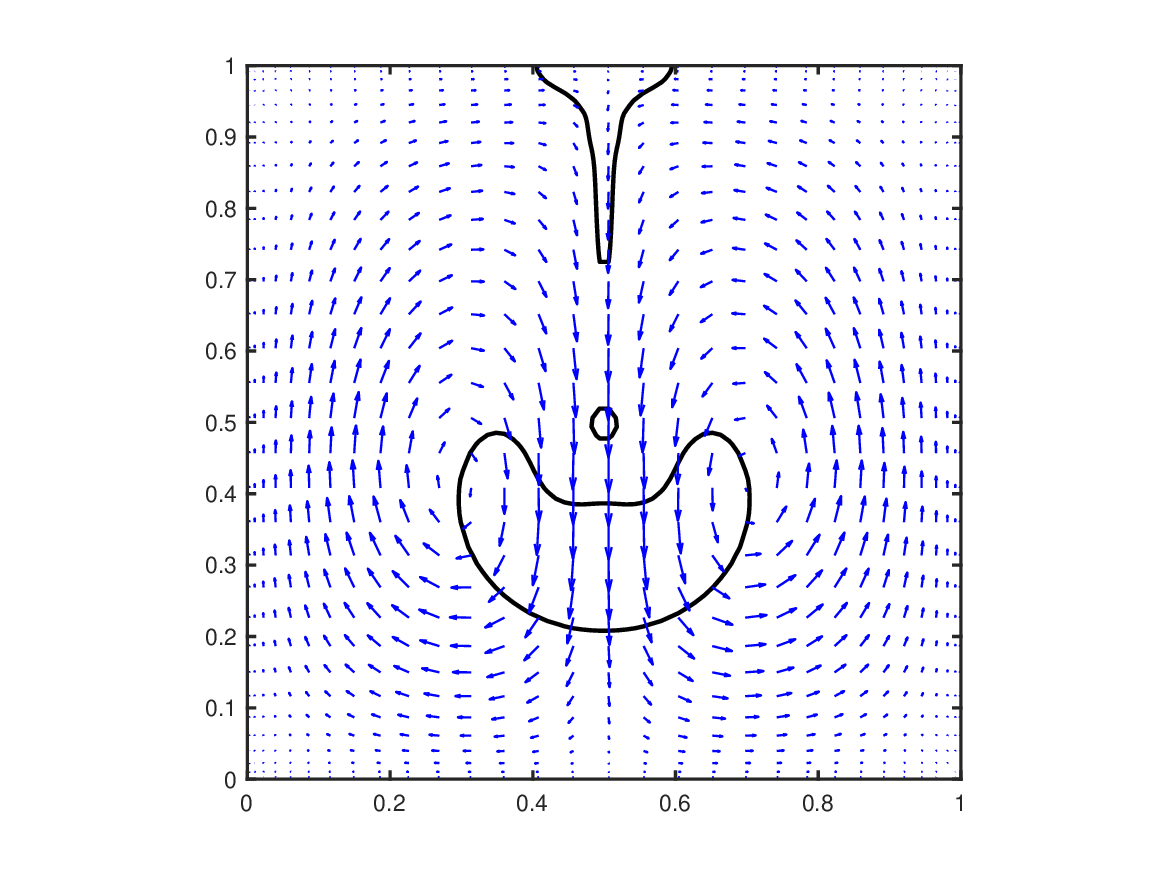}
			%	\vspace{-1cm}
		\end{minipage}
	}
	%	\vspace{-0.5cm}
	{\caption{Numerical velocity at different $T$.}\label{figure7}}
\end{figure}
%==================================================================
\section{Concluding remarks}
% The Cahn-Hilliard-Navier-Stokes phase field model is a strongly nonlinear system where the dissipation of energy depends on intricate cancellations of nonlinear interactions.
In this paper we constructed novel fully decoupled and high-order IMEX schemes for the two-phase incompressible flows based on the EOP-GSAV method for Cahn-Hilliard equation and consistent splitting method for Navier-Stokes equation. The resulting high-order schemes are fully decoupled, linear, unconditional energy stable and only require solving several elliptic equations with constant coefficients at each time step. So they are very efficient and easy to implement. We also carried out a rigorous error analysis for the first-order scheme and derived optimal error estimates for all relevant function in different norms  in two and three-dimensional cases.

% \section*{References}
\bibliographystyle{siamplain}
\bibliography{CHNS_Decoulped_SAV}

\end{document}